\documentclass{article}

\usepackage{microtype}
\usepackage{graphicx}
\usepackage{subfigure}
\usepackage{booktabs} 

\usepackage{amssymb}
\usepackage{epstopdf}
\usepackage{epsfig}

\usepackage{natbib}

\usepackage{amsmath}

\usepackage{amsfonts}

\usepackage{cases}

\usepackage{algorithm}
\usepackage{algorithmic}

\usepackage{amsthm}   

\usepackage{hyperref}

\newtheorem{definition}{Definition}
\newtheorem{proposition}{Proposition}

\newtheorem{remark}{Remark}
\newtheorem{theorem}{Theorem}
\newtheorem{lemma}{Lemma}

\newtheorem*{theorem*}{Theorem}   
\newtheorem*{proposition*}{Proposition}   

\usepackage[accepted]{icml2018}

\icmltitlerunning{An Algorithmic Framework of Variable Metric Over-Relaxed
            Hybrid Proximal Extra-Gradient Method}

\begin{document}

\twocolumn[
\icmltitle{An Algorithmic Framework of Variable Metric Over-Relaxed\\
            Hybrid Proximal Extra-Gradient Method} 




\begin{icmlauthorlist}
\icmlauthor{Li Shen}{to}
\icmlauthor{Peng Sun}{to}
\icmlauthor{Yitong Wang}{to}
\icmlauthor{Wei Liu}{to}
\icmlauthor{Tong Zhang}{to}
\end{icmlauthorlist}

\icmlaffiliation{to}{Tencent AI Lab, Shenzhen, China}

\icmlcorrespondingauthor{Li Shen}{mathshenli@gmail.com}
\icmlcorrespondingauthor{Peng Sun}{pengsun000@gmail.com}
\icmlcorrespondingauthor{Yitong Wang}{yitongwang@tencent.com}
\icmlcorrespondingauthor{Wei Liu}{wl2223@columbia.edu}
\icmlcorrespondingauthor{Tong Zhang}{tongzhang@tongzhang-ml.org}

\icmlkeywords{Maximal monotone operator, Hybrid proximal extra-gradient, Over-relaxed, Iteration complexities, Local linear convergence rate, first-order algorithm, Multi-block convex optimization}

\vskip 0.3in
]



\printAffiliationsAndNotice{}  

\begin{abstract}
We propose a novel algorithmic framework of {\bf V}ariable {\bf M}etric {\bf O}ver-{\bf R}elaxed  {\bf H}ybrid {\bf P}roximal {\bf E}xtra-gradient (VMOR-HPE) method with a global convergence guarantee for the maximal monotone operator inclusion problem.
Its iteration complexities and local linear convergence rate are provided, which theoretically demonstrate that a large over-relaxed step-size contributes to accelerating the proposed VMOR-HPE as a byproduct. Specifically, we find that a large class of primal and primal-dual operator splitting algorithms are all special cases of VMOR-HPE. Hence, the proposed framework offers a new insight into these operator splitting algorithms.
In addition, we apply VMOR-HPE to the Karush-Kuhn-Tucker (KKT) generalized equation of linear equality constrained multi-block composite convex optimization, yielding a new algorithm, namely nonsymmetric {\bf P}roximal {\bf A}lternating {\bf D}irection {\bf M}ethod of {\bf M}ultipliers with a preconditioned {\bf E}xtra-gradient step in which the preconditioned metric is generated by a blockwise {\bf B}arzilai-{\bf B}orwein line search technique (PADMM-EBB).
We also establish iteration complexities of PADMM-EBB in terms of the KKT residual.
Finally, we apply PADMM-EBB to handle the nonnegative dual graph regularized low-rank representation problem.
Promising results on synthetic and real datasets corroborate the efficacy of PADMM-EBB.
\vspace{-13pt}
\end{abstract}

\section{Introduction}\label{Introduction}
\vspace{-1pt}
Maximal monotone operator inclusion, as an extension of the KKT generalized equations for nonsmooth convex optimization and convex-concave saddle-point optimization, encompasses a class of important problems and has extensive applications in statistics, machine learning, signal and image processing, and so on.
More concrete applications can be found in the literature \cite{combettes2011proximal,boyd2011distributed,bauschke2017convex} and the references therein.
Let $\mathbb{X}$ be a finite-dimensional linear vector space. We focus on the operator inclusion problem:
 \begin{equation}\label{inclusion-problem}
 0 \in T(x),\ x\in \mathbb{X},
 \end{equation}
 where $T:\mathbb{X}\rightrightarrows \mathbb{X}$ is a maximal monotone operator.

One of the most efficient algorithms for problem \eqref{inclusion-problem} is {\bf P}roximal {\bf P}oint {\bf A}lgorithm (PPA) in the seminal work \cite{minty1962}, which was further accelerated \cite{eckstein1992douglas} by attaching an over-relaxed parameter $\theta_{k}$,
 \begin{equation*}
   x^{k+1} := x^{k}+(1+\theta_{k})\big(\mathcal{J}_{c_{k}T}(x^{k})-x^{k}\big),\ \theta_{k}\!\in\!(-1,1)
 \end{equation*}
 for a given positive penalty parameter $c_{k}$. Here, $\mathcal{J}_{c_{k}T}(\cdot)=(I+c_{k}T)^{-1}(\cdot)$ is called the resolvent operator \cite{bauschke2017convex} of $T$.
 In addition, its inexact version
 \begin{equation}\label{inexact-over-relaxed-PPA}
   x^{k+1} := x^{k}+(1+\theta_{k})\big(\overline{x}^{k}-x^{k}\big)
 \end{equation}
 was proposed \cite{rockafellar1976monotone}
 by requiring that either absolute error \eqref{absolute-error}
 or relative error criterion \eqref{relative-error} holds,
  \vspace{-1pt}
 \begin{subnumcases}{}
   \left\|\overline{x}^{k}-\mathcal{J}_{c_{k}T}(x^{k})\right\|\le \xi_{k},  \label{absolute-error}\\
   \left\|\overline{x}^{k}-\mathcal{J}_{c_{k}T}(x^{k})\right\|\le \xi_{k}\big\|\overline{x}^{k}-x^{k}\big\|,  \label{relative-error}
\end{subnumcases}
where $\sum_{k=1}^{\infty}\xi_{k} < \infty$. However, it is too flexible to preset the sequence $\{\xi_{k}\}$ which highly influences the level of the computational cost and quality of iteration \eqref{inexact-over-relaxed-PPA}. For more research on PPA and its inexact variants,
 we refer the readers to the literature \cite{guler1991convergence,burke1999variable,corman2014generalized,shen2015linear,tao2017on}.

 Later on, a novel inexact PPA called {\bf H}ybrid {\bf P}roximal {\bf E}xtra-gradient (HPE) algorithm \cite{Solodov1999A1} was proposed.  
 This algorithm first seeks a triple point $(y^{k},v^{k},\epsilon_{k})\!\in\! \mathbb{X}\!\times\!\mathbb{X}\!\times\!\mathbb{R}_{+}$
 satisfying error criterion \eqref{HPE-a1}-\eqref{HPE-a2}:
  \vspace{-2.5pt}
  \begin{subnumcases}{}
  \!(y^{k},v^{k})\in{\rm gph}\,T^{[\epsilon_{k}]}, \label{HPE-a1}\\
  \!\big\|c_{k}v^{k}+(y^{k}-x^{k})\big\|^2\!+\!2c_{k}\epsilon_{k}\! \leq\! \sigma\big\|y^{k}-x^{k}\big\|^2, \qquad \label{HPE-a2}\\
  \!x^{k+1} := x^{k}-c_{k}v^{k}, \label{HPE-a3}
 \end{subnumcases}
 where $T^{[\epsilon]}$ is the enlargement operator \cite{Burachik1997Enlargement,burachik1998varepsilon,Svaiter2000A} of $T$ and $\sigma\!\in\![0,1)$ is a prespecified parameter,
 and then executes an extra-gradient step \eqref{HPE-a3}
 to ensure its global convergence. Whereafter, a new inexact criterion \eqref{metric-relaxed-HPE-a1}-\eqref{metric-relaxed-HPE-a2} is adopted,
 yielding an over-relaxed HPE algorithm \cite{Svaiter2001A,Parente2008A} as below:
 \begin{subnumcases}{}
 (y^{k},v^{k})\in{\rm gph}\,T^{[\epsilon_{k}]}, \label{metric-relaxed-HPE-a1} \\
  \big\| c_{k} \mathcal{M}^{-1}_{k} v^{k}+(y^{k}\!-\!x^{k})\big\|_{\mathcal{M}_{k}}^2+2 c_{k}\epsilon_{k} \label{metric-relaxed-HPE-a2}\\
  \qquad \leq \sigma\big(\big\|y^{k}-x^{k}\big\|_{\mathcal{M}_{k}}^2+\big\| c_{k} \mathcal{M}^{-1}_{k}\big\|_{\mathcal{M}_{k}}^2\big), \qquad \nonumber\\
  x^{k+1} := x^{k}- (1+\tau_{k})a_{k}\mathcal{M}_{k}v^{k}, \label{metric-relaxed-HPE-a3}
 \end{subnumcases}
 where $\tau_{k}\in (-1,1)$ is the over-relaxed step-size, $a_{k} = \big[\langle v^{k},x^{k}-y^{k}\rangle -\epsilon_{k}\big]{\big /}\big\|\mathcal{M}^{-1}_{k} v^{k}\big\|_{\mathcal{M}_{k}}^2$,
 and $\mathcal{M}_k$ is a self-adjoint positive definite linear operator. An obvious defect of the above algorithm
 is that extra-gradient step-size $a_{k}$ has to be adaptively determined to ensure its global convergence,
 which requires extra computation and may be time-consuming. In addition, Korpelevich's extra-gradient algorithm \cite{Korpelevi1976An},
 forward-backward algorithm \cite{Passty1979Ergodic}, and forward-backward-forward algorithm \cite{Tseng2000A}
 are all shown to be special cases of the HPE algorithm in \citep{Solodov1999A1,Svaiter2014A}.

 In this paper, we propose a new algorithmic framework of {\bf V}ariable {\bf M}etric {\bf O}ver-{\bf R}elaxed {\bf H}ybrid {\bf P}roximal {\bf E}xtra-gradient (VMOR-HPE) method
 with a global convergence guarantee for solving problem \eqref{inclusion-problem}. This framework, in contrast to the existing HPE algorithms,
 generates the iteration sequences in terms of a novel relative error criterion and introduces an
 over-relaxed step-size in the extra-gradient step to improve its performance.
 In particular, the extra-gradient step-size and over-relaxed step-size here can both be set as a fixed constant in advance,
 instead of those obtained from a projection problem, which saves extra computation.
 Its global convergence, $\mathcal{O}(\frac{1}{\sqrt{k}})$ pointwise and $\mathcal{O}(\frac{1}{k})$ weighted iteration complexities,
 and the local linear convergence rate under some mild metric subregularity condition \cite{Dontchev2009Implicit} are also built.
 Interestingly, the coefficients of iteration complexities and linear convergence rate are inversely proportional to the over-relaxed step-size,
 which theoretically demonstrates that a large over-relaxed step-size contributes to accelerating the proposed VMOR-HPE as a byproduct.
 In addition, we rigorously show that a class of primal-dual algorithms, including
 {\bf A}symmetric {\bf F}orward {\bf B}ackward {\bf A}djoint {\bf S}plitting {\bf P}rimal-{\bf D}ual (AFBAS-PD) algorithm \cite{latafat2017asymmetric},
 Condat-Vu {\bf P}rimal-{\bf D}ual {\bf S}plitting (Condat-Vu  PDS) algorithm \cite{vu2013splitting,condat2013primal},
 {\bf P}rimal-{\bf D}ual {\bf F}ixed {\bf P}oint (PDFP) algorithm \cite{chen2016primal},
 {\bf P}rimal-{\bf D}ual three {\bf O}perator {\bf S}plitting (PD3OS) algorithm \cite{yan2016primal},
 Combettes {\bf P}rimal-{\bf D}ual {\bf S}plitting (Combettes PDS) algorithm \cite{combettes2012primal},
 {\bf M}onotone+{\bf S}kew {\bf S}plitting (MSS) algorithm \cite{briceno2011monotone},
 {\bf P}roximal {\bf A}lternating {\bf P}redictor {\bf C}orrector (PAPC) algorithm \cite{drori2015simple},
 and {\bf P}rimal-{\bf D}ual {\bf H}ybrid {\bf G}radient (PDHG) algorithm \cite{chambolle2011first},
 all fall into the VMOR-HPE framework with specific variable metric operators $\mathcal{M}_{k}$ and $T$.
 Besides, {\bf P}roximal-{\bf P}roximal-{\bf G}radient (PPG) algorithm \cite{ryu2017proximal},
 {\bf F}orward-{\bf B}ackward-{\bf H}alf {\bf F}orward (FBHF) algorithm as well as its non self-adjoint metric extensions \cite{briceno2017forward},
 Davis-Yin three {\bf O}perator {\bf S}plitting (Davis-Yin 3OS) algorithm \cite{davis2015three},
 {\bf F}orward {\bf D}ouglas-{\bf R}achford {\bf S}plitting (FDRS) algorithm \cite{Brice2012Forward},
 {\bf G}eneralized {\bf F}orward {\bf B}ackward {\bf S}plitting (GFBS) algorithm \cite{raguet2013generalized},
 and {\bf F}orward {\bf D}ouglas-{\bf R}achford {\bf F}orward {\bf S}plitting (FDRFS) algorithm \cite{Brice2015Forward}
 also fall into the VMOR-HPE framework. Thus, VMOR-HPE largely expands the HPE algorithmic framework to cover a large class of primal and primal-dual
 algorithms and their non self-adjoint metric extensions compared with \cite{Solodov1999A1,shen2017over}. 
 As a consequence, the VMOR-HPE algorithmic framework offers a new insight into aforementioned primal and primal-dual algorithms and serves as a powerful analysis technique for establishing their convergences, iteration complexities, and local linear convergence rates.

 In addition, we apply VMOR-HPE to the KKT generalized equation of linear equality constrained multi-block composite nonsmooth convex optimization as follows:
\begin{align}\label{linearly-constriant}
 & \min_{x_i\in\mathbb{X}_i}\ f(x_1,\ldots,x_{p})+g_1(x_1)+\cdots+g_p(x_p)\\
 &\ \ \ {\rm s.t.}\ \ \ \mathcal{A}_1^*x_1+\mathcal{A}_2^*x_2+\cdots+\mathcal{A}_p^*x_p=b,\nonumber
 \end{align}
 where $\mathcal{A}_{i}^*\!\!:\!\mathbb{Y}\!\!\to\!\!\mathbb{X}_{i}$ is the adjoint linear operator of $\mathcal{A}_{i}$,
 $\mathbb{Y}$ and $\mathbb{X}_{i}$ are given finite-dimensional vector spaces, $g_{i}\!:\mathbb{X}_i\!\to\!(-\infty,+\infty]$ is a proper closed convex function,
 and $f\!:\mathbb{X}_{1}\!\times\!\cdots\!\times\mathbb{X}_{p}\!\to\!\mathbb{R}$ is a gradient Lipschitz continuous convex function.
 Specifically, the proposed VMOR-HPE for solving problem \eqref{linearly-constriant} firstly generates points satisfying the relative inexact criterion in the VMOR-HPE framework
 by a newly developed nonsymmetric {\bf P}roximal {\bf A}lternating {\bf D}irection {\bf M}ethod of {\bf M}ultipliers,
 and then performs an over-relaxed metric {\bf E}xtra-gradient correction step to ensure its global convergence.
 Notably, metric $\mathcal{M}_{k}$ in the extra-gradient step is generated by using a blockwise {\bf B}arzilai-{\bf B}orwein line search technique \cite{barzilai1988two}
 to exploit the curvature information of the KKT generalized equation of \eqref{linearly-constriant}. We thus name the resulting new algorithm as PADMM-EBB.
 Moreover, we establish the $\mathcal{O}(\frac{1}{\sqrt{k}})$ pointwise
 and $\mathcal{O}(\frac{1}{k})$ weighted iteration complexities and the local linear convergence rate for PADMM-EBB on the KKT residual of \eqref{linearly-constriant} by employing
 the VMOR-HPE framework.
 Besides, it is worth emphasizing that the derived iteration complexities do not need any assumption on the boundedness of the feasible set of \eqref{linearly-constriant}.
 At last, we conduct experiments on the nonnegative dual graph regularized low-rank representation problem to verify the efficacy of PADMM-EBB,
 which shows great superiority over {\bf P}roximal {\bf L}inearized ADMM with {\bf P}arallel {\bf S}plitting and {\bf A}daptive {\bf P}enalty (PLADMM-PSAP) \cite{liu2013linearized, lin2015linearized},
 {\bf P}roximal {\bf G}auss-{\bf S}eidel ADMM (PGSADMM) with nondecreasing penalty, and {\bf M}ixed {\bf G}auss-{\bf S}eidel and {\bf J}acobi ADMM (M-GSJADMM) with nondecreasing penalty \cite{lu2017unified}
 on both synthetic and real datasets.

 The major contributions of this paper are fourfold.
 {\bf (i)} We propose a new algorithmic framework of VMOR-HPE for problem \eqref{inclusion-problem} and also establish its global convergence, iteration complexities, and local linear convergence rate.
 {\bf(ii)} The proposed VMOR-HPE gives a new insight into a large class of primal and primal-dual algorithms and provides a unified analysis framework for their convergence properties.
 {\bf(iii)} Applying VMOR-HPE to problem \eqref{linearly-constriant} yields a new convergent primal-dual algorithm
 whose iteration complexities on the KKT residual are also provided without requiring the boundedness of the feasible set of \eqref{linearly-constriant}.
 {\bf(iv)} Numerical experiments on synthetic and real datasets are conducted to demonstrate the superiority of the proposed algorithm.
\vspace{-4pt}
\section{Preliminaries}\label{preliminary}

 Given $\beta>0$, a single-valued mapping $C\!:\mathbb{X}\to\mathbb{X}$ satisfing $\big\langle x-x',C(x)-C(x')\big\rangle \ge \beta\big\|C(x)-C(x')\big\|^2$
 for all $x,x'\in \mathbb{X}$ is called a $\beta$-cocoercive operator.
 A set-valued mapping $T\!:\mathbb{X}\rightrightarrows\mathbb{X}$ satisfying $\langle x-x',v-v'\rangle\ge\alpha\|x-x'\|^2$ with $\alpha\ge0$ for all $v\in T(x)$ and $v' \in T(x')$
 is called $\alpha$-strongly monotone operator if $\alpha>0$, and is called a monotone operator if $\alpha=0$.
 Moreover, $T$ is called a maximal monotone operator if there does not exit any monotone operator $T'$ satisfying ${\rm gph}\,T\subseteq {\rm gph}\,T'$, where ${\rm gph}\,T\!:=\{(x,v)\in \mathbb{X}\times\mathbb{X}\mid v\in T(x), x\in \mathbb{X}\}$.
 In addition, given $\epsilon\ge0$ and a maximal monotone operator ${T}$,
 the $\epsilon$-enlargement ${T}^{[\epsilon]}\!:\mathbb{X}\rightrightarrows\mathbb{X}$ of $T$ \cite{Burachik1997Enlargement,burachik1998varepsilon,Svaiter2000A}
 is defined as
 \begin{equation*}
 {T}^{[\epsilon]}(x):=\big\{v\in\mathbb{X} \mid \langle w-v,z-x \rangle \geq -\epsilon, \forall w\in {T}(z)\big\}.
 \end{equation*}
 Below, we recall the definition of metric subregularity \cite{Dontchev2009Implicit} of set-valued mapping $T$.
 \begin{definition}\label{metric-subregular}
  A set-valued mapping ${T}\!:\mathbb{X}\rightrightarrows\mathbb{X}$ is metric subregular at $(\overline{x},\overline{y})\in{\rm gph}{T}$ with modulus $\kappa>0$,
  if there exists a neighborhood $U$ of $\overline{x}$ such that for all $x\in U$,
  \[
  {\rm dist}\big(x,{T}^{-1}(\overline{y})\big)\le \kappa{\rm dist}\big(\overline{y},{T}(x)\big).
  \]
\end{definition}
\vspace{-6pt}
 Given a self-adjoint positive definite linear operator $\mathcal{M}$, $\|\cdot\|_{\mathcal{M}}$ denotes the generalized norm induced by $\mathcal{M}$, which is defined as
 $\|\cdot\|_{\mathcal{M}}=\sqrt{\langle\cdot,\mathcal{M}\cdot \rangle}$. The generalized distance between a point $z$ and a set $\Omega$ induced by $\mathcal{M}$ is defined as
 ${\rm dist}_{\mathcal{M}}(z,\Omega)\!:=\!\inf_{ x\in \Omega} \|x\!-\!z\|_{\mathcal{M}}$. Let $\mathcal{M}\!=\!\mathcal{I}$. ${\rm dist}_{\mathcal{M}}(z,\Omega)$ reduces to the standard distance function as ${\rm dist}(z,\Omega)\!:=\!\inf_{ x\in \Omega} \|x\!-\!z\|$.
 In addition, given a proper closed convex function $g:\mathbb{X}\!\to\!(\infty,+\infty]$ and a non self-adjoint linear operator $\mathcal{R}$, ${\rm Prox}_{\mathcal{R}^{-1}g}(\cdot)$
 denoting the generalized proximal mapping of $g$ induced by $\mathcal{R}$ is the unique root of inclusion:
 \[
 0 \in \partial{g}(x) + \mathcal{R}(x-\cdot),\ x\in \mathbb{X}.
 \]
 Particularly, if $g(x)=\sum_{i=1}^{n}g_{i}(x_{i})$ is decomposable, ${\rm Prox}_{\mathcal{R}^{-1}g}(\cdot)$ can be calculated in a Gauss-Seidel manner
 by merely setting $\mathcal{R}$ as a block lower-triangular linear operator.

\section{VMOR-HPE Framework}\label{VM-OR-HPE Algorithm}

 In this section, we propose the algorithmic framework of VMOR-HPE (described in Algorithm \ref{Alg:VMOR-HPE}), and establish its global convergence rate, iteration complexities, and local linear convergence rate.
 Let $\mathcal{M}_{k}\!\!=\!\mathcal{I}$ in VMOR-HPE. We recover an enhanced version of an over-relaxed HPE algorithm \cite{shen2017over}
 by allowing a larger over-relaxed step-size $\theta_{k}$.

\begin{algorithm}[H]
   \caption{\, VMOR-HPE Framework}
   \label{Alg:VMOR-HPE}
\begin{algorithmic}
   \STATE {\bf Parameters:} Given $\underline{\omega},\,\overline{\omega}> 0,\,\underline{\theta}>-1,\,\sigma\in[0,1)$ and $\xi_{k} \ge 0$ satisfying $\sum_{k=1}^{\infty}\xi_{k}\!<\infty$.
                           Choose a self-adjoint operator $\mathcal{M}_{0}$ satisfying $\underline{\omega}\mathcal{I}\preceq \mathcal{M}_{0} \preceq \overline{\omega}\mathcal{I}$
                           and $x^{0}\in\mathbb{X}$.
   \FOR{$k=1,2,\cdots,$}
   \STATE  Choose $c_k\ge\underline{c}>0,\, \theta_k\in[\underline{\theta},\infty)$. Find $(\epsilon_k,y^{k},v^{k})\in\mathbb{R}_{+}\times\mathbb{X}\times\mathbb{X}$ satisfying
           the relative error criterion that
    \vspace{-5pt}
                \begin{subnumcases}{}
                 \!(y^{k},v^{k})\in{\rm gph}\,T^{[\epsilon_{k}]},  \label{VMOR-HPE-a} \\
                 \!\theta_k\big\| c_k\mathcal{M}_{k}^{-1}v^{k}\big\|_{\mathcal{M}_{k}}^2+\big\| c_k\mathcal{M}_{k}^{-1}v+(y^{k}-x^{k})\big\|_{\mathcal{M}_{k}}^2 \nonumber\\
                               \qquad\qquad\qquad+2 c_k\epsilon_{k} \leq \sigma\big\|y^{k}- x^{k}\big\|_{\mathcal{M}_{k}}^2. \quad \label{VMOR-HPE-b}
                \end{subnumcases}
   \vspace{-9pt}
   \STATE Let $x^{k+1} := x^k -(1+\theta_k)c_k\mathcal{M}_{k}^{-1} v^{k}$.
   \STATE Update $\mathcal{M}_{k+1}$ with $\underline{\omega}\mathcal{I}\preceq \mathcal{M}_{k+1} \preceq (1+\xi_{k})\mathcal{M}_{k}$.
   \ENDFOR
 \end{algorithmic}
 \end{algorithm}
  \vspace{-8pt}
\begin{remark}
{\bf (i)}  $\theta_k\!\in\![\underline{\theta},\infty)$ breaks the ceiling of over-relaxed step-sizes in the literature \cite{eckstein1992douglas,chambolle2016ergodic,bauschke2017convex,shen2017over,tao2017on} ,in which $\theta_{k}\!\in\!(-1,1)$. Besides, $\mathcal{M}_{k}$ can exploit the curvature information of $T$.
\vspace{1pt}\\
{\bf (ii)} Let $\theta_{k} = -\sigma$ in the VMOR-HPE framework. Criterion \eqref{VMOR-HPE-a}-\eqref{VMOR-HPE-b} coincides with \eqref{metric-relaxed-HPE-a1}-\eqref{metric-relaxed-HPE-a2} in \cite{Parente2008A}, which makes the step-size $(1+\theta_{k})$ be $(1-\sigma)$ that is too small to update $x^{k+1}$ if $\sigma$ is close to $1$. That is the reason why $a_{k}$ in \eqref{metric-relaxed-HPE-a3} has to be adaptively computed with extra computation instead of being a constant.
\end{remark}

\subsection{Convergence Analysis}

 In this subsection, we build the global convergence for the algorithmic framework of VMOR-HPE, as well as its local linear convergence rate
 under a metric subregularity condition of $T$. In addition, its $\mathcal{O}(\frac{1}{\sqrt{k}})$ pointwise and
 $\mathcal{O}(\frac{1}{k})$ weighted iteration complexities depending solely on $(T^{-1}(0),x^{0})$ are provided.
 Denote $\Xi\!:=\!\prod\!_{i=0}^{\infty}(1\!\!+\xi_{i})\!\!<\!\exp\big(\sum\!_{i=0}^{\infty}\xi_{i}\big)\!<\!\infty$.
 \begin{theorem}\label{VMOR-HPE-convergence}
 Let $\big\{(x^{k},y^{k})\big\}$ be the sequence generated by the VMOR-HPE framework. Then,
 $\{x^{k}\}$ and $\{y^{k}\}$ both converge to a point $x^{\infty}$ belonging to $T^{-1}(0)$.
 \end{theorem}
 \begin{theorem}\label{linear-rate}
 Let $\{(x^{k},\,y^{k})\}$ be the sequence generated by the VMOR-HPE framework. Assume that
 the metric subregularity of $T$ at $(x^{\infty},0)\in {\rm gph}\,T$ holds with $\kappa >0$. Then, there exits
 $\overline{k} >0$ such that for all $k\ge \overline{k}$,
 \[
  {\rm dist}^{2}_{\mathcal{M}_{k\!+\!1}}\big(x^{k+1},T^{-1}(0)\big)
  \!\leq\! \Big(1\!-\!\frac{\varrho_{k}}{2}\Big){\rm dist}^{2}_{\mathcal{M}_{k}}\big(x^{k},T^{-1}(0)\big),
 \]
 where $\varrho_{k}= \frac{(1-\sigma)(1+\theta_{k})}
            {\Big(1+\frac{\kappa}{\underline{c}}\sqrt{\frac{\Xi\overline{\omega}}{\underline{\omega}}}\Big)^{2}\Big(1+\sqrt{\sigma+\frac{4\max\{-\theta_{k},0\}}{(1+\theta_{k})^2}}\Big)^{2}}\in (0,1)$.
\end{theorem}
 Polyhedra operators \cite{robinson1981some} and strongly monotone operators all satisfy metric subregularity.
 For other sufficient conditions that guarantee metric subregulaity of $T$, we refer the readers to the monographs \cite{Dontchev2009Implicit,rockafellar2009variational,ying2016large}.

 Point $x\!\in\!\mathbb{X}$ is called $\varepsilon$-solution \cite{Monteiro2010On} of problem \eqref{inclusion-problem}
 if there exists $(v,\epsilon)\! \in \mathbb{X}\!\times\mathbb{R}_{+}$ satisfying $v \in T^{[\epsilon]}(x)$ and $\max(\|v\|, \epsilon)\le \varepsilon$.
 Below, we globally characterize the rate of $\max(\|v\|,\epsilon)$ decreasing to zero.
\begin{theorem}\label{iteration-complexity-VMOR-HPE}
 Let $\{(x^{k},y^{k},v^{k})\}$ and $\{\epsilon_{k}\}$ be the sequences generated by the VMOR-HPE framework. \\
 {\bf (i)} There exists an integer $k_{0}\!\in\! \{1,2,\ldots,k\}$ such that $v^{k_{0}} \in T^{[\epsilon_{k_0}]}(y^{k_{0}})$
 with $v^{k_{0}}$ and $\epsilon_{k_0}\ge 0$ respectively satisfying
 \begin{gather*}
 \|v^{k_{0}}\|\le \sqrt{\frac{4(1+\sum_{i=1}^{k}\xi_{i})\Xi^2\overline{\omega}}{k(1-\sigma)(1+\underline{\theta})^3\underline{c}^2}}\|x^{0}-x^*\|_{\mathcal{M}_{0}},\\
{\rm and\quad } \epsilon_{k_0} \le \frac{(1+\sum_{i=1}^{k}\xi_{i})\Xi}{k(1-\sigma)(1+\underline{\theta})^2\underline{c}}\|x^{0}-x^*\|_{\mathcal{M}_{0}}^2.
 \end{gather*}
   \vspace{-5pt}
 {\bf (ii)} Let $\{\alpha_{k}\}$ be the nonnegative weight sequence satisfying $\sum_{i=1}^{k}\alpha_{i}>0$.
 Denote $\tau_{i}\!=\!(1+\theta_{i})c_{i}$ and $\overline{y}^{k}\!=\!\frac{{\sum_{i=1}^{k}}\tau_{i}\alpha_{i}y^{i}}{{\sum_{i=1}^{k}}\tau_{i}\alpha_{i}}$,
  \vspace{-6pt}
 \begin{align*}
 \overline{v}^{k}\!=\!\frac{{\sum_{i=1}^{k}}\tau_{i}\alpha_{i}v^{i}}{{\sum_{i=1}^{k}}\tau_{i}\alpha_{i}},
 \overline{\epsilon}_{k}\!=\!\frac{{\sum_{i=1}^{k}}\tau_{i}\alpha_{i}\big(\epsilon_{i}
                    \!+\!\langle y^{i}\!-\!\overline{y}^{k},v^{i}\!-\!\overline{v}^{k}\rangle\big)}{{\sum_{i=1}^{k}}\tau_{i}\alpha_{i}}.
 \end{align*}
 Then, it holds that $\overline{v}^{k} \in T^{[\overline{\epsilon}_{k}]}(\overline{y}^{k})$ with $\overline{\epsilon}_{k}\ge 0$.
 Moreover, if $\mathcal{M}_{k}\le (1+\xi_{k})\mathcal{M}_{k+1}$, it holds that
  \vspace{-5pt}
 \begin{gather}
 \!\!\|\overline{v}^{k}\|
 \!\le\! \frac{\max\limits_{ 1\le i\le k}\{\alpha_{i+1}\}\sum\limits_{i=1}^{k}\xi_{i}\!\!+\!\sum\limits_{i=1}^{k}\big|\alpha_{i}\!-\!\alpha_{i+1}\big|\!+\!\alpha_{k+1}\!+\!\alpha_{1}}
                 {\underline{c}(1+\underline{\theta})\sum_{i=1}^{k}\alpha_{i}}M, \nonumber\\
 \!\!\overline{\epsilon}_{k}\!=\!\frac{(10\!+\!\underline{\theta})\!\!\max\limits_{1\le i\le k}\{\alpha_{i}\}\big(1\!\!+\!\!\sum\limits_{i=1}^{k}\xi_{i}\big)
                               \!+\!(2\!+\!\underline{\theta}){\sum\limits_{i=1}^{k}}\big|\alpha_{i+1}\!\!-\!\alpha_{i}\big|}
         {\underline{c}(1+\underline{\theta})^2\sum_{i=1}^{k}\alpha_i}B, \nonumber 
 \end{gather}
where $M$ and $B$ are two constants which are respectively defined as $M = \Xi\overline{\omega}\big[\|x^*\|+\sqrt{{\Xi}/{\underline{\omega}}}\|x^{0}-x^*\|_{\mathcal{M}_{0}}\big]$ and
\[
B =\max\left\{ \begin{array}{cc}
                          M,\, \Xi\big\|x^*\big\|^2\!+\!\frac{\Xi^2}{\underline{\omega}}\big\|x^{0}\!-\!x^*\big\|_{\mathcal{M}_{0}}^2,\\
                          \frac{\Xi^2}{(1\!-\!\sigma)\underline{\omega}}\big\|x^{0}\!-\!x^*\big\|_{\mathcal{M}_{0}}^2, \frac{\Xi}{(1\!-\!\sigma)}\big\|x^{0}\!-\!x^*\big\|_{\mathcal{M}_{0}}^2\\
                        \end{array}
        \right\}.
\]
\end{theorem}
\begin{remark}
{\bf (i)}\ The iteration complexities in Theorem \ref{iteration-complexity-VMOR-HPE} merely depend on the solution set $T^{-1}(0)$ and initial point $x^{0}$.
The upper bounds of $(v^{k_{0}},\epsilon_{k_{0}})$ and $(\overline{v}_{k},\overline{\epsilon}_{0})$ are both inversely proportional to $\theta_{k}$,
 which, in combination with Theorem \ref{linear-rate}, theoretically demonstrates that a large over-relaxed step-size contributes to accelerating VMOR-HPE.\\
{\bf (ii)}\ Set $\alpha_k\!= 1$ or $k$. It holds that $\|\overline{v}^{k}\|\!\le \mathcal{O}(\frac{1}{k})$ and $\overline{\epsilon}_{k}\!\le \mathcal{O}(\frac{1}{k})$.
 However, setting $\alpha_k = k$ may lead to better performance than setting $\alpha_k = 1$, since
  $\alpha_{k} = k$ gives more weights on the latest generated points $y^{k}$ and $v^{k}$.
\end{remark}

\subsection{Connection to Existing Algorithms}

 First, we consider $\mathcal{M}_{k}=\mathcal{I}$. Under this situation, the proposed VMOR-HPE reduces to the over-relaxed HPE algorithm \cite{shen2017over} which covers
a number of primal first-order algorithms as special cases, such as FDRS algorithm, GFBS algorithm, FDRFS algorithm, \textit{etc}.
 Hence, they are also covered by the algorithmic framework of VMOR-HPE. Below, we show a large collection of other primal and primal-dual algorithms which
 fall into VMOR-HPE.

 \subsubsection{primal algorithms}

 {\bf FBHF Algorithm}\ tackles problem \eqref{inclusion-problem} as
 \begin{equation*}
 0 \in T(x)=(A+B_1+B_2)(x),\, x\in \Omega,
 \end{equation*}
 where $A$ is a maximal monotone operator, $B_{1}\!:\mathbb{X}\to \mathbb{X}$ is a $\beta$-cocoercive operator, $B_{2}\!:\mathbb{X}\to \mathbb{X}$ is a monotone and $L$-Lipschitz continuous operator,
 and $\Omega$ is a subset of $\mathbb{X}$. The FBHF algorithm has the iterations:
  \vspace{-2pt}
 \begin{subnumcases}{}
  y^{k} := \mathcal{J}_{\gamma_{k}A}\big(x^{k}-\gamma_{k}(B_1+B_2)x^{k}\big),           \nonumber \\
  x^{k+1} := P_{\Omega}\big(y^{k}+\gamma_{k} B_{2}(x^{k})-\gamma_{k} B_{2}(y^{k})\big). \nonumber
 \end{subnumcases}
  \vspace{-2pt}
 In the following, we focus on $\Omega\!=\!\mathbb{X}$ and replace $x^{k+1}$ by
 \begin{equation*}
 \!x^{k+1}\!:=\!x^{k}\!+\!(1+\theta_{k})\big(y^{k}-x^{k}\!+\!\gamma_{k} B_{2}(x^{k})-\gamma_{k}B_{2}(y^{k})\big)
 \end{equation*}
 to obtain an over-relaxed FBHF algorithm.
 The proposition below rigorously reformulates the over-relaxed FBHF algorithm as a specific case of the VMOR-HPE framework.
 \vspace{-6pt}
 \begin{proposition}\label{FBF-VMOR-HPE}
 Let $\{(x^{k},y^{k})\}$ be the sequence generated by the over-relaxed FBHF algorithm.
 Denote $\epsilon_{k}\!=\!\|x^{k}\!-\!y^{k}\|^2/(4\beta)$ and $v^{k}\!=\!\gamma_{k}^{-1}(x^{k}\!-\!y^{k})\!-\!B_{2}(x^{k})\!+\!B_{2}(y^{k})$. Then,
 \vspace{-2pt}
 \begin{subnumcases}{}
 \!(y^{k},v^{k})\in{\rm gph}\,T^{[\epsilon_{k}]}={\rm gph}\,(A+B_1+B_2)^{[\epsilon_{k}]},\nonumber\\
 \!\theta_k\big\| \gamma_kv^{k}\big\|^2\!+\!\big\| \gamma_kv^{k}\!+\!(y^{k}\!-\!x^{k})\big\|^2\!+\!2 \gamma_k\epsilon \!\leq\! \sigma\big\|y^{k}\!-\! x^{k}\big\|^2,\nonumber\\%
 \!x^{k+1}=x^{k}-(1+\theta_{k})\gamma_{k}v^{k},\nonumber 
 \end{subnumcases}
  \vspace{-1pt}
 where $(\gamma_{k},\,\theta_{k})$ satisfies $\theta_{k} \le \frac{\sigma-(\gamma_{k}L)^2+\gamma_{k}/(2\beta))}{1+(\gamma_{k}L)^2}$.
 \end{proposition}
 \vspace{-4pt}
 \begin{remark}
 {\bf (i)}\ If $\theta_{k}\!=\!0$, $\gamma_{k}$ reduces to 
 $\gamma_{k}^2L^2\!+\!\gamma_{k}/(2\beta)\!\le\!\sigma\!<\!1\!\Leftrightarrow\! 0\!<\!\gamma_{k}\!<\!4\beta/(1\!+\!\sqrt{1\!+\!16\beta^2L^2})$
 which coincides with the properties of $\gamma_{k}$ in \cite{briceno2017forward}.\\
 {\bf (ii)}\  By \cite{Solodov1999A1}, a slightly modified VMOR-HPE by attaching an extra projection step $P_{\Omega}$ on $x^{k+1}$ can cover the original FBHF algorithm.\\
 {\bf (iii)}\ Let $B_{1}=0$ or $B_{2}=0$.  The over-relaxed FBHF algorithm reduces to over-relaxed Tseng's forward-backward-forward splitting algorithm \cite{Tseng2000A}
 or over-relaxed forward-backward splitting algorithm \cite{Passty1979Ergodic}. Thus, they are special cases of VMOR-HPE by Proposition \ref{FBF-VMOR-HPE}.
 \end{remark}

 {\bf nMFBHF Algorithm }\ The {\bf n}on self-adjoint {\bf M}etric variant of FBHF (nMFBHF) algorithm takes the iterations:
 \begin{subnumcases}\!{}
 \!\!\!y^{k}\!:=\! \mathcal{J}_{P^{-1} A}\big(x^{k}-P^{-1}(B_1+B_2)(x^{k})\big), \nonumber\\
 \!\!\!x^{k+1}\!:=\! P^{U}_{\Omega}\big(y^{k}\!+\!U^{-1}[B_{2}(x^{k})\!-\! B_{2}(y^{k})\!-\!S(x^{k}\!-\!y^{k})]\big), \qquad \nonumber 
 \end{subnumcases}
 where $P$ is a bounded linear operator, $U=(P\!+\!P^*)/2$, $S=(P\!-\!P^*)/2$, and $P^{U}_{\Omega}$ is the projection operator of $\Omega$ under the weighted inner product $\langle\cdot,U\cdot\rangle$.
 Similarly, let $\Omega\!=\!\mathbb{X}$. We obtain the over-relaxed nMFBHF algorithm by replacing the updating step $x^{k+1}$ as the following form
 \begin{align*}
 x^{k+1}&:=x^{k}+(1+\theta_{k})\big(y^{k}-x^{k}\!+\!U^{-1}[B_{2}(x^{k})\!-\! B_{2}(y^{k})]\nonumber\\
                  & \qquad\qquad\qquad\qquad\qquad -U^{-1}[S(x^{k}-y^{k})]\big).
 \end{align*}
 Below, we show that the over-relaxed nMFBHF algorithm also falls into the VMOR-HPE framework.
 Notice that $B_{2}-S$ preserves the monotonicity by the skew symmetry of $S$, and $K$ is denoted as its Lipschitz constant.
 \begin{proposition}\label{nMFBHF}
 Let $\{(x^{k},y^{k})\}$ be the sequence generated by the over-relaxed nMFBHF algorithm.
 Denote $\epsilon_{k}\!= \!\|x^{k}\!-\!y^{k}\|^2/(4\beta)$ and $v^{k}\!=\!P(x^{k}\!-\!y^{k})\!+\!B_{2}(y^{k})\!-\!B_{2}(x^{k})$. The step-size $\theta_{k}$ satisfies
 $\theta_{k}\!+\!\frac{K^2(1\!+\!\theta_{k})}{\lambda^{2}_{\min}(U)}\!+\!\frac{1}{2\beta\lambda_{\min}(U)}\!\le\!\sigma$. Then,
 \begin{subnumcases}\!{}
 \!\!\!(y^{k},v^{k})\in{\rm gph}\,T^{[\epsilon_{k}]}={\rm gph}\,(A+B_1+B_2)^{[\epsilon_{k}]},\nonumber\\
 \!\!\!\theta_k\big\|U^{-1}\!v^{k}\big\|_{\!U}^2\!+\!\big\|U^{-1}\!v\!+\!\!(y^{k}\!-\!x^{k})\big\|_{\!U}^2\!+\!\!2\epsilon\! \leq\! \sigma\big\|y^{k}\!- \!x^{k}\big\|_{\!U}^2, \nonumber\\
 \!\!\! x^{k+1}=x^{k} - (1+\theta_{k})U^{-1}v^{k}. \nonumber
 \end{subnumcases}
 \end{proposition}
 Let $\theta_{k}=0$, and then $\theta_{k}\!+\frac{K^2(1\!+\!\theta_{k})}{\lambda^{2}_{\min}(U)}\!+\!\frac{1}{2\beta\lambda_{\min}(U)}\le\sigma<1$ reduces to
 $\frac{K^2}{\lambda^{2}_{\min}(U)}+\frac{1}{2\beta\lambda_{\min}(U)}< 1$, which coincides with the required condition in \cite{briceno2017forward}.

 \noindent
 {\bf PPG Algorithm}\ Consider the following minimization of a sum of many smooth and nonsmooth convex functions
 \begin{equation}\label{PPG-problem}
 \min_{x\in\mathbb{X}}  r(x)+\frac{1}{n}\sum_{i=1}^{n}f_{i}(x)+\frac{1}{n}\sum_{i=1}^{n}g_{i}(x).
 \end{equation}
 Let $\alpha\in (0,\frac{3}{2L})$. The PPG algorithm takes iterations as
 \begin{subnumcases}{}
 \!\!\!x^{k+\frac{1}{2}}:= {\rm Prox}_{\alpha r}\big(\frac{1}{n}\sum_{i=1}^{n}z_{i}^{k}\big),\nonumber\\
 \!\!\!x_{i}^{k+1}\!\! := \!{\rm Prox}_{\alpha g_{i}}\big(2x^{k\!+\!\frac{1}{2}}\!-\!z_{i}^{k}\!-\!\alpha\nabla\!f_{i}(x^{k\!+\!\frac{1}{2}})\big),\ i\!=\!1,\ldots,n, \qquad\nonumber\\
 \!\!\!z_{i}^{k+1} := z_{i}^{k}+x_{i}^{k+1}-x^{k+\frac{1}{2}},\ i=1,2,\ldots,n,  \nonumber
 \end{subnumcases}
  where $g_{i},r\!:\!\mathbb{X}\!\to\!\!(-\infty,+\infty]$ are proper closed convex functions, and $f_{i}\!:\!\mathbb{X}\!\to\!(-\infty,+\infty)$
 is a differentiable convex function satisfying $\|\nabla f_{i}(x)\!-\!\nabla f_{i}(y)\|\!\le\! L\|x-y\|$ for all $i$.

 Denote $\overline{f}({\bf{x}})\!=\!\frac{1}{n}\sum_{i=1}^{n}f_{i}(x_{i})$, $\overline{g}({\bf{x}})\!=\!\frac{1}{n}\sum_{i=1}^{n}g_{i}(x_{i})$
 and $\overline{r}({\bf x})\!=\!{\bf 1}_{V}({\bf x})\!+\!\frac{1}{n}\sum_{i=1}^{n}r(x_{i})$,
 where ${\bf 1}_{V}({\bf x})$ is an indicator function over $V$. $V=\{{\bf x}=(x_1,x_{2},\ldots,x_{n})\in \mathbb{X}^{n}\mid \mathbb{X}^{n}=\mathbb{X}\times\mathbb{X}\times\ldots\times\mathbb{X},\ x_{1}=x_{2}=\cdots=x_{n}\}$. Then, problem \eqref{PPG-problem} is equivalent to $\min_{\bf x}  \overline{f}({\bf x})+\overline{g}({\bf x})+\overline{r}({\bf x})$ and
 \begin{equation}\label{PPG-problem-eq}
  0\in \nabla\overline{f}({\bf x})+\partial\overline{r}({\bf x}) +\partial\overline{g}({\bf x}), {\bf x} \in \mathbb{X}^{n}.
 \end{equation}
 Following the notation in \cite{shen2017over}, for $\alpha > 0$ we define the set-valued mapping
 $\mathcal{S}_{\alpha,\nabla\overline{f}+\partial\overline{g},\overline\partial{r}}:\mathbb{X}^{n}\rightrightarrows\mathbb{X}^{n}$
  as:
 \begin{align*}
 &\!{\rm gph}\big(\mathcal{S}_{\alpha,\nabla\overline{f}+\partial\overline{g},\overline\partial{r}}\big)
   \!\!=\!\!  \left\{\!({\bf x}_{1}\!+\!\alpha {\bf y}_{2}, {\bf x}_{2}\!-\!{\bf x}_{1})\!\mid\!({\bf x}_{2},{\bf y}_{2})\!\in\!{\rm gph}\overline\partial{r}, \right.\nonumber\\
               & \qquad\quad\quad \left. ({\bf x}_{1},{\bf y}_{1})\!\in\!{\rm gph}\,(\nabla\overline{f}\!+\!\overline\partial{g}),  {\bf x}_{1}\!+\!\alpha{\bf y}_{1} \!=\! {\bf x}_{2}\!-\! \alpha{\bf y}_{2}\right\}.
 \end{align*}
 By the convexity of $\overline{f},\overline{g}$ and $\overline{r}$, $\mathcal{S}_{\alpha,\nabla\overline{f}+\partial\overline{g},\partial\overline{r}}$ is a maximal monotone operator \cite{eckstein1992douglas}.
 To obtain the over-relaxed PPG algorithm, we replace $z_{i}^{k+1}$ by
 \[
 z_{i}^{k+1}\!:=\!z_{i}^{k}\!+\!(1\!+\!\theta_{k})(x_{i}^{k+1}\!-\!x^{k+\frac{1}{2}}),\ i=1,\ldots,n.
 \]
 Below, we show that the over-relaxed PPG algorithm is a specific case of the VMOR-HPE framework.
 \begin{proposition}\label{PPG-HPE}
 Let $(x^{k+\frac{1}{2}},x_{i}^{k},z_{i}^{k})$ be the sequence generated by the over-relaxed PPG algorithm.
 Denote ${\bf x}^{k}=(x_{1}^{k},\cdots,x_{n}^{k})$, ${\bf z}^{k}=(z_{1}^{k},\cdots,z_{n}^{k})$, ${\bf 1}=(1,\cdots,1)\!\in\!\mathbb{X}^{n}$,
 ${\bf y}^{k}={\bf z}^{k}+ {\bf x^{k+1}}- x^{k+\frac{1}{2}}{\bf 1}$, ${\bf v}^{k}= x^{k+\frac{1}{2}}{\bf 1}-{\bf x}^{k+1}$ and $\epsilon_{k}=L\sum_{i=1}^{n}\|x_{i}^{k+1}-x^{k+\frac{1}{2}}\|/4$.
 Parameters $(\theta_{k},\alpha)$ are constrained by $\theta_k + L\alpha/2\le \sigma$. Then, it holds that
  \begin{subnumcases}{}
  \!\!\!({\bf y}^{k},{\bf v}^{k})\in {\rm gph}\,\mathcal{S}_{\alpha,\nabla\overline{f}+\partial\overline{g},\overline\partial{r}}^{[\alpha\epsilon_{k}]}={\rm gph}\,T^{[\alpha\epsilon_{k}]}, \nonumber\\
  \!\!\!\theta_k\big\|{\bf v}^{k}\big\|^2\!+\!\big\|{\bf v}^{k}+({\bf y}^{k}\!-\!{\bf z}^{k})\big\|^2+2\alpha\epsilon_{k}
               \!\leq\! \sigma\big\|{\bf y}^{k}\!-\!{\bf z}^{k}\big\|^2, \nonumber\\
  \!\!\!{\bf z}^{k+1}={\bf z}^{k}- (1+\theta_{k}){\bf v}^{k}. \nonumber
 \end{subnumcases}
 \end{proposition}
 \begin{remark}\label{PPG-HPE-remark}
 {\bf (i)}\ Let $\theta_{k}\!=\!0$.  $\alpha\!<\!2/L$ can guarantee the global convergence of the original PPG algorithm, which largely expands the region $\alpha\!<\!3/(2L)$ in \cite{ryu2017proximal}.\\
 {\bf (ii)}\ PPG algorithm has been shown to cover ADMM \cite{boyd2011distributed} and Davis-Yin 3OS algorithm \cite{davis2015three}.
 Thus, they also fall into the VMOR-HPE framework.
 \end{remark}
 {\bf AFBAS Algorithm}\ Let $A\!:\!\!\mathbb{X}\!\rightrightarrows\!\mathbb{X}$ be a maximally monotone operator, $M\!:\!\mathbb{X}\!\to\!\mathbb{X}$ be a linear operator,
 and $C\!:\!\mathbb{X}\!\to\!\mathbb{X}$ be a $\beta$-cocoercive operator with respect to $\|\cdot\|_{P}$ satisfying
  $\big\langle x\!-\!x',C(x)\!-\!C(x')\big\rangle\!\ge\!\beta\big\|C(x)\!-\!C(x')\big\|^2_{\!P^{-1}}$, respectively.
  The AFBAS algorithm solves problem \eqref{inclusion-problem} as below:
 \[
 0 \in T(x)=(A+M+C)(x),\ x\in \mathbb{X}.
 \]
 Let $S\!:\mathbb{X}\to \mathbb{X}$ be any self-adjoint positive definite linear operator and $K\!:\mathbb{X}\to\mathbb{X}$ be a skew adjoint operator, respectively.
 Denote $H=P+K$. Then, the AFBAS algorithm is defined as:
 \vspace{-2pt}
 \begin{subnumcases}{}
 \overline{x}^{k} := (H+A)^{-1}\big(H-M-C\big)x^{k}, \nonumber\\ 
 x^{k+1} := x^{k}+\alpha_{k}S^{-1}(H+M^*)(\overline{x}^{k}-x^{k}), \qquad  \nonumber 
  \vspace{-6pt}
 \end{subnumcases}
 where  $\alpha_{k}\!=\!\left[\lambda_{k}\|\overline{z}^{k}\!-\!z^{k}\|^2_{\!P}\|\right]{\big/}\left[\|(H+M^*)(\overline{z}^{k}\!-\!z^{k})\|^{2}_{\!S^{-1}}\right]$
 and $\lambda_{k}\in [\underline{\lambda},\overline{\lambda}]\le [0,2-1/(2\beta)]$.
 Throughout \cite{latafat2017asymmetric}, $M$ is specified to a skew-adjoint linear operator, \textit{i.e.}, $M^*=-M$.
 \begin{proposition}\label{AFBAS-VMORHPE}
 Let $(x^{k},\overline{x}^{k})$ be the sequence generated by the AFBAS algorithm.
 Denote $\theta_{k}=\alpha_{k}-1$, $v^{k}=(H+M^*)(x^{k})-(H+M^*)(\overline{x}^{k})$ and $\epsilon_{k} = \frac{\|\overline{z}^{k}-z^{k}\|^2_{P}}{4\beta}$. Then,
 \vspace{-4pt}
  \begin{subnumcases}\!{}
 \!\!(\overline{x}^{k},v^{k}) \in {\rm gph}\,(A+M+C)^{[\epsilon_{k}]},\nonumber\\
 \!\!\theta_k\big\|S^{-1}v^{k}\big\|_{S}^2\!+\!\big\|S^{-1}v\!+\!(\overline{x}^{k}\!-\!x^{k})\big\|_{S}^2\!+\!2\epsilon \!\leq\! \sigma\big\|\overline{x}^{k}\!-\!x^{k}\big\|_{S}^2,\nonumber\\
 \!\!x^{k+1}:=x^{k}-(1+\theta_{k})S^{-1}v^{k}.\nonumber
 \end{subnumcases}
 \vspace{-12pt}
 \end{proposition}
 In \cite{latafat2017asymmetric}, a few new algorithms, such as forward-backward-forward splitting algorithm with only one evaluation of $C$,
 Douglas-Rachford splitting algorithm with an extra forward step, \textit{etc}, are put forward based on the AFBAS algorithm.
 By Proposition \ref{AFBAS-VMORHPE}, VMOR-HPE also covers these new splitting algorithms as special cases.

\subsubsection{primal-dual algorithms}\label{primal-dual algorithms}

In this subsection, we focus on the existing primal-dual algorithms in the literature for solving the problem below:
 \begin{equation}\label{three-convex}
 \min f(x)+g(x)+h(Bx),\ x\in \mathbb{X},
  \vspace{-1pt}
 \end{equation}
 where $B\!:\mathbb{X}\!\to\!\mathbb{Y}$ is a linear operator, $g\!:\mathbb{X}\!\to\!(-\infty, +\infty]$ and $h\!:\mathbb{Y}\!\to\!(-\infty, +\infty]$ are closed proper convex functions,
 and $f\!:\mathbb{X}\!\to\!(-\infty,\infty)$ is a differentiable convex function satisfying $\|\nabla f(x)-\nabla f(x')\| \le L\|x-x'\|$ for all $x,x'\in\mathbb{X}$.
  By introducing the dual variable $y\in \mathbb{Y}$ and denoting $\mathbb{Z}=\mathbb{X}\times\mathbb{Y}$, problem \eqref{three-convex} can be formulated as:
  \begin{equation}\label{three-convex-inclusion}
  0\!\in\! T(z)\!=\! \left[
    \begin{array}{c}
    \partial g(x)\\
    \partial h^*(y)\\
    \end{array}
       \right]\!+\!
   \left[
   \begin{array}{c}
   \nabla f(x)\!+\! {B}^*y \\
    -{B}x \\
   \end{array}
   \right], z\in \mathbb{Z}.
 \end{equation}
 {\bf Condat-Vu PDS Algorithm} is proposed to solve problem \eqref{three-convex} with the following iterations:
  \vspace{-2pt}
 \begin{subnumcases}\!{}
 \!\!\!\widetilde{x}^{k+1}:= {\rm Prox}_{r^{-1} g}\big(x^{k}- r^{-1}\nabla f(x^{k})-r^{-1}{B}^*y^{k}\big), \nonumber\\
 \!\!\!\widetilde{y}^{k+1}:={\rm Prox}_{s^{-1} h^*}\big(y^{k}+s^{-1}{B}(2\widetilde{x}^{k+1}-x^{k})\big), \nonumber\\
 \!\!\!\!(x^{k\!+1},y^{k\!+1}) \!:=\! (x^{k},y^{k})\!+\!(1\!+\!\theta_{k})\big((\widetilde{x}^{k\!+1},\widetilde{y}^{k\!+1})\!\!-\!\!(x^{k},y^{k})\big). \nonumber
 \end{subnumcases}
 We denote $\mathcal{M}\!:\!\mathbb{Z}\!\to\!\mathbb{Z}$ as $\mathcal{M}\!\!=\!\![r~ -B^*;-B~ s]$
 and show that the Condat-Vu PDS algorithm is covered by VMOR-HPE.
 \begin{proposition}\label{condat-vu-HPE}
 Let $\{(x^{k},y^{k},\widetilde{x}^{k},\widetilde{y}^{k})\}$ be the sequence generated by the Condat-Vu PDS algorithm.
 Let $z^{k}\!=\!(x^{k},y^{k}),w^{k}\!=\!(\widetilde{x}^{k+1},\widetilde{y}^{k+1})$.
 Parameters $(r,s,\theta_{k})$ satisfy
  \vspace{-2pt}
 \begin{equation}\label{condition-r-s-theta}
 s-r^{-1}\|\mathcal{B}\|^2 > 0, \theta_{k}+L/[2(s-r^{-1}\|\mathcal{B}\|^2)] \le \sigma.
 \end{equation}
 Denote $v^{k}=\mathcal{M}(z^{k}-w^{k})$, $\epsilon_{k}=L\|x^{k}-\widetilde{x}^{k+1}\|^2/4$. Then,
 \vspace{-1pt}
  \begin{subnumcases}{}
 \!\! v^{k} \in T^{[\epsilon_{k}]}(w^{k}), \nonumber\\
 \!\!\!\theta_{k}\!\big\|\!\mathcal{M}^{-1}v^{k}\big\|_{\!\mathcal{M}}^2\!\!+\!\!\big\|\!\mathcal{M}^{-1}v^{k}\!\!+\!w^{k}\!\!-\!\!z^{k}\big\|_{\!\mathcal{M}}^2 \!\!+\!\!2\epsilon_{k}\!\!\leq\! {\sigma}\big\|w^{k}\!\!-\!\!z^{k}\big\|_{\!\mathcal{M}}^2, \qquad \nonumber\\
 \!\! z^{k+1} = z^{k}-(1+\theta_{k})\mathcal{M}^{-1}v^{k}. \nonumber
 \end{subnumcases}
 \end{proposition}
 \begin{remark}
 {\bf (i)}\ The condition \eqref{condition-r-s-theta} is much more mild compared with $s-r^{-1}\|\mathcal{B}\|^2\!>L/2,\theta_{k}\!+\!L/[2(s\!-\!r^{-1}\|\mathcal{B}\|^2)]\!<1$
 in \cite{condat2013primal,vu2013splitting} and $s-r^{-1}\|\mathcal{B}\|^2 > L/2,\theta_{k}+L/[s-r^{-1}\|\mathcal{B}\|^2] < 1$ in \cite{chambolle2016ergodic}. \\
 {\bf (ii)}\ The metric version of Condat-Vu PDS algorithm \cite{Li2016Fast} with $(s = S, r = R)$ also falls into the VMOR-HPE framework
   by replacing condition \eqref{condition-r-s-theta} with $\|R^{-\frac{1}{2}}BS^{-\frac{1}{2}}\|\!< 1$  and $\theta_{k}\!+\!L/(2\lambda_{\min}(\mathcal{M}))\le \sigma$.\\
 {\bf (iii)}\ If $f=0$, the Condat-Vu PDS algorithm recovers PDHG algorithm \cite{chambolle2011first}  which is also covered by the VMOR-HPE framework.
 \end{remark}
 \noindent
 {\bf AFBAS-PD Algorithm}\  Applying the AFBAS algorithm for \eqref{three-convex-inclusion}
 yields the {\bf P}rimal-{\bf D}ual (AFBAS-PD) algorithm:
 \begin{subnumcases}{}
 \overline{x}^{k}:= {\rm Prox}_{\gamma_{1}g}\big(x^{k}-\gamma_{1}{B}^*y^{k}-\gamma_{1}\nabla f(x^{k})\big), \nonumber\\
 \overline{y}^{k}:={\rm Prox}_{\gamma_{2} h^*}\big(y^{k}+\gamma_{2}{B}((1-\theta)x^{k}+\theta\overline{x}^{k})\big),\nonumber\\
 \!\!x^{k+1} \!:=\! x^{k}\!+\!\alpha_{k}\big((\overline{x}^{k}-x^{k})\!-\mu\gamma_{1}(2\!-\theta)B^*(\overline{y}^{k}-y^{k})\big),\nonumber\\
 \!\!y^{k+1} \!:=\!  y^{k}\!+\!\alpha_{k}\big(\gamma_2(1\!-\!\mu)(2\!-\!\theta)B(\overline{x}^{k}\!-\!x^{k})\!+\!(\overline{y}^{k}\!-\!y^{k})\big),\nonumber
 \end{subnumcases}
 where $\alpha_{k}$ is adaptively tuned and $(\gamma_1,\gamma_2, \theta, \mu)$ satisfy $\mu\in [0,1]$, $\theta\in [0,\infty)$ and $\gamma_{1}^{-1}-\gamma_{2}\theta^2\|B\|^2/4 > L/4$.

 Denote a linear operator $M\!:\!\mathbb{Z}\!\to\!\mathbb{Z}$ with $M=\!RS^{-1}$, where $(R,S)$ are defined as
 $R=[\gamma_{1}^{-1}\ \ -B^*;\ \ (1\!-\!\theta)B\ \ \gamma_{2}^{-1} ]$ and
 \vspace{-0.1cm}
 \begin{equation*}
  S =\left[
   \begin{array}{cc}
     1 & -\mu\gamma_{1}(2\!-\!\theta)B^* \\
     \gamma_2(1\!-\!\mu)(2\!-\!\theta)B & 1 \\
   \end{array}
 \right].
 \end{equation*}
 In addition, by \cite{horn1990matrix}, it is easy to verify that $M$ is a self-adjoint positive definite linear operator.
 \vspace{-0.1cm}
 \begin{proposition}\label{AFBAS-PD-HPE}
 Let $\{(\overline{x}^{k},\overline{y}^{k},x^{k},y^{k})\}$ be the sequence generated by the AFBAS-PD algorithm.
 Denote $w^{k}\!=\!(\overline{x}^{k},\overline{y}^{k})$, $z^{k}\!=\!(x^{k},y^{k})$, $v^{k}\!=\!R(z^{k}\!-\!w^{k})$, $\epsilon_{k}\!=\!L\|x^{k}\!-\!\overline{x}^{k}\|^2/4$,
 and $\theta_{k}=\alpha_{k}-1$. Then, it holds that
  \begin{subnumcases}\!{}\!
 \!\!v^{k} \in T^{[\epsilon_{k}]}(w^{k}),\nonumber\\
 \!\!\!\theta_{k}\!\big\|\!\mathcal{M}^{-1}v^{k}\big\|_{\!\mathcal{M}}^2\!\!+\!\!\big\|\!\mathcal{M}^{-1}v^{k}\!\!+\!w^{k}\!\!-\!\!z^{k}\big\|_{\!\mathcal{M}}^2 \!\!+\!\!2\epsilon_{k}\!\!\leq\! {\sigma}\big\|w^{k}\!\!-\!\!z^{k}\big\|_{\!\mathcal{M}}^2,  \nonumber\\
 \!\!z^{k+1} = z^{k}-(1+\theta_{k})\mathcal{M}^{-1}v^{k}.\nonumber 
 \end{subnumcases}
 \end{proposition}
 \vspace{-0.1cm}
  The AFBAS-PD algorithm \cite{latafat2017asymmetric} recovers: (i) the Condat-Vu PDS algorithm with an adaptive over-relaxed step-size if $\theta\!=\!2$;
(ii) the Combettes PDS algorithm if $\theta\!=\!0$ and $\mu\!=\!\frac{1}{2}$;
(iii)  the MSS algorithm if $\theta\!=\!0,\mu\!= 1/2$ and $h\!=\!0$;
(iv)  the PAPC algorithm if $\theta\!=\!1,\mu\!=\!1$ and $f\!=\!0$.
  Thus, they are also covered by VMOR-HPE.

 To close this subsection \ref{primal-dual algorithms}, we make some comments on the PD3OS and PDFP algorithms which coincide with each other by \cite{tang2017general}.
 By Remark \ref{PPG-HPE-remark} and \cite{o2017equivalence},
 the PD3OS and PDFP algorithms are both covered by the algorithmic framework of VMOR-HPE.
\section{PADMM-EBB Algorithm}
The KKT generalized equation of problem \eqref{linearly-constriant} is defined as
\begin{equation}\label{inclusion-linearly-constriant}
\!\!\!T(z) \!\!=\!\! \left[
   \begin{array}{c}
 \partial g_{1}(x_{1})     \\
 \vdots                    \\
 \partial g_{p}(x_{p})     \\
  b
   \end{array}
   \right]
\!\!+\!\!\left[
   \begin{array}{c}
 \nabla f_{1}(x)  \!+\! \mathcal{A}_{1}y   \\
                   \vdots               \\
  \nabla f_{p}(x) \!+\!  \mathcal{A}_{p}y   \\
  - \sum_{i=1}^{p}\mathcal{A}_{i}^*x_{i}
   \end{array}
   \right],\, 0\!\in\! T(z),
\end{equation}
where $\nabla{f}_{i}(x)$ is the $i$-th component of $\nabla{f}(x)$ and $y\in \mathbb{Y}$ is the Lagrange multiplier.
Let $\mathbb{Z}\!=\!\mathbb{X}\!\times\!\mathbb{Y}$, $\mathbb{X}\!:=\! \mathbb{X}_{1}\!\times\!\cdots\!\times\!\mathbb{X}_{p}$, $x\!=\!(x_{1},\ldots,x_{p})\!\in\!\mathbb{X}$,
$z\!=\!(x_{1},\ldots,x_{p},y)\!\in\!\mathbb{Z}$, and $\widehat{L}_{(\beta,\,x^{k})}$ be the majorized augmented Lagrange function as
\begin{align}
 \widehat{L}_{(\beta_{k},\,x^{k})}(x,y)&=f(x^{k},x)+\big\langle\textstyle\sum_{i=1}^{p}\mathcal{A}^*_{i}x_{i}- b,y\big\rangle  \label{majorized-ALF} \\
                                      &\quad +\!\textstyle{\sum_{i=1}^{p}}g_i(x_i)\!+\!\frac{\beta_{k}}{2}\big\|\textstyle \sum_{i=1}^{p}\mathcal{A}_{i}x_{i}\!-\!b\big\|^2, \nonumber
 \vspace{-5pt}
\end{align}
where $f(x^{k},x)=f(x^{k})+\langle \nabla f(x^{k}),x-x^{k}\rangle+\frac{1}{2}\|x-x^{k}\|^2_{\widehat{\Sigma}}$ and $\hat{\Sigma}$ is a self-adjoint positive semi-definite linear operator.

In the implementation of VMOR-HPE, generating $(v^{k},y^{k},\epsilon_{k})$ satisfying \eqref{VMOR-HPE-a}-\eqref{VMOR-HPE-b} is equal to performing a non self-adjoint {\bf P}roximal ADMM to
problem \eqref{linearly-constriant} and $x^{k+1}\!:=\!x^k\!-\!(1\!+\!\theta_k)c_k\mathcal{M}_{k}^{-1} v^{k}$ in VMOR-HPE for problem \eqref{linearly-constriant} corresponds to
performing an {\bf E}xtra-gradient correction step to ensure the global convergence of PADMM. Additionally, $\mathcal{M}_{k}$ is determined by a {\bf B}arzilai-{\bf B}orwein line search technique
to explore the curvature information of the KKT operator $T$. The PADMM-EBB is described in Algorithm \ref{Alg:Practical-HPE-VMPD}.
\vspace{-5pt}
\begin{algorithm}[h]
   \caption{\, PADMM-EBB Algorithm}
   \label{Alg:Practical-HPE-VMPD}
\begin{algorithmic}
   \STATE {\bf Parameters:} Given $\xi_{k}\ge 0$ satisfying $\sum_{i=1}^{\infty}\!\xi_{i}\!<\infty$,
             $\tau,\overline{\theta}>0,\,-1<\!\underline{\theta}\!<0$, and $\overline{\sigma}\!\in [0,1)$.
             Choose a linear operator $\mathcal{M}_{0}\succ 0$ and starting points $x^{0}\!\in\!\mathbb{X}$, $y^{0}\!\in\!\mathbb{Y}$.
   \FOR{$k=0,1,2,\ldots,$}
   \STATE    For $i=1,2,\ldots, p$, $\widetilde{x}_{i}^{k+1}$ solves the inclusion as below
                 \vspace{-0.3cm}
              \begin{equation*}
               \!0\!\in\! \partial_{x_{i}}\widehat{L}_{(\beta_{k},x^{k})}(\ldots,\widetilde{x}^{k\!+\!1}_{i\!-\!1},x_{i},x^{k}_{i\!+\!1},\ldots,y^{k})
              \!+\!P^{k}_{i}(x_{i}\!-\!x^{k}_{i}).
              \end{equation*}
   \vspace{-0.4cm}
   \STATE $\widetilde{y}^{k+1}:= y^{k}+\beta_{k}\big(\mathcal{A}^*_{1}\widetilde{x}^{k+1}_{1}+\mathcal{A}^*_{2}x^{k}_{2}+\ldots+\mathcal{A}^*_{p}x^{k}_{p}-b\big)$.
   \vspace{-0.1cm}
   \STATE \! Set $\theta_{k}\!\in\!\![\theta^{\rm fix}_{k},\theta^{\rm adap}_{k}]$ with $\theta^{\rm fix}_{k}\!\in\!\!\big[\underline{\theta},\overline{\theta}_k\big]$
           via \eqref{PADMM-EBB-theta1}-\eqref{PADMM-EBB-theta1}.
    \vspace{-0.2cm}
   \STATE $z^{k+1}\!:=\!z^{k}+(1\!+\!\theta_k)\mathcal{M}^{-1}_{k}U_{k}(w^{k}-z^{k})$, where $(z^{k},w^{k})$ are defined as $z^{k}\!=\!(x^{k},y^{k})^\top$, $w^{k} \!=\! (\widetilde{x}^{k+1},\widetilde{y}^{k+1})^\top$.
   \STATE Update $\mathcal{M}^{-1}_{k+1}\!\!=\!\!{\rm Diag}(M^{k+1}_1,\cdots,M^{k+1}_{p},M^{k+1}_{p+1})$.
   \ENDFOR
 \end{algorithmic}
 \end{algorithm}
   \vspace{-10pt}

  In this algorithm, each $M^{k+1}_i$ for $i\!=\!1,\ldots,p$ is defined as
  \begin{gather*}
  M^{k+1}_{i}\!:=\min\big({\|\widetilde{x}_{i}^{k+1}\!\!-\!\widetilde{x}^{k}_{i}\|}/{\|s_{k+1}\!-\!s_{k}\|},(1\!+\!\xi_{k})M^{k}_{i+1}\big),
  \vspace{-5pt}
  \end{gather*}
  where $s_{k+1}=(U^{k}(z^{k}-w^{k}))_{i}+\nabla f_{i}(\widetilde{x}_{i}^{k+1})-\nabla{f}_{i}(x_{i}^{k})$. In addition, let $r_{k+1} = \beta^{-1}_{k}(y^{k}-\widetilde{y}^{k+1})+\sum_{i=2}^{p}\mathcal{A}^*_{p}({x}^{k}_{i}-\widetilde{x}_{i}^{k+1})$. The metric $M^{k+1}_{p+1}$ is defined as
  \[
  M^{k+1}_{p+1}:=\min\big({\big\|\widetilde{y}^{k+1}-\widetilde{y}^{k}\big\|}/{\|r_{k+1}-r_{k} \|},(1+\xi_{k})M^{k}_{p+1}\big).
  \]
  Let $\mathcal{D}={\rm Diag}(L_{1}\mathcal{I}\ \ \cdots\ \ L_{p}\mathcal{I}\ \  0)$ and
  $\Gamma_{k}\!=\! U^{k}\!+\!(U^{k})^*\!+\!(\overline{\sigma}\!-\!1)\mathcal{M}_{k}\!-\!\mathcal{D}/2$.
   Parameters $(\overline{\theta}_{k},  \theta^{\rm adap}_{k})$ are defined as
 \begin{subnumcases}{}
 \!\!\!\overline{\theta}_{k}\!=\!\max\left\{ \theta \mid  (\theta+1)(U^{k})^*\mathcal{M}^{-1}_{k}U^{k}\preceq \Gamma_{k}
  \right\}, \label{PADMM-EBB-theta1} \\
 \!\!\!\theta^{\rm adap}_{k} \!\!=\!-1\!+\! {\big\|z^{k}\!-\!w^{k}\big\|^2_{\Gamma_{k}}}
                  {\Big /}\big\|z^{k}\!-\!w^{k}\big\|^2_{(U^{k})^*\mathcal{M}^{-1}_{k}U^{k}}.\qquad~ \label{PADMM-EBB-theta1}
 \end{subnumcases}
 In addition, $P^{k}_{i}:\mathbb{X}_{i}\to \mathbb{X}_{i}$ for $i=1,2,\ldots,p$ are non self-adjoint linear operators, $\mathcal{T}_{i}=\widehat{\Sigma}_{i}+P^{k}_{i}+\beta_{k}\mathcal{A}_{i}\mathcal{A}_{i}^*$,
 and $U_{k}:\mathbb{Z}\to\mathbb{Z}$ is a block linear operator defined as below
\begin{align*}
\!\!\!U^{k}\!=\!\left(\begin{matrix}
\widehat{\Sigma}_{1}+P^{k}_1 &  0  & \ldots  & 0  & 0  \\
 0    & \mathcal{T}_{2}& \ldots   &   0  & 0  \\
 \vdots    & \vdots & \ddots & \vdots &  \vdots  \\
 0   & \beta_{k}\mathcal{A}_{p}\mathcal{A}_{2}^* & \cdots  &\mathcal{T}_{p} & 0  \\
 0    & \mathcal{A}_{2}^* & \cdots & \mathcal{A}_{n}^* & \beta_{k}^{-1}\mathcal{I}\! \\
   \end{matrix}\right).
\end{align*}
\begin{remark}
To ensure $1+\theta_{k} >0$, $P^{k}_{i}$ should be chosen to make $U^{k}+(U^{k})^*\succ \mathcal{D}/2$. In addition,
the non self-adjoint linear operator $P^{k}_{i}$ in inclusion with respect to $\widetilde{x}_{i}^{k+1}$
is chosen to approximate $\beta_{k}\mathcal{A}_{i}\mathcal{A}_{i}^*+\widehat{\Sigma}$ more tightly
and make the inclusion easier to solve than the common settings.
\end{remark}

 \begin{theorem}\label{convergence-nPADMM-EQN}
 Let $(\widetilde{x}^{k},\widetilde{y}^{k},x^{k},y^{k})$ be the sequence generated by the PADMM-EBB algorithm.
 Denote $v^{k}\!=\!U^{k}(z^{k}\!-\!w^{k})$, $\epsilon_{k}\!=\!\|x^{k}\!-\!\widetilde{x}^{k+1}\|_{\mathcal{D}}/4$ and operator $T$ as \eqref{inclusion-linearly-constriant}.
 Then, it holds
 \begin{subnumcases}{}
 \!\!\!\!v^{k} \in T^{[\epsilon_{k}]}(w^{k}),\nonumber\\
 \!\!\!\!\theta_{k}\big\|\mathcal{M}_{k}^{-1}v^{k}\!\big\|_{\!\mathcal{M}_{k}}^2\!\!\!\!\!+\!\!\big\|\mathcal{M}_{k}^{-1}v^{k}\!+\!w^{k}\!\!\!-\!\!z^{k}\!\big\|_{\!\mathcal{M}_{k}}^2
                 \!\!\!\!\!+\!2\epsilon_{k}\!\leq\! \sigma\big\|w^{k}\!\!\!-\!\!z^{k}\!\big\|_{\!\mathcal{M}_{k}}^2,\qquad\quad \nonumber\\
 \!\!\!\!z^{k+1} = z^{k}-(1+\theta_{k})\mathcal{M}_{k}^{-1}v^{k}.\nonumber
 \end{subnumcases}
 Besides, {\bf (i)}\ $(x^{k},\widetilde{x}^{k})$ and $(y^{k},\widetilde{y}^{k})$ converge to $x^{\infty}$ and $y^{\infty}$ belonging to the primal-dual solution set of problem \eqref{linearly-constriant}. \\
 {\bf (ii)}\ There exits an integer $\overline{k}\in \{1,2,\ldots,k\}$ such that
 \vspace{-0.2cm}
 \[
 \!\!\sum_{i=1}^{p}\!\!{\rm dist}\big((\partial g_{i}\!+\!\nabla\!f_{i})(\widetilde{x}^{\overline{k}})+ \mathcal{A}_{i}\widetilde{y}^{\overline{k}},0\big)
 \!+\!\big\|b-\!\!\sum_{i=1}^{p}\!\!\mathcal{A}_{i}^{*}\widetilde{x}^{\overline{k}}_{i}\big\|\!\le\! \mathcal{O}(\frac{1}{\sqrt{k}}).
 \vspace{-0.1cm}
 \]
 {\bf (iii)}\ Let $\alpha_{i}\!=\!1\ {\rm or}\ i$. There exists $0\le\overline{\epsilon}^{x_{i}}_{k}\le\mathcal{O}(\frac{1}{k})$ such that
 \vspace{-0.2cm}
 \[
 \!\sum_{i=1}^{p}\!\!{\rm dist}\!\big((\partial g_{i}\!+\!\nabla f_{i})_{\overline{\epsilon}^{x_{i}}_{k}}\!(\overline{x}^{k})\!+ \mathcal{A}_{i}\overline{y}^{k},0\big)
 \!+\!\big\|b-\!\!\sum_{i=1}^{p}\!\mathcal{A}_{i}^{*}\overline{x}^{k}_{i}\big\|\!\le\! \mathcal{O}(\frac{1}{k}),
 \vspace{-0.1cm}
 \]
 where $\overline{x}^{k}\!\!=\!\!\frac{\sum_{i=1}^{k}(1+\theta_{i})\alpha_{i}\widetilde{x}^{i+1}}{{\sum_{i=1}^{k}}(1+\theta_{i})\alpha_{i}}$ and
 $\overline{y}^{k}\!\!=\!\!\frac{{\sum_{i=1}^{k}}(1+\theta_{i})\alpha_{i}\widetilde{y}^{i+1}}{{\sum_{i=1}^{k}}(1+\theta_{i})\alpha_{i}}$.
   \vspace{1pt} \\
 {\bf (iv)}\ If $T$ satisfies metric subregularity at $\big((x^{\infty},y^{\infty}),0\big)\!\in\! {\rm gph}T$ with modulus $\kappa\!>\!0$.
 Then, there exits $\overline{k}\!>\!0$ such that
 \begin{align*}
  &\qquad~{\rm dist}_{\mathcal{M}_{k+1}}\big((x^{k+1},y^{k+1}),T^{-1}(0) \big) \nonumber\\
  &\leq  \Big(1-\frac{\varrho_{k}}{2}\Big){\rm dist}_{\mathcal{M}_{k}}\big((x^{k},y^{k}),T^{-1}(0)\big),\ \forall k\ge\overline{k},
 \end{align*}
  \vspace{-0.15cm}
 where $\varrho_{k}= \frac{(1-\sigma)(1+\theta_{k})}{\big(1+\kappa\sqrt{\frac{\Xi\overline{\omega}}{\underline{\omega}}}\big)^{2}\big(1+\sqrt{\sigma+\frac{4\max\{-\theta_{k},0\}}{(1+\theta_{k})^2}}\big)^{2}}\in (0,1)$.
 \end{theorem}
 \begin{remark}
  By Proposition \ref{iteration-complexity-VMOR-HPE}, the constants in $\mathcal{O}(\frac{1}{\sqrt{k}})$ pointwise iteration
  complexity and $\mathcal{O}(\frac{1}{k})$ weighted iteration complexity both depend merely on the primal-dual solution set of problem \eqref{linearly-constriant}
  without requiring the boundedness of $(\mathbb{X},\mathbb{Y})$.
 \end{remark}

\begin{figure*}[htpb]
\centering
\subfigure{\label{fig:pic1}
\includegraphics[width=0.24\linewidth]{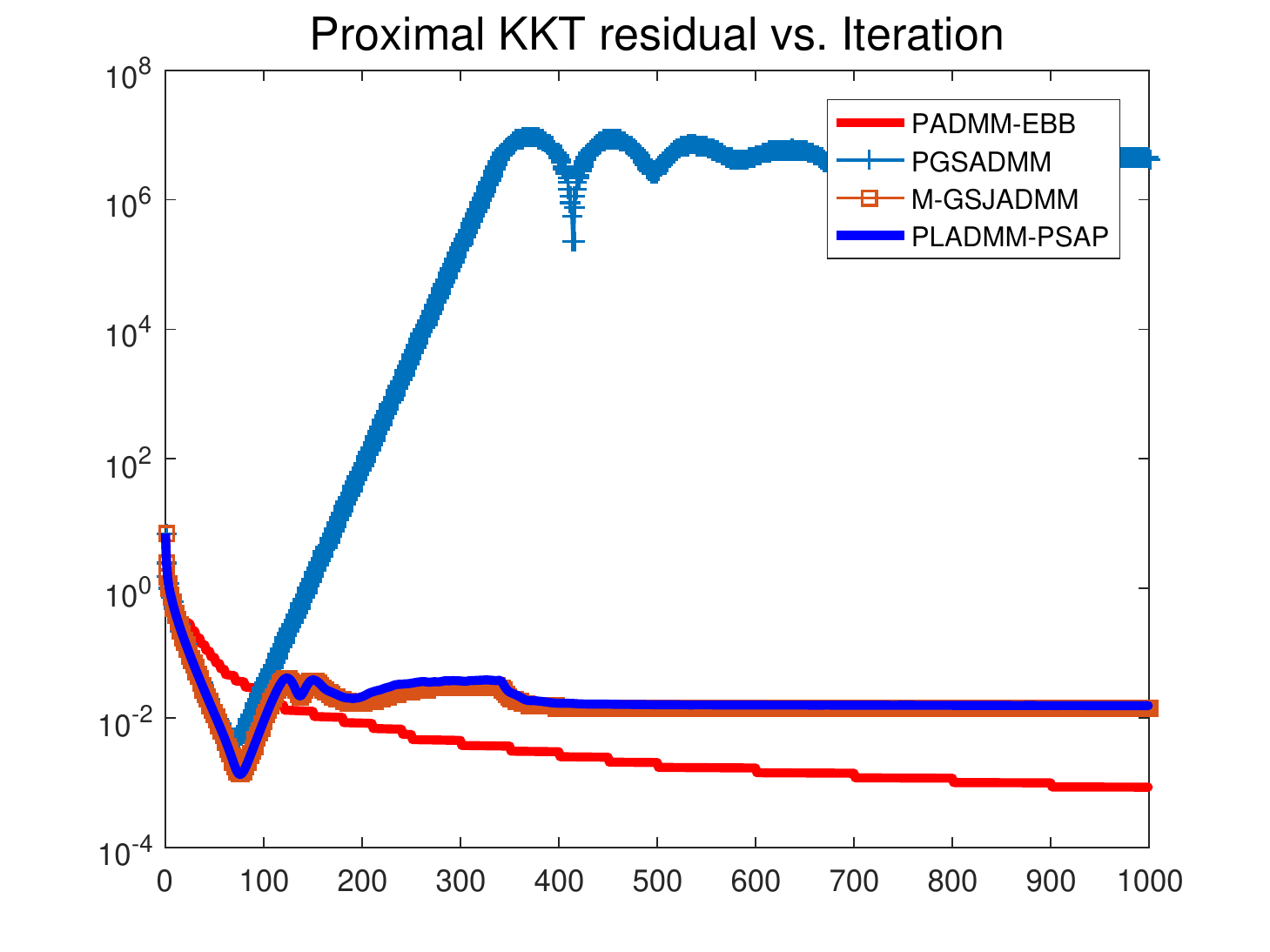}}
\subfigure{\label{fig:pic1}
\includegraphics[width=0.24\linewidth]{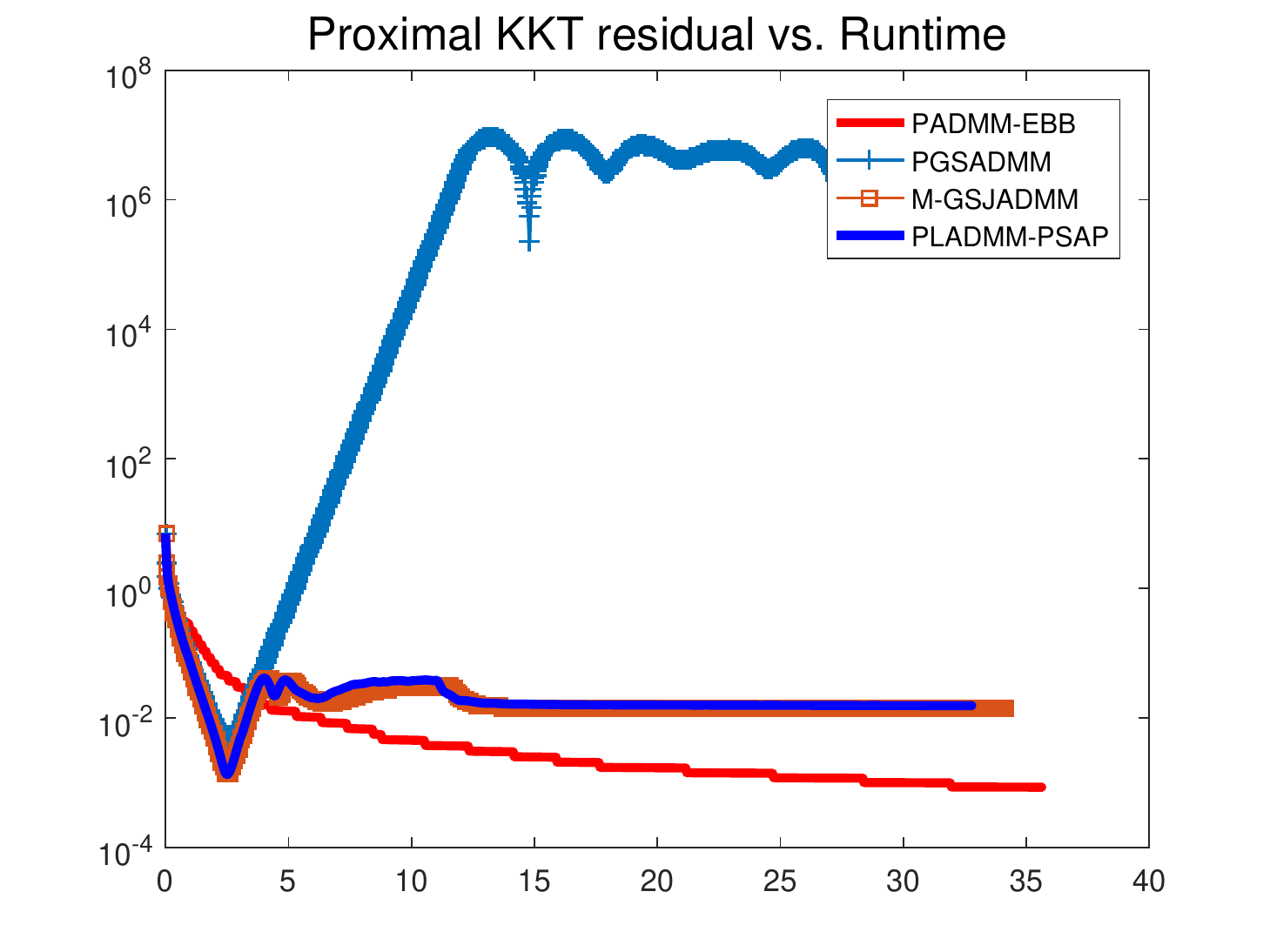}}
\subfigure{\label{fig:pic1}
\includegraphics[width=0.24\linewidth]{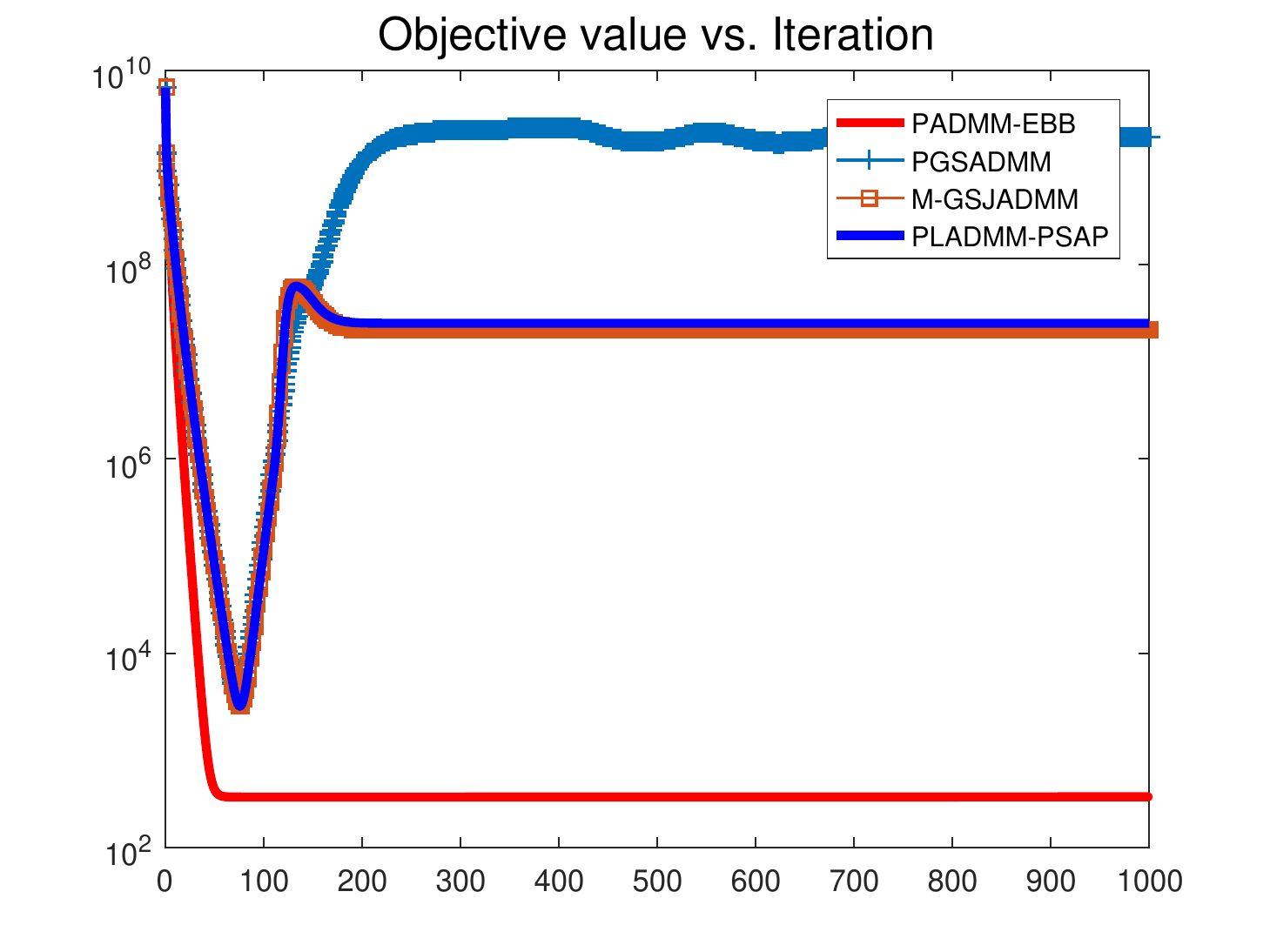}}
\subfigure{\label{fig:pic1}
\includegraphics[width=0.24\linewidth]{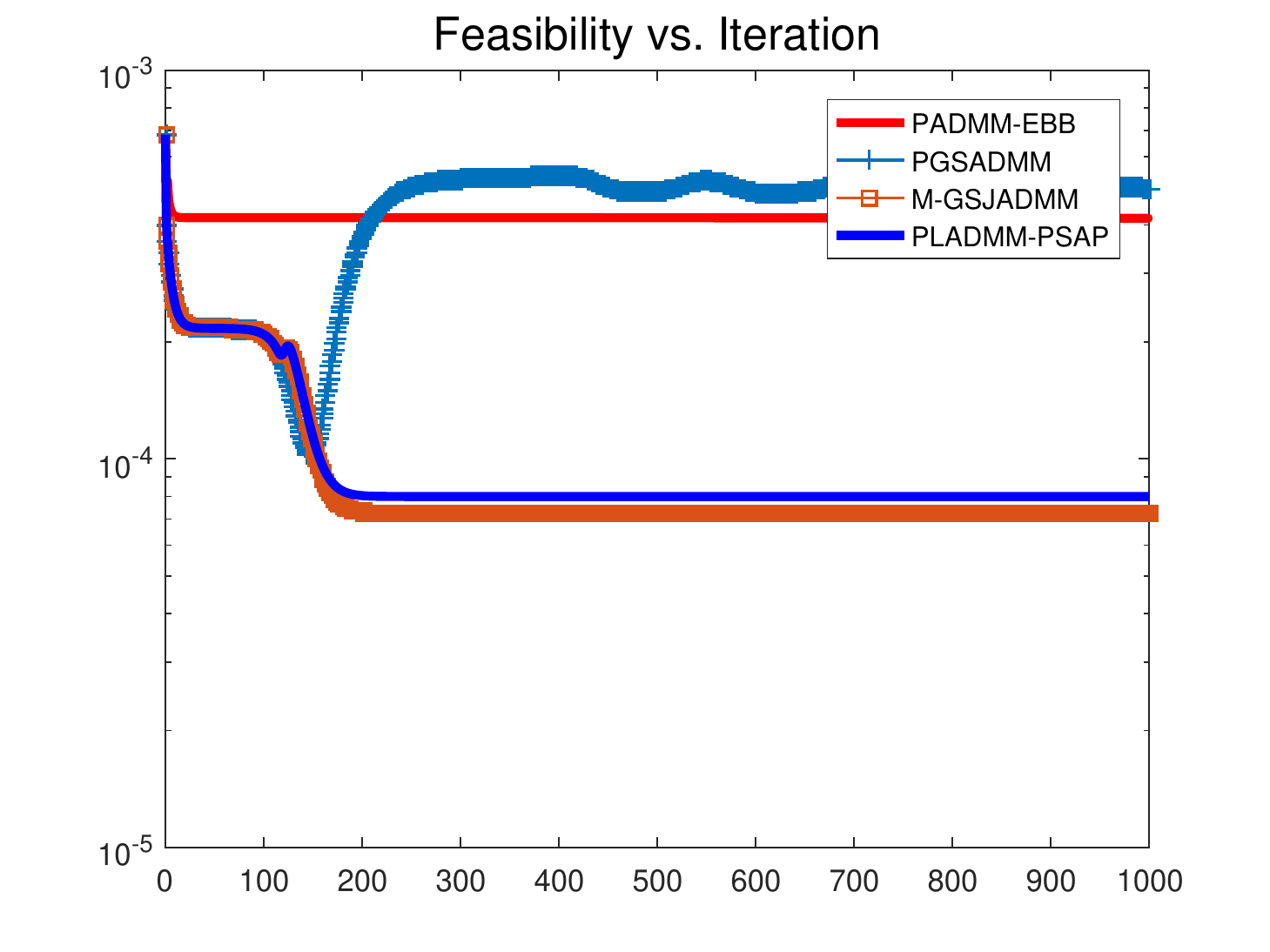}}
\vspace{-0.5cm}
 \caption{ The above four figures illustrate the proximal KKT residual vs. iteration,
  proximal KKT residual vs. runtime, objective value vs. iteration, and
  feasibility vs. iteration on the synthetic dataset with parameters $(\lambda,\mu, \gamma)=(10^3,10^4,10^4)$, respectively.}
\label{fig:overlap-1}
\centering
\subfigure{\label{fig:pic1}
\includegraphics[width=0.24\linewidth]{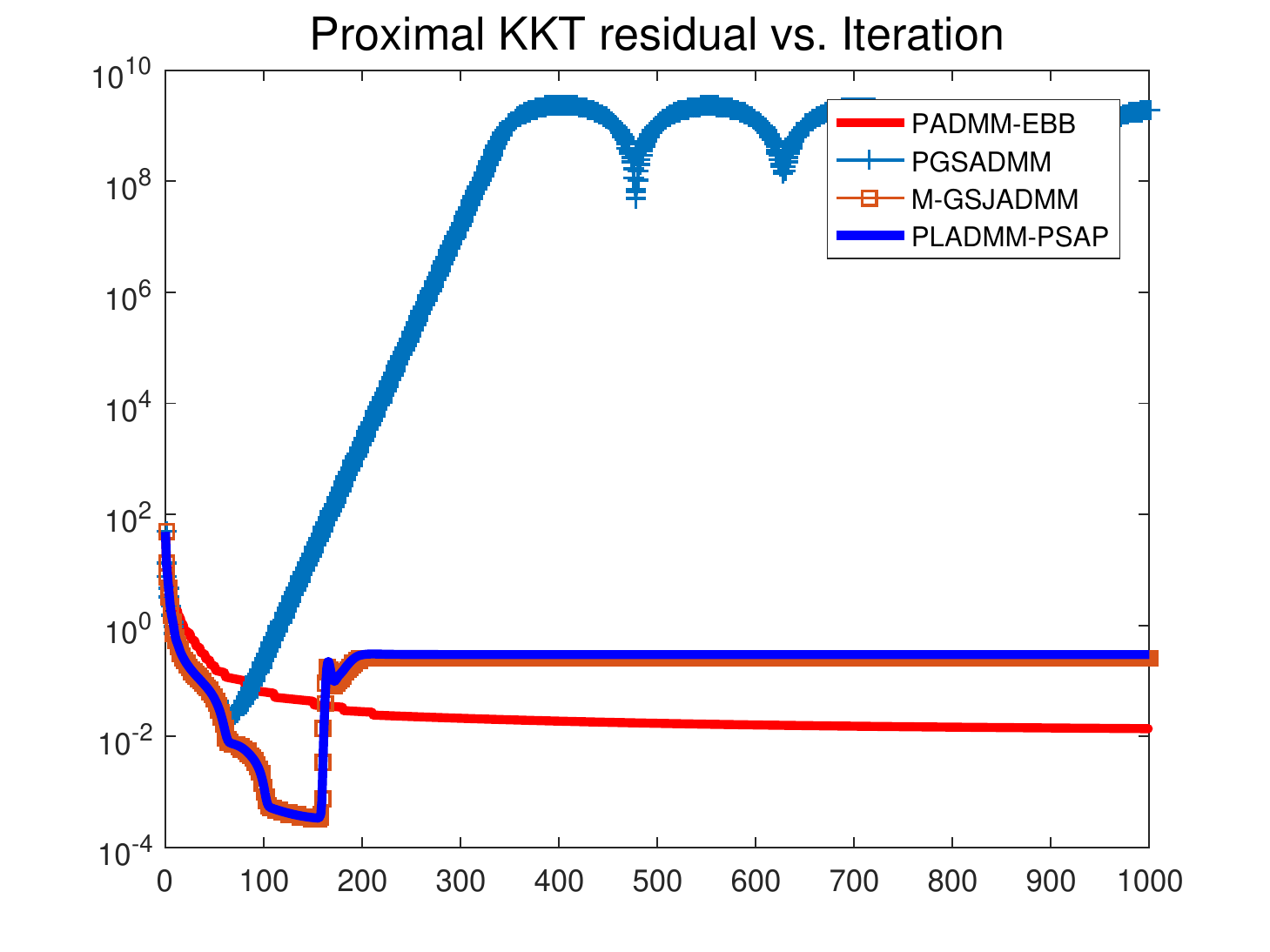}}
\subfigure{\label{fig:pic1}
\includegraphics[width=0.24\linewidth]{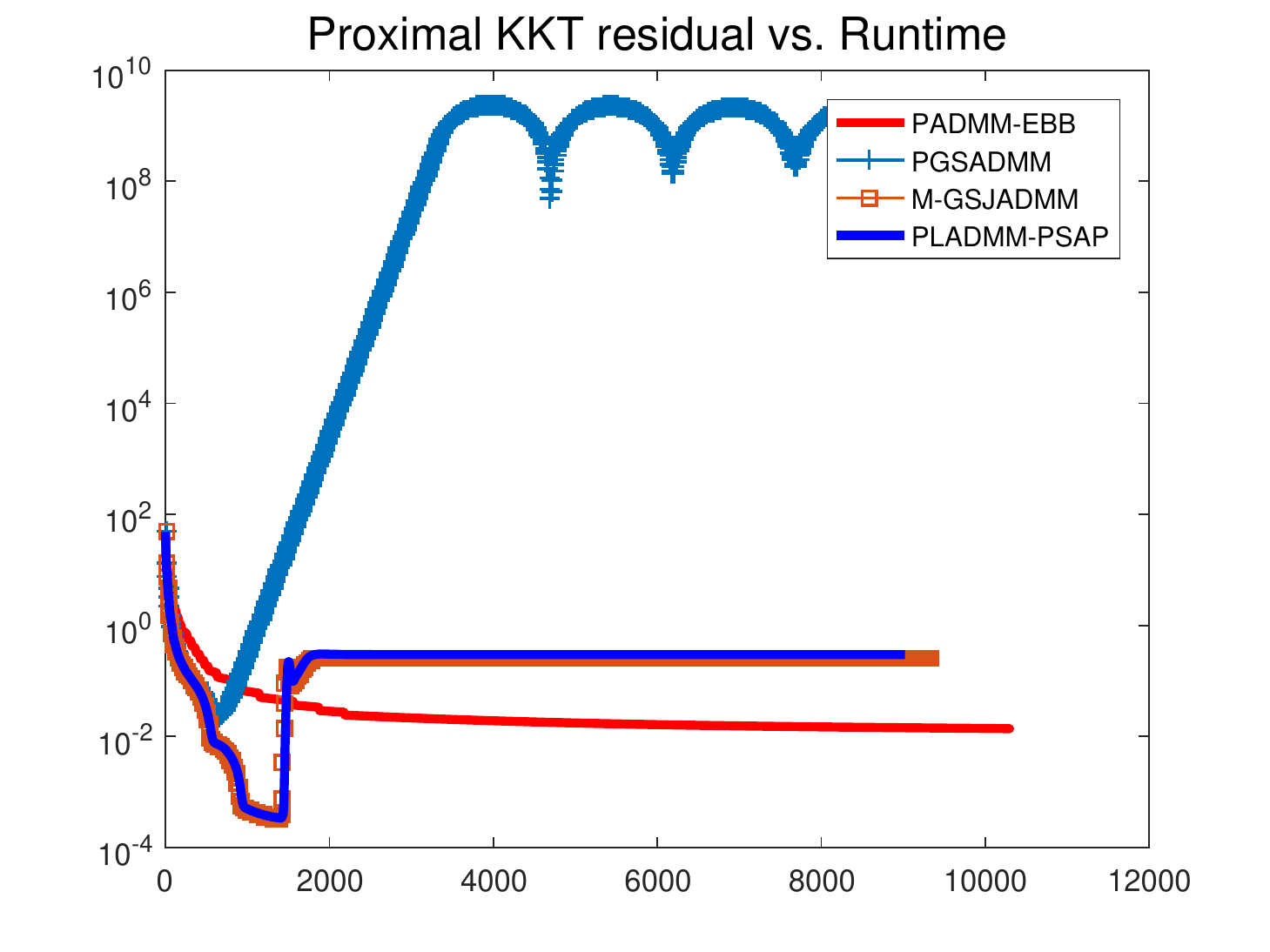}}
\subfigure{\label{fig:pic1}
\includegraphics[width=0.24\linewidth]{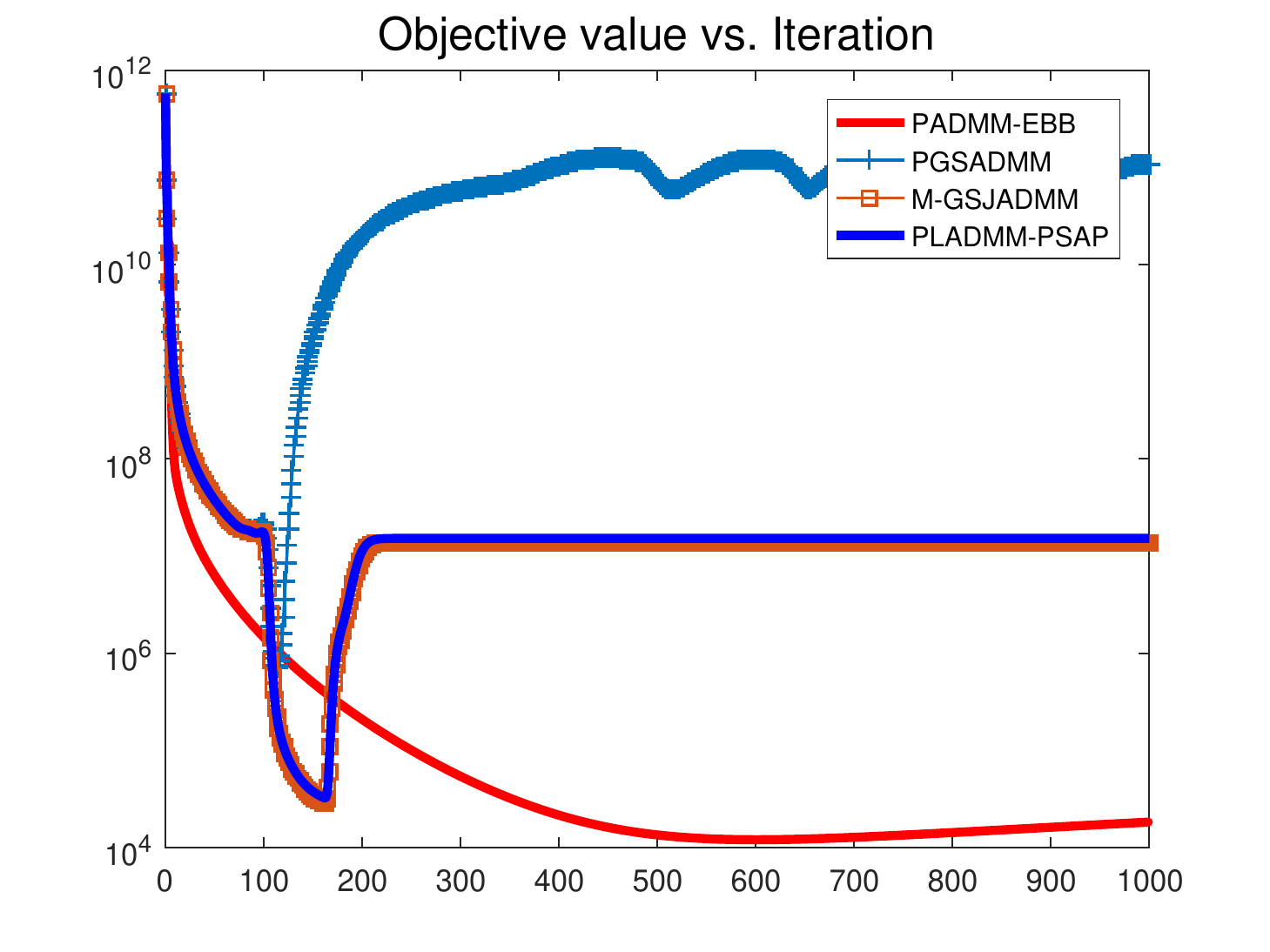}}
\subfigure{\label{fig:pic1}
\includegraphics[width=0.24\linewidth]{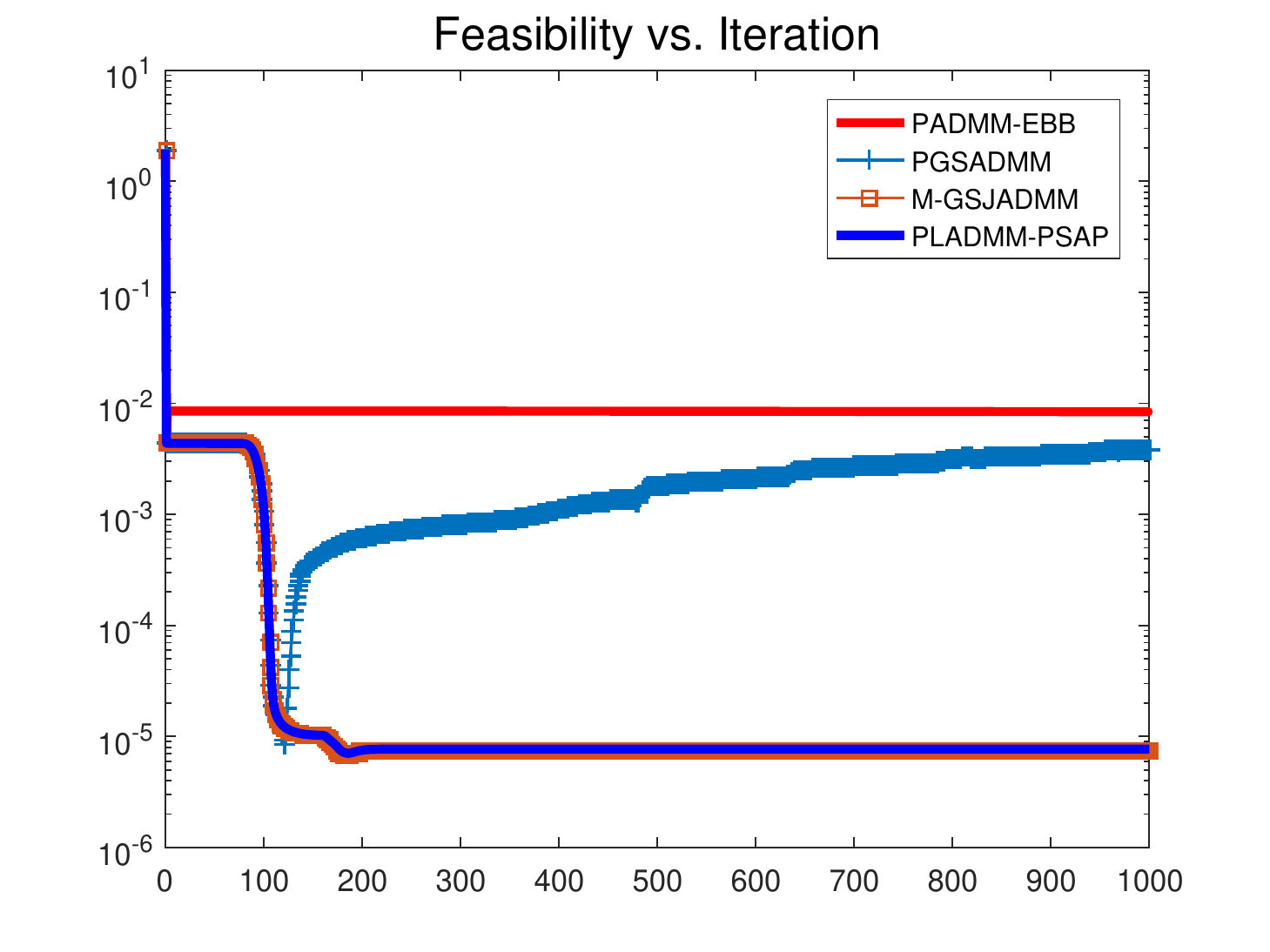}}
\vspace{-0.4cm}
 \caption{The above four figures illustrate the proximal KKT residual vs. iteration,
  proximal KKT residual vs. runtime, objective value vs. iteration, and
  feasibility vs. iteration on the real dataset PIE\_pose27 with parameters $(\lambda,\mu, \gamma)=(10^3,10^4,10^4)$, respectively.}
\label{fig:overlap-2}
\vspace{-0.5cm}
\end{figure*}

\subsection{Experiments}\label{experiments}

 We verify the efficacy of the proposed PADMM-EBB algorithm by solving the nonnegative dual graph regularized low-rank representation problem \cite{yin2015dual} as below:
 \begin{align}\label{DGR-LRR}
 &\min \|Z\|_{*}+\|G\|_{*}+\lambda\|E\|_{1}+\frac{\mu}{2}\|Z\|^2_{L_{Z}}+\frac{\gamma}{2}\|G\|^2_{L_{G}} \nonumber\\
 &~~{\rm s.t.}\  X=XZ+GX+E, Z\ge 0,\ G\ge 0,
 \end{align}
  where $(X, L_{Z}, L_{G})$ are given parameters and $(\lambda,\mu,\gamma)$ are the parameters to control the level of the reconstruction error and graph regularization.
  It is obvious that problem \eqref{DGR-LRR} can be formulated as problem \eqref{linearly-constriant} with $f$ being quadratic and $p = 3$.
  Define the proximal KKT residual of problem \eqref{linearly-constriant} as
  \begin{equation}\label{KKT-residual}
  R(z) =
  \left[
   \begin{array}{c}
        x_1 - {\rm Prox}_{g_{1}}\big(x_{1}- \nabla f_{1}(x)-\mathcal{A}_{1}y\big)   \\
                   \vdots                                             \\
        x_p - {\rm Prox}_{g_{p}}\big(x_{p}- \nabla f_{p}(x) -  \mathcal{A}_{p}y\big)\\
        b  - \sum_{i=1}^{p}\mathcal{A}_{i}^*x_{i}
   \end{array}
   \right].
\end{equation}
 The proximal KKT residual, as a complete characterization of optimality for constrained optimization,
 simultaneously evaluates the performance in terms of the feasibilities of primal-dual equalities,
 violation of nonnegativity, and complementarity condition of nonnegativity for problem \eqref{DGR-LRR}.

  We compare PADMM-EBB with three existing state-of-the-art primal-dual algorithms which are suitable for problem \eqref{linearly-constriant},
  namely PLADMM-PSAP \cite{liu2013linearized,lin2015linearized}, PGSADMM and M-GSJADMM \cite{lu2017unified} in terms of the objective value, feasibility, and proximal KKT residual $R(z)$ over iteration and runtime.
  Notably, PGSADMM and PADMM-EBB are performed with a full Gauss-Seidel updating for the majorized augmented Lagrange function \eqref{majorized-ALF}.
  We conduct experiments on a synthetic dataset $X={\rm randn}(200,200)$ and a real dataset PIE\_pose27\footnote{http://dengcai.zjulearning.org:8081/Data/FaceDataPIE.html}.
 Graph matrices $(L_{Z},L_{G})$ and parameters $(\lambda,\mu,\gamma)=(10^3, 10^4, 10^4)$ are directly borrowed from \cite{yin2015dual}.
 In the implementation, we strictly follow the advice in \cite{lin2015linearized,lu2017unified}
 to adaptively tune the penalty parameter $\beta_{k}$ for PLADMM-PSAP, PGSADMM and M-GSJADMM.

 According to Figures \ref{fig:overlap-1} and \ref{fig:overlap-2}, we know that PADMM-EBB is slightly better than PLADMM-PSAP, PGSADMM and M-GSJADMM in terms of the proximal KKT residual
 and the objective value due to the efficient block Barzilai-Borwein technique, which exploits the curvature information of the KKT generalized equation \eqref{inclusion-linearly-constriant} and the Gauss-Seidel updating for primal variables. PGSADMM, PLADMM-PSAP and M-GSJADMM have lower feasibilities since their penalty parameters $\beta_{k}$ are increasing as iterations proceed to force the equality constraint to hold.
 More experimental results are placed into the supplementary material.

\vspace{-8pt}
\section{Conclusions}

In this paper, we proposed a novel algorithmic framework of {\bf V}ariable {\bf M}etric {\bf O}ver-{\bf R}elaxed  {\bf H}ybrid {\bf P}roximal {\bf E}xtra-gradient (VMOR-HPE) method
and established its global convergence, iteration complexities, and local linear convergence rate.
This framework covers a large class of primal and primal-dual algorithms as special cases, and serves
as a powerful analysis technique for characterizing their convergences.
In addition, we applied the VMOR-HPE framework to linear equality constrained optimization, yielding
a new convergent primal-dual algorithm. The numerical experiments on synthetic and real datasets
demonstrate the efficacy of the proposed algorithm.

\newpage
\bibliography{OR_VM_HPE_ref}
\bibliographystyle{icml2018}

\newpage
\onecolumn
\icmltitle{Supplementary Material for\\ `` An Algorithmic Framework of Variable Metric Over-Relaxed\\ Hybrid Proximal Extra-Gradient Method "}

\appendix
\section{Proof of Theorem \ref{VMOR-HPE-convergence}}
\begin{theorem*}
 Let $\big\{(x^{k},y^{k})\big\}$ be the sequence generated by the VMOR-HPE framework. \\
 {\bf (i)}\ For any given $x^* \in T^{-1}(0)$, the following approximation contractive sequence of $\big\|x^{k}-x^*\big\|_{\mathcal{M}_{k}}^2$ holds
 \begin{align}\label{contractive-sequence}
 \big\|x^{k+1}-x^*\big\|_{\mathcal{M}_{k+1}}^2
 \le (1+\xi_{k})\big\|x^{k}-x^*\big\|_{\mathcal{M}_{k}}^2-(1-\sigma)(1+\xi_{k})(1+\theta_k)\big\|x^{k}-y^{k}\big\|_{\mathcal{M}_{k}}^2.
 \end{align}
 {\bf (ii)}\ $\{x^{k}\}$ and $\{y^{k}\}$ both converge to a point $x^{\infty}$ belonging to $T^{-1}(0)$.
 \end{theorem*}
\begin{proof}
 {\bf (i)} Notice that $v^{k} \in T^{[\epsilon_{k}]}(y^{k})$ and $x^*\in T^{-1}(0)$.
 By utilizing the definition of $T^{[\epsilon_{k}]}$, it holds that $\big\langle v^k,y^k-x^*\big\rangle\ge -\epsilon_k$.
 In combination with this inequality and $x^{k+1} = x^k -(1+\theta_k)c_k\mathcal{M}_{k}^{-1} v^{k}$, we obtain that
 \begin{align}\label{convergence-eq1}
\big\|x^{k+1}-x^*\big\|_{\mathcal{M}_{k}}^2
 &=\big\|x^{k}-x^*\big\|_{\mathcal{M}_{k}}^2\!+(1\!+\!\theta_{k})^2\big\| c_k\mathcal{M}_{k}^{-1}v^{k}\big\|_{\mathcal{M}_{k}}^2
  -2(1\!+\!\theta_{k})\big\langle c_kv^{k},x^{k}-x^*\big\rangle \\
 &= \big\|x^{k}\!-x^*\big\|_{\mathcal{M}_{k}}^2\!+(1\!+\!\theta_{k})^2\big\| c_k\mathcal{M}_{k}^{-1}v^{k}\big\|_{\mathcal{M}_{k}}^2\!-2(1\!+\!\theta_{k})\big\langle c_kv^{k},x^{k}\!-y^{k}\!+y^{k}\!-x^*\big\rangle\nonumber\\
 &= \big\|x^{k}\!-x^*\big\|_{\mathcal{M}_{k}}^2\!+\!(1\!+\!\theta_{k})^2\big\| c_k\mathcal{M}_{k}^{-1}v^{k}\big\|_{\mathcal{M}_{k}}^2\!-2(1\!+\!\theta_{k})\big\langle c_kv^{k},x^{k}\!-y^{k}\big\rangle
              \!-2(1\!+\!\theta_{k})c_k\big\langle v^{k},y^{k}\!-x^*\big\rangle\nonumber\\
 &\leq \big\|x^{k}-x\big\|_{\mathcal{M}_{k}}^2\!+(1\!+\!\theta_{k})^2\big\|c_k\mathcal{M}_{k}^{-1}v^{k}\big\|_{\mathcal{M}_{k}}^2
     \!-2(1\!+\!\theta_{k})\big\langle c_k\mathcal{M}_{k}^{-1}v^{k},\mathcal{M}_{k}(x^{k}-y^{k})\big\rangle\!+ 2(1\!+\!\theta_{k})c_k\epsilon_{k}\nonumber\\
 &=\big\|x^{k}-x\big\|_{\mathcal{M}_{k}}^2\!+\!(1\!+\!\theta_{k})\big[\theta_k\big\|c_k\mathcal{M}_{k}^{-1}v^{k}\big\|_{\mathcal{M}_{k}}^2
       \!+\!\big\| c_k\mathcal{M}_{k}^{-1}v^{k}+y^{k}\!-\!x^{k}\big\|_{\mathcal{M}_{k}}^2\!+\!2c_k\epsilon_{k}\!-\!\big\|y^{k}\!-\!x^{k}\big\|_{\mathcal{M}_{k}}^2\big]\nonumber\\
 &\leq \big\|x^{k}-x\big\|_{\mathcal{M}_{k}}^2-(1-\sigma)(1+\theta_{k})\big\|y^{k}-x^{k}\big\|_{\mathcal{M}_{k}}^2, \nonumber
 \end{align}
 where the last inequality holds according to \eqref{VMOR-HPE-b}. Moreover, according to $\mathcal{M}_{k+1}\preceq (1+\xi_{k})\mathcal{M}_{k}$,
 we obtain $\frac{1}{1+\xi_{k}}\big\|z^{k+1}-z^*\big\|_{\mathcal{M}_{k+1}}^2 \le \big\|z^{k+1}-z^*\big\|_{\mathcal{M}_{k}}^2$.
 Substituting this inequality into \eqref{convergence-eq1} yields the desired approximation contractive sequence
 \begin{align*}
 \big\|x^{k+1}-x^*\big\|_{\mathcal{M}_{k+1}}^2
 \le (1+\xi_{k})\big\|x^{k}-x^*\big\|_{\mathcal{M}_{k}}^2-(1-\sigma)(1+\xi_{k})(1+\theta_k)\big\|x^{k}-y^{k}\big\|_{\mathcal{M}_{k}}^2.
 \end{align*}

 \noindent
 {\bf (ii)}
 By the inequality \eqref{contractive-sequence}, $\theta_{k}\ge\underline{\theta}\ge-1$ and $\sigma<1$, we obtain
 $\big\|x^{k+1}-x^*\big\|_{\mathcal{M}_{k+1}}^2\le (1+\xi_{k})\big\|x^{k}-x^*\big\|_{\mathcal{M}_{k}}^2$ and
 \begin{align}\label{bound-xk}
 \big\|x^{k+1}-x^*\big\|_{\mathcal{M}_{k+1}}^2\le \prod_{i=1}^{k}(1+\xi_{i})\big\|x^{0}-x^*\big\|_{\mathcal{M}_{0}}^2.
 \end{align}
 In addition, for any $t\ge0$, it is easy to verify that $\log(1+t) \le t$. Hence, $\sum_{i=0}^{\infty}\xi_{i}< +\infty$ implies
 \begin{equation*}
 \Xi := \prod_{i=0}^{\infty}(1+\xi_{i})<\exp\Big(\sum_{i=0}^{\infty}\xi_{i}\Big)<+\infty.
 \end{equation*}
 Combing the above two inequalities implies $\big\|x^{k+1}-x^*\big\|_{\mathcal{M}_{k+1}}^2 \le \Xi\big\|x^{0}-x^*\big\|_{\mathcal{M}_{0}}^2$.
 This inequality, in combination with $\mathcal{M}_{k}\succeq \underline{\omega}\mathcal{I}$, implies the boundedness of sequence $\{x^{k}\}$.
 According to \eqref{contractive-sequence} again, we obtain
 \begin{align*}
  (1-\sigma)(1+\xi_{k})(1\!+\!\theta_k)\|x^{k}-y^{k}\|_{\mathcal{M}_{k}}^2
 &\le  (1\!+\!\xi_{k})\|x^{k}-x^*\|_{\mathcal{M}_{k}}^2-\|x^{k+1}-x^*\|_{\mathcal{M}_{k+1}}^2 \nonumber\\
 &\le  \|x^{k}-x^*\|_{\mathcal{M}_{k}}^2-\|x^{k+1}-x^*\|_{\mathcal{M}_{k+1}}^2+ \xi_{k}\Xi\|x^{0}-x^*\|_{\mathcal{M}_{0}}^2.
 \end{align*}
 Using $ \theta_k\ge \underline{\theta}>-1, \sigma <1$ and taking a summation of both sides of the above inequality, we obtain
 \begin{align}\label{convergence-eq1-1}
 (1-\sigma)(1+\underline{\theta})\sum_{i=1}^{k}\big\|x^{i}-y^{i}\big\|_{\mathcal{M}_{i}}^2
 &\le \sum_{i=1}^{k}(1-\sigma)(1+\xi_{i})(1+\theta_i)\big\|x^{i}-y^{i}\big\|_{\mathcal{M}_{i}}^2 \nonumber \\
 &\le \big\|x^{1}-x^*\big\|_{\mathcal{M}_{1}}^2-\big\|x^{k+1}-x^*\big\|_{\mathcal{M}_{k+1}}^2+\sum_{i=1}^{k}\xi_{i}\Xi\big\|x^{0}-x^*\big\|_{\mathcal{M}_{0}}^2\nonumber\\
 &\le \big(1+\sum_{i=1}^{k}\xi_{i}\big)\Xi\big\|x^{0}-x^*\big\|_{\mathcal{M}_{0}}^2.
 \end{align}
 Dividing the term $(1-\sigma)(1+\underline{\theta})$ on both sides of the above inequality, we obtain
 \begin{align}\label{convergence-eq2}
 \sum_{i=1}^{k}\big\|x^{i}-y^{i}\big\|_{\mathcal{M}_{i}}^2
 \le \frac{\big(1+\sum_{i=1}^{k}\xi_{i}\big)\Xi}{(1-\sigma)(1+\underline{\theta})}\big\|x^{0}-x^*\big\|_{\mathcal{M}_{0}}^2.
 \end{align}
 According to $\sum_{i=1}^{\infty}\xi_{i}<\infty$, $\mathcal{M}_{k}\succeq \underline{\omega}\mathcal{I}$, the boundedness of $\{x^{k}\}$ and inequality \eqref{convergence-eq2},
 sequence $\{y^{k}\}$ is apparently bounded and has the same limitation points as sequence $\{x^{k}\}$.
 To show the convergences of $\{x^{k}\}$ and $\{y^{k}\}$, we further need to argue that the accumulated residuals
 $\sum_{i=1}^{k}\|\mathcal{M}^{-1}_{i}v^{i}\|^2_{\mathcal{M}_{i}}$ and the accumulated error $\sum_{i=1}^{k}\epsilon_{i}$ are bounded.
 Expanding the term $\big\|c_k\mathcal{M}_{k}^{-1}v^{k}+y^{k}-x^{k}\big\|_{\mathcal{M}_{k}}^2$ in \eqref{VMOR-HPE-b}, we acquire
 $2\langle c_kv^{k},x^{k}-y^{k}\rangle \ge (1+\theta_k)\big\|c_k\mathcal{M}_{k}^{-1}v^{k}\big\|_{\mathcal{M}_{k}}^2+(1-\sigma)\big\|y^{k}-x^{k}\big\|_{\mathcal{M}_{k}}^2+2c_{k}\epsilon_{k}$.
 In addition, by the Cauchy-Schwartz inequality, it holds that
 \[
 2\langle c_kv^{k},x^{k}-y^{k}\rangle
 \!\le 2\big\|c_k\mathcal{M}_{k}^{-1}v^{k}\big\|_{\mathcal{M}_{k}}\big\|x^{k}-y^{k}\big\|_{\mathcal{M}_{k}}
 \!\le\! \frac{1\!+\theta_{k}}{2}\big\|c_k\mathcal{M}_{k}^{-1}v^{k}\big\|^2_{\mathcal{M}_{k}}\!+ \frac{2}{1\!+\!\theta_{k}}\big\|x^{k}-y^{k}\big\|^2_{\mathcal{M}_{k}}.
 \]
 Substituting the inequality into the above inequality, we obtain
 \begin{equation}\label{convergence-eq3}
 (1+\theta_k)\big\|c_k\mathcal{M}_{k}^{-1}v^{k}\big\|_{\mathcal{M}_{k}}^2+2c_{k}\epsilon_{k}-\frac{1+\theta_{k}}{2}\big\|c_k\mathcal{M}_{k}^{-1}v^{k}\big\|^2_{\mathcal{M}_{k}}
 -\frac{2}{1+\theta_{k}}\big\|x^{k}\!-\!y^{k}\big\|_{\mathcal{M}_{k}}^2 \le0,
 \end{equation}
 which further indicates $\frac{1+\theta_k}{2}\big\|c_k\mathcal{M}_{k}^{-1}v^{k}\big\|_{\mathcal{M}_{k}}^2+2c_{k}\epsilon_{k}\le \frac{2}{1+\theta_{k}}\big\|x^{k}\!-\!y^{k}\big\|_{\mathcal{M}_{k}}^2$. Hence, we have
 \begin{equation}\label{convergence-eq5}
 \big\|c_k\mathcal{M}_{k}^{-1}v^{k}\big\|_{\mathcal{M}_{k}}^2\le \frac{4}{(1+\theta_{k})^2}\big\|x^{k}-y^{k}\big\|_{\mathcal{M}_{k}}^2,\
 c_{k}\epsilon_{k}\le \frac{1}{1+\theta_{k}}\big\|x^{k}\!-\!y^{k}\big\|_{\mathcal{M}_{k}}^2.
 \end{equation}
 Combining \eqref{convergence-eq2} and \eqref{convergence-eq5} yields the bounds of ${\sum_{i=1}^{k}}(1\!+\!\theta_{i})^2\big\|c_{i}\mathcal{M}_{i}^{-1}v^{i}\big\|^2_{\mathcal{M}_{i}}$
 and ${\sum_{i=1}^{k}}(1\!+\!\theta_{i})c_{i}\epsilon_{i}$, which are
 \begin{gather}
 {\sum_{i=1}^{k}}(1+\theta_{i})^2\big\|c_{i}\mathcal{M}_{i}^{-1}v^{i}\big\|^2_{\mathcal{M}_{i}}
 \le \frac{4(1+\sum_{i=1}^{k}\xi_{i}) \Xi}{(1-\sigma)(1+\underline{\theta})}
                                        \big \|x^{0}-x^*\big\|_{\mathcal{M}_{0}}^2, \label{convergence-eq5-1} \\
 {\sum_{i=1}^{k}}(1+\theta_{i})c_{i}\epsilon_{i}
   \le \frac{(1+\sum_{i=1}^{k}\xi_{i})\Xi}{(1-\sigma)(1+\underline{\theta})}\big\|x^{0}-x^*\big\|_{\mathcal{M}_{0}}^2. \label{convergence-eq5-2}
 \end{gather}
 By $\theta_{k}\ge\underline{\theta}$ and $c_{k}\ge \underline{c} >0$,  the upper estimations for $\sum_{i=1}^{k}\big\|\mathcal{M}^{-1}_{i}v^{i}\big\|^2_{\mathcal{M}_{i}}$ and $\sum_{i=1}^{k}\epsilon_{i}$ are given below:
  \begin{align}\label{convergence-eq6}
 {\sum_{i=1}^{k}}\big\|\mathcal{M}_{i}^{-1}v^{i}\big\|^2_{\mathcal{M}_{i}}
 \le \frac{4(1\!+\!\sum_{i=1}^{k}\xi_{i}) \Xi}{(1\!-\!\sigma)\underline{c}^2(1\!+\!\underline{\theta})^{3}}
                                         \big\|x^{0}\!-\!x^*\big\|_{\mathcal{M}_{0}}^2,\
  {\sum_{i=1}^{k}}\epsilon_{i}
   \le \frac{(1\!+\!\sum_{i=1}^{k}\xi_{i})\Xi}{\underline{c}(1\!-\!\sigma)(1\!+\!\underline{\theta})^2}\big\|x^{0}\!-\!x^*\big\|_{\mathcal{M}_{0}}^2.
 \end{align}
 By \eqref{convergence-eq2}, \eqref{convergence-eq6} and $\mathcal{M}_{k}\succeq \underline{\omega}\mathcal{I}$,
 it holds that $\lim_{k\to \infty} \epsilon_{k} = \lim_{k\to \infty} \|v^{k}\| =\lim_{k\to \infty} \|x^{k}-y^{k}\| = 0.$
 In addition, due to the boundedness of $\{x^{k}\}$ and $\{y^{k}\}$, there exists a subsequence $\mathcal{K} \subseteq \{1,2,\ldots\}$ such that
 $\lim_{k\in \mathcal{K}, k\to \infty} x^{k} = \lim_{k\in \mathcal{K}, k\to \infty} y^{k} = x^{\infty}$.
 Let $k\in \mathcal{K}$ tend to be infinity in $v^{k}\in T^{[\epsilon_{k}]}(y^{k})$ in \eqref{VMOR-HPE-a}, and then
 it holds that $0 \in T(x^{\infty})$ by verifying the definition of enlargement operator $T^{[\epsilon_{k}]}$. Hence, $x^{\infty}$ is a root of inclusion problem \eqref{inclusion-problem}.
 Replacing $x^*$ by $x^{\infty}$ in inequality \eqref{contractive-sequence}, we derive
 \begin{align*}
 \big\|x^{k+1}-x^{\infty}\big\|_{\mathcal{M}_{k+1}}^2
 \le (1+\xi_{k})\big\|x^{k}-x^{\infty}\big\|_{\mathcal{M}_{k}}^2-(1+\xi_{k})(1-\sigma)(1+\theta_k)\big\|x^{k}-y^{k}\big\|_{\mathcal{M}_{k}}^2.
 \end{align*}
 Notice that $\lim_{k\in \mathcal{K}, k\to \infty} x^{k}=x^{\infty}$. Therefore, for any given $\epsilon >0$, there exists $\overline{k} \in \mathcal{K}> 0$ such that
 $\|x^{\overline{k}}-x^{\infty}\|_{\mathcal{M}_{\overline{k}}}^2 \le \frac{\epsilon}{\Xi}$.
 Then, for all $k\ge \overline{k}$, the above inequality indicates
 \begin{align*}
 \|x^{k+1}-x^{\infty}\|_{\mathcal{M}_{k+1}}^2
 \le \prod_{i=\overline{k}}^{k}(1+\xi_{i})\|x^{\overline{k}}-x^{\infty}\|_{\mathcal{M}_{\overline{k}}}^2
 \le \prod_{i=0}^{k}(1+\xi_{i})\frac{\epsilon}{\Xi} \le \epsilon.
 \end{align*}
 Hence, it holds that $\lim_{k\to \infty} x^{k}=\lim_{k\to \infty} y^{k}=x^{\infty}$ by $\mathcal{M}_{k}\succeq \underline{w}\mathcal{I}$. We complete the proof.
 \end{proof}

\section{Proof of Theorem \ref{linear-rate}}
 \begin{theorem*}
 Let $\{(x^{k},\,y^{k})\}$ be the sequence generated by the VMOR-HPE framework. Assume that
 the metric subregularity of $T$ at $(x^{\infty},0)\in {\rm gph}\,T$ holds with $\kappa >0$. Then, there exists
 $\overline{k} >0$ such that for all $k\ge \overline{k}$,
 \begin{align}\label{linear-rate-xk}
  {\rm dist}^{2}_{\mathcal{M}_{k\!+\!1}}\big(x^{k+1},T^{-1}(0)\big)
  \leq \Big(1\!-\!\frac{\varrho_{k}}{2}\Big){\rm dist}^{2}_{\mathcal{M}_{k}}\big(x^{k},T^{-1}(0)\big),
 \end{align}
 where $\varrho_{k}= \left[(1-\sigma)(1+\theta_{k})\right]{\Big /}
            \left[\Big(1+\frac{\kappa}{\underline{c}}\sqrt{\frac{\Xi\overline{\omega}}{\underline{\omega}}}\Big)^{2}\Big(1+\sqrt{\sigma+\frac{4\max\{-\theta_{k},0\}}{(1+\theta_{k})^2}}\Big)^{2}\right]\in (0,1)$.
\end{theorem*}
\begin{proof}
Let $x^{\infty}$ be the limitation point of $\{x^{k}\}$ and $z^k$ be the point satisfying $0\in c_kT(z^k)+\mathcal{M}_{k}(z^k-x^{k})$, respectively.
By the metric subregularity of $T$ at $(x^{\infty},0) \in {\rm gph}\,T$, there exists $\widetilde{k}\in\mathbb{N}$ such that for all $k\geq\widetilde{k}$,
\begin{align}\label{linear-rate-eq1}
  {\rm dist}_{\mathcal{M}_{k}}\big(z^{k}, T^{-1}(0)\big)
  &\leq \sqrt{\Xi\overline{\omega}}{\rm dist}\big(z^{k}, T^{-1}(0)\big)
  \leq \sqrt{\Xi\overline{\omega}}\kappa{\rm dist}\big(0,T(z^{k})\big) \nonumber\\
  &\leq \frac{\sqrt{\Xi\overline{\omega}}\kappa}{\underline{c}}\big\|\mathcal{M}_{k}(z^{k}-x^{k})\big\|
  \leq \frac{\kappa}{\underline{c}}\sqrt{\frac{\Xi\overline{\omega}}{\underline{\omega}}}\big\|z^{k}-x^{k}\big\|_{\mathcal{M}_{k}},
 \end{align}
 where the third inequality holds due to $-c^{-1}_{k}\mathcal{M}_{k}(z^{k}-x^{k})\in T(z^{k})$ and $c_{k}\ge \underline{c}$,
 and the last inequality holds due to $\big\|\mathcal{M}^{\frac{1}{2}}_{k}(z^{k}\!-\!x^{k})\big\|\!\ge\! \lambda_{\min}(\mathcal{M}^{\frac{1}{2}}_{k})\big\|z^{k}-x^{k}\big\|$.
 By the triangle inequality, inequality \eqref{linear-rate-eq1} indicates
 \begin{align}\label{linear-rate-eq2}
 {\rm dist}_{\mathcal{M}_{k}}\big(x^{k},T^{-1}(0)\big)
  \le \big\|x^{k}-z^{k}\big\|_{\mathcal{M}_{k}} + {\rm dist}_{\mathcal{M}_{k}}\big(z^{k},T^{-1}(0)\big) \leq \Big(1+\frac{\kappa}{\underline{c}}\sqrt{\frac{\Xi\overline{\omega}}{\underline{\omega}}}\Big)\big\|z^{k}-x^{k}\big\|_{\mathcal{M}_{k}}.
 \end{align}
 Next, we build the connection between $\|z^{k}-x^{k}\|_{\mathcal{M}_{k}}$ and $\|y^{k}-x^{k}\|_{\mathcal{M}_{k}}$, which is crucial
 for establishing the linear convergence rate \eqref{linear-rate-xk}.
 Due to inequality \eqref{VMOR-HPE-a}, $0 \in c_{k}T(z^{k})+\mathcal{M}_{k}(z^{k}-x^{k})$ and the definition of $T^{[\epsilon_{k}]}$, we obtain
 $\big\langle c_kv^{k}\!-\!\mathcal{M}_{k}(x^{k}\!-\!z^{k}),  y^{k}\!-\!z^{k}\big\rangle \!\geq -c_k\epsilon_{k}$ .
 Let $r^{k}:= c_{k}\mathcal{M}_{k}^{-1}v^{k}+y^{k}-x^{k}$, and then it holds that $c_kv^{k}\!=\!\mathcal{M}_{k}r^{k}\!+\!\mathcal{M}_{k}(x^{k}-y^{k})$.
 Substituting this equality into the last inequality yields
 \[
   \|z^{k}-y^{k}\|_{\mathcal{M}_{k}}^2-\|r^{k}\|_{\mathcal{M}_{k}}\|z^{k}-y^{k}\|_{\mathcal{M}_{k}}-c_{k}\epsilon_{k}\leq 0.
 \]
 The above quadratic inequality on the term $\big\|z^{k}-y^{k}\big\|_{\mathcal{M}_{k}}$ directly implies the following result that
 \begin{equation}\label{linear-rate-eq3}
 \big\|z^{k}-y^{k}\big\|_{\mathcal{M}_{k}}
 \le \frac{1}{2}\Big[\big\|r^k\big\|_{\mathcal{M}_{k}}+\sqrt{\big\|r^{k}\big\|_{\mathcal{M}_{k}}^2+4c_{k}\epsilon_{k}}\Big]
 \leq \sqrt{\big\|r^{k}\big\|_{\mathcal{M}_{k}}^2+2c_{k}\epsilon_{k}}.
 \end{equation}
 Moreover, arranging the terms in \eqref{VMOR-HPE-b}, and then using notations $r^{k}$ and inequality \eqref{convergence-eq5}, we have
 \begin{align*}
 \big\|r^{k}\big\|_{\mathcal{M}_{k}}^2+2c_{k}\epsilon_{k}
 \le \sigma\big\|x^k-y^k\big\|_{\mathcal{M}_{k}}^2-\theta_k\big\|c_k\mathcal{M}_{k}^{-1}v^k\big\|_{\mathcal{M}_{k}}^2 \le \big(\sigma+\max\{-\theta_{k},0\}/(1+\theta_k)^2\big)\big\|x^k-y^k\big\|_{\mathcal{M}_{k}}^2.
 \end{align*}
 Substituting this inequality into \eqref{linear-rate-eq3} and using the triangle inequality, we further obtain
 \[
 \big\|z^{k}-x^{k}\big\|_{\mathcal{M}_{k}}
 \leq \big\|z^{k}-y^{k}\big\|_{\mathcal{M}_{k}}+\big\|y^{k}-x^{k}\big\|_{\mathcal{M}_{k}}
  \le\Big(1+\!\sqrt{\sigma+\frac{4\max\{-\theta_{k},0\}}{(1+\theta_k)^2}}\Big)\big\|x^{k}-y^{k}\big\|_{\mathcal{M}_{k}}.
 \]
Substituting this inequality into inequality \eqref{linear-rate-eq2}, for all $k\geq\widetilde{k}$ it holds that
 \begin{align}\label{linear-rate-eq4}
 {\rm dist}_{\mathcal{M}_{k}}\big(x^{k},T^{-1}(0)\big)
  &\le \Big(1+\frac{\kappa}{\underline{c}}\sqrt{\frac{\Xi\overline{\omega}}{\underline{\omega}}}\Big)\Big(1+\!\sqrt{\sigma+\frac{4\max\{-\theta_{k},0\}}{(1+\theta_k)^2}}\Big)\|x^{k}-y^{k}\|_{\mathcal{M}_{k}}  \nonumber\\
  &\le  \Big(1+\frac{\kappa}{\underline{c}}\sqrt{\frac{\Xi\overline{\omega}}{\underline{\omega}}}\Big)\Big(1+\!\sqrt{\sigma+\frac{4\max\{-\theta_{k},0\}}{(1+\theta_{k})^2}}\Big)\|x^{k}-y^{k}\|_{\mathcal{M}_{k}}.
 \end{align}
 According to \eqref{contractive-sequence} in Theorem \ref{VMOR-HPE-convergence}, for all $k\in\mathbb{N}$, it holds that
 \begin{align}\label{linear-rate-eq4.5}
   {\rm dist}_{\mathcal{M}_{k+1}}^2\big(x^{k+1},T^{-1}(0)\big)
  &=\big\|x^{k+1}-\Pi_{T^{-1}(0)}(x^{k+1})\big\|_{\mathcal{M}_{k+1}}^2
      \leq \big\|x^{k+1}-\Pi_{T^{-1}(0)}(x^{k})\big\|_{\mathcal{M}_{k+1}}^2\\
  &\leq (1+\xi_{k})\big\|x^{k}-\!\Pi_{T^{-1}(0)}(x^{k})\big\|_{\mathcal{M}_{k}}^2
   -(1+\!\xi_{k})(1-\sigma)(1+\!\theta_{k})\big\|x^{k}-\!y^{k}\big\|_{\mathcal{M}_{k}}^2\nonumber\\
  &=(1+\xi_{k}){\rm dist}_{\mathcal{M}_{k}}^2\big(x^{k},T^{-1}(0)\big)
      -(1+\xi_{k})(1-\sigma)(1+\theta_{k})\big\|x^{k}-y^{k}\big\|_{\mathcal{M}_{k}}^2,
\end{align}
 where $\Pi_{T^{-1}(0)}(\cdot)=\arg\inf_{x \in T^{-1}(0)}\big\|\cdot-x\big\|_{\mathcal{M}_{k+1}}$, and the first equality and the first inequality hold
 due to the definition of ${\rm dis}_{\mathcal{M}_{k+1}}(\cdot,T^{-1}(0))$. Utilizing  inequalities \eqref{linear-rate-eq4} and \eqref{linear-rate-eq4.5}, we obtain
 \begin{equation}\label{linear-rate-eq5}
  {\rm dist}_{\mathcal{M}_{k+1}}^2\!\big(x^{k+1},T^{-1}(0) \big)
  \leq  (1+\xi_{k})(1-\varrho){\rm dist}_{\mathcal{M}_{k}}^2\big(x^{k},T^{-1}(0) \big),
 \end{equation}
 where $\varrho_{k}= [(1-\sigma)(1+\theta_{k})]{\Big/}
       \Big[\Big(1+\frac{\kappa}{\underline{c}}\sqrt{\frac{\Xi\overline{\omega}}{\underline{\omega}}}\Big)\Big(1+\!\sqrt{\sigma+\frac{4\max\{-\theta_{k},0\}}{(1+\theta_{k})^2}}\Big)\Big]^{2}\in (0,1)$.
In addition, recall $\sum_{k=1}^{\infty}\xi_{k}<\infty$. Hence, there exists
$\widehat{k}\in\mathbb{N}$ such that for all $k \ge \widehat{k}$, it holds that
$\xi_{k}\le \frac{\varrho_{k}}{2(1-\varrho_{k})}$, which means that $(1+\xi_{k})(1-\varrho_{k})\le 1-\frac{\varrho_{k}}{2} < 1$.
Substituting this inequality into \eqref{linear-rate-eq5} and setting $\overline{k}=\max\{\widetilde{k},\widehat{k}\}$,
we acquire the desired result \eqref{linear-rate-xk}. The proof is finished.
\end{proof}

\section{Proof of Theorem \ref{iteration-complexity-VMOR-HPE}}
 \begin{theorem*}
 Let $\{(x^{k},y^{k},v^{k})\}$ and $\{\epsilon_{k}\}$ be the sequences generated by the VMOR-HPE framework. \\
 {\bf (i)} There exists an integer $k_{0}\in \{1,2,\ldots,k\}$ such that $v^{k_{0}} \in T^{[\epsilon_{k_0}]}(y^{k_{0}})$
 with $v^{k_{0}}$ and $\epsilon_{k_0}\ge 0$ respectively satisfying
 \begin{gather}\label{pointwise-vk-epsilonk}
 \|v^{k_{0}}\|\le \sqrt{\frac{4(1+\sum_{i=1}^{k}\xi_{i})\Xi^2\overline{\omega}}{k(1-\sigma)(1+\underline{\theta})^3\underline{c}^2}}\|x^{0}-x^*\|_{\mathcal{M}_{0}},\
{\rm and\quad } \epsilon_{k_0} \le \frac{(1+\sum_{i=1}^{k}\xi_{i})\Xi}{k(1-\sigma)(1+\underline{\theta})^2\underline{c}}\|x^{0}-x^*\|_{\mathcal{M}_{0}}^2.
 \end{gather}
 {\bf (ii)} Let $\{\alpha_{k}\}$ be the nonnegative weight sequence satisfying $\sum_{i=1}^{k}\alpha_{i}>0$.
 Denote $\tau_{i}=(1+\theta_{i})c_{i}$, and
 \begin{align}\label{weight-sequence}
 \overline{y}^{k}=\frac{{\sum_{i=1}^{k}}\tau_{i}\alpha_{i}y^{i}}{{\sum_{i=1}^{k}}\tau_{i}\alpha_{i}}\
 \overline{v}^{k}=\frac{{\sum_{i=1}^{k}}\tau_{i}\alpha_{i}v^{i}}{{\sum_{i=1}^{k}}\tau_{i}\alpha_{i}},\
  \overline{\epsilon}_{k}=\frac{{\sum_{i=1}^{k}}\tau_{i}\alpha_{i}\big(\epsilon_{i}
                    +\langle y^{i}-\overline{y}^{k},v^{i}-\overline{v}^{k}\rangle\big)}{{\sum_{i=1}^{k}}\tau_{i}\alpha_{i}}.
 \end{align}
 Then, it holds that $\overline{v}^{k} \in T^{[\overline{\epsilon}_{k}]}(\overline{y}^{k})$ with $\overline{\epsilon}_{k}\ge 0$.
 Moreover, if $\mathcal{M}_{k}\le (1+\xi_{k})\mathcal{M}_{k+1}$, it holds that
 \begin{gather}\label{weighted-vk}
\|\overline{v}^{k}\|
 \le \frac{\max\limits_{ 1\le i\le k}\{\alpha_{i+1}\}\sum\limits_{i=1}^{k}\xi_{i}+\sum\limits_{i=1}^{k}\big|\alpha_{i}-\alpha_{i+1}\big|+\alpha_{k+1}+\alpha_{1}}
                 {\underline{c}(1+\underline{\theta})\sum_{i=1}^{k}\alpha_{i}}M, \\
 \overline{\epsilon}_{k}=\frac{(10+\underline{\theta})\max\limits_{1\le i\le k}\{\alpha_{i}\}\big(1+\sum\limits_{i=1}^{k}\xi_{i}\big)
                               +(2+\underline{\theta}){\sum\limits_{i=1}^{k}}\big|\alpha_{i+1}-\alpha_{i}\big|}
         {\underline{c}(1+\underline{\theta})^2\sum_{i=1}^{k}\alpha_i}B, \label{weighted-epsilonk}
 \end{gather}
where $M$ and $B$ are two constants that are respectively defined as $M = \Xi\overline{\omega}\left[\big\|x^*\big\|+\sqrt{\frac{\Xi}{\underline{\omega}}}\big\|x^{0}-x^*\big\|_{\mathcal{M}_{0}}\right]$ and
\[
B =\max\left\{ M,\  \Xi\big\|x^*\big\|^2+\frac{\Xi^2}{\underline{\omega}}\big\|x^{0}-x^*\big\|_{\mathcal{M}_{0}}^2,\
                       \frac{\Xi^2}{(1-\sigma)\underline{\omega}}\big\|x^{0}-x^*\big\|_{\mathcal{M}_{0}}^2,\  \frac{\Xi}{(1-\sigma)}\big\|x^{0}-x^*\big\|_{\mathcal{M}_{0}}^2 \right\}.
\]
\end{theorem*}
\begin{proof}
{\bf (i)}\ By \eqref{convergence-eq2}, there exists an integer $k_{0}\in \{1,2,\ldots,k\}$ such that the following inequality holds:
\begin{align}\label{iteration-complexity-VMOR-HPE-eq1}
\big\|x^{k_{0}}-y^{k_{0}}\big\|_{\mathcal{M}_{k_0}}^2
\le \frac{(1+\sum_{i=1}^{k}\xi_{i})\Xi}{k(1-\sigma)(1+\underline{\theta})}\big\|x^{0}-x^*\big\|_{\mathcal{M}_{0}}^2.
\end{align}
Combining this inequality with \eqref{convergence-eq5} and using $\underline{\omega}\mathcal{I}\preceq \mathcal{M}_{k+1} \preceq (1+\xi_{k})\mathcal{M}_{k}$, $c_{k}\ge \underline{c}$, we obtain
\begin{align*}
\|v^{k_{0}}\|\le \sqrt{\frac{4(1+\sum_{i=1}^{k}\xi_{i})\Xi^2\overline{\omega}}{k(1-\sigma)(1+\underline{\theta})^3\underline{c}^2}}\big\|x^{0}-x^*\big\|_{\mathcal{M}_{0}},\
 \epsilon_{k_0} \le \frac{(1+\sum_{i=1}^{k}\xi_{i})\Xi}{k(1-\sigma)(1+\underline{\theta})^2\underline{c}}\big\|x^{0}-x^*\big\|_{\mathcal{M}_{0}}^2.
\end{align*}
In addition, $v^{k_{0}}\in T^{[\epsilon_{k_{0}}]}(y^{k_{0}})$ holds directly due to \eqref{VMOR-HPE-a}. Hence, result {\bf (i)} has been established.

\noindent
{\bf (ii)} By \citep{Monteiro2010On}, it holds that $\overline{v}^{k}\!\in T^{[\overline{\epsilon}_{k}]}(\overline{y}^{k})$ and $\overline{\epsilon}^{k}\!\geq 0$. By \eqref{weight-sequence}, it holds that
 \begin{align}\label{iteration-complexity-VMOR-HPE-eq4}
  \|\overline{v}^{k}\|
  &=\frac{1}{\sum_{i=1}^{k}c_{i}\alpha_{i}(1+\theta_{i})}
            \big\|{\sum_{i=1}^{k}}c_{i}\alpha_{i}(1+\theta_{i})v^{i}\big\|
   =\frac{1}{\sum_{i=1}^{k}c_{i}\alpha_{i}(1+\theta_{i})}\big\|{\sum_{i=1}^{k}}\alpha_{i}\mathcal{M}_{i}(x^{i+1}-x^{i})\big\|\nonumber\\
  &= \frac{1}{\sum_{i=1}^{k}c_{i}\alpha_{i}(1+\theta_{i})}\big\|{\sum_{i=1}^{k}}\big(\alpha_{i+1}\mathcal{M}_{i+1}x^{i+1}-\alpha_{i}\mathcal{M}_{i}x^{i}\big)
     +{\sum_{i=1}^{k}}\big(\alpha_{i}\mathcal{M}_{i}-\alpha_{i+1}\mathcal{M}_{i+1}\big)x^{i+1}\big\|\nonumber\\
  &\le \frac{\big\|{\sum_{i=1}^{k}}\big(\alpha_{i+1}\mathcal{M}_{i+1}x^{i+1}-\alpha_{i}\mathcal{M}_{i}x^{i}\big)\big\|}{\sum_{i=1}^{k}c_{i}\alpha_{i}(1+\theta_{i})}
     +\frac{\big\|{\sum_{i=1}^{k}}\big(\alpha_{i}\mathcal{M}_{i}-\alpha_{i+1}\mathcal{M}_{i+1}\big)x^{i+1}\big\|}{\sum_{i=1}^{k}c_{i}\alpha_{i}(1+\theta_{i})}\nonumber\\
  &\le  \frac{\big\|\alpha_{k+1}\mathcal{M}_{k+1}x^{k+1}-\alpha_{1}\mathcal{M}_{1}x^{1}\big\|}{\sum_{i=1}^{k}c_{i}\alpha_{i}(1+\theta_{i})}
     +\frac{{\sum_{i=1}^{k}}\big\|\alpha_{i}\mathcal{M}_{i}-\alpha_{i+1}\mathcal{M}_{i+1}\big\|\max\limits_{1\le i \le k}\{\|x^{i+1}\|\}}{\sum_{i=1}^{k}c_{i}\alpha_{i}(1+\theta_{i})}\nonumber\\
  &\le  \frac{\alpha_{k+1}\big\|\mathcal{M}_{k+1}x^{k+1}\|+\alpha_{1}\big\|\mathcal{M}_{1}x^{1}\big\|}{\sum_{i=1}^{k}c_{i}\alpha_{i}(1+\theta_{i})}
     +\frac{{\sum_{i=1}^{k}}\big\|\alpha_{i}\mathcal{M}_{i}-\alpha_{i+1}\mathcal{M}_{i+1}\big\|\max\limits_{1\le i \le k}\{\|x^{i+1}\|\}}{\sum_{i=1}^{k}c_{i}\alpha_{i}(1+\theta_{i})}\nonumber\\
  &\le \frac{\alpha_{k+1}\big\|\mathcal{M}_{k+1}\big\|+\alpha_{1}\big\|\mathcal{M}_{1}\big\|+{\sum_{i=1}^{k}}\big\|\alpha_{i}\mathcal{M}_{i}-\alpha_{i+1}\mathcal{M}_{i+1}\big\|}
         {\sum_{i=1}^{k}c_{i}\alpha_{i}(1+\theta_{i})}\max_{1\le i \le k}\{\|x^{i+1}\|\},
 \end{align}
 where the first and the third inequalities hold by the Cauchy-Schwartz inequality.
 By using  $\mathcal{M}_{k}\le (1+\xi_{k})\mathcal{M}_{k+1}$ and $\mathcal{M}_{k+1}\le (1+\xi_{k})\mathcal{M}_{k}$, the following inequality holds that
 \begin{align*}
 &\quad \sum_{i=1}^{k}\|\alpha_{i}\mathcal{M}_{i}-\alpha_{i+1}\mathcal{M}_{i+1}\|  \nonumber\\
 &\le \sum_{i=1}^{k}|\alpha_{i}-\alpha_{i+1}|\max\{\|\mathcal{M}_{i+1}\|,\|\mathcal{M}_{i}\|\} +\sum_{i=1}^{k}\xi_{i}\max\{\alpha_{i+1}\|\mathcal{M}_{i+1}\|,\alpha_{i}\|\mathcal{M}_{i}\|\} \nonumber\\
 & \le \max\limits_{1\le i \le k}\{\|\mathcal{M}_{i+1}\|\}\sum_{i=1}^{k}|\alpha_{i}-\alpha_{i+1}| +\max\limits_{1\le i \le k}\{\alpha_{i+1}\|\mathcal{M}_{i+1}\|\}\sum_{i=1}^{k}\xi_{i}\nonumber\\
 & \le \max\limits_{1\le i \le k}\{\|\mathcal{M}_{i+1}\|\}\Big[\sum_{i=1}^{k}|\alpha_{i}-\alpha_{i+1}|+\max\limits_{1\le i \le k}\{\alpha_{i+1}\}\sum_{i=1}^{k}\xi_{i}\Big].
 \end{align*}
 Substituting this inequality into \eqref{iteration-complexity-VMOR-HPE-eq4} and using $\|\mathcal{M}_{k+1}\|\le \Xi\overline{\omega}$ and $c_{k}\ge \underline{c}>0$, we obtain
  \begin{align*}
 \|\overline{v}^{k}\|
 \le \frac{\max\limits_{1\le i \le k}\{\alpha_{i+1}\}\sum_{i=1}^{k}\xi_{i}+\sum_{i=1}^{k}|\alpha_{i}-\alpha_{i+1}|+\alpha_{k+1}+\alpha_{1}}{\underline{c}\sum_{i=1}^{k}\alpha_{i}(1+\theta_{i})}
             \max\limits_{1\le i \le k}\{\|x^{i+1}\|\}\Xi\overline{\omega}.
 \end{align*}
 By inequality \eqref{bound-xk}, we have $\big\|x^{k}\| \le \big\|x^*\big\|+\sqrt{\frac{\Xi}{\underline{\omega}}}\big\|x^{0}-x^*\big\|_{\mathcal{M}_{0}}$. By using the notation $M$ and $\theta_{k}\ge \underline{\theta}$, it holds that
 \begin{align*}
 \|\overline{v}^{k}\|
 \le \frac{\max\limits_{1\le i \le k}\{\alpha_{i+1}\}\sum_{i=1}^{k}\xi_{i}+\sum_{i=1}^{k}|\alpha_{i}-\alpha_{i+1}|+\alpha_{k+1}+\alpha_{1}}{\underline{c}(1+\underline{\theta})\sum_{i=1}^{k}\alpha_{i}}M.
 \end{align*}
 In the following, we estimate the upper bound for $\overline{\epsilon}_k$. By the definition of $\overline{\epsilon}_k$, we obtain
 \begin{align}\label{iteration-complexity-VMOR-HPE-eq9}
  \overline{\epsilon}_{k}
  &=\frac{{\sum_{i=1}^{k}}\alpha_{i}c_{i}(1\!+\!\theta_{i})\big({\epsilon}^{i}
     \!+\!\langle y^{i}\!-\!\overline{y}^{k},v^{i}\rangle\big)}{\sum_{i=1}^{k}c_{i}\alpha_i(1+\theta_{i})}
  \!=\!\frac{{\sum_{i=1}^{k}}\alpha_{i}c_{i}(1\!+\!\theta_{i}){\epsilon}^{i}}{{\sum_{i=1}^{k}}c_{i}\alpha_i(1+\theta_{i})}
  \!+\!\frac{{\sum_{i=1}^{k}}\alpha_{i}c_{i}(1\!+\!\theta_{i})\langle y^{i}\!-\!\overline{y}^{k},v^{i}\rangle}{\sum_{i=1}^{k}c_{i}\alpha_i(1+\theta_{i})}\nonumber\\
  &=\frac{{\sum_{i=1}^{k}}{\alpha}_{i}(1+\theta_{i})c_{i}{\epsilon}^{i} }{\sum_{i=1}^{k}c_{i}\alpha_i(1+\theta_{i})}
  +\frac{{\sum_{i=1}^{k}}\alpha_{i}c_{i}(1+\theta_{i})\langle x^{i}-\overline{y}^{k},v^{i}\rangle}{\sum_{i=1}^{k}c_{i}\alpha_i(1+\theta_{i})}
  +\frac{{\sum_{i=1}^{k}}\alpha_{i}c_{i}(1+\theta_{i})\langle y^{i}-x^{i},v^{i}\rangle}{\sum_{i=1}^{k}c_{i}\alpha_i(1+\theta_{i})}\nonumber\\
 &\le\frac{\max\limits_{1\le i \le k}\{\alpha_{i}\}{\sum_{i=1}^{k}}(1+\theta_{i})c_{i}{\epsilon}^{i}}{\sum_{i=1}^{k}c_{i}\alpha_i(1+\theta_{i})}
     +\frac{{\sum_{i=1}^{k}}\alpha_{i}c_{i}(1+\theta_{i})\langle x^{i}-\overline{y}^{k},v^{i}\rangle}{\sum_{i=1}^{k}c_{i}\alpha_i(1+\theta_{i})} \nonumber\\
 &\quad   +\frac{\max\limits_{1\le i \le k}\{\alpha_{i}\}{\sum_{i=1}^{k}}\big((1+\theta_{i})^2\|c_{i}\mathcal{M}_{i}^{-1}v^{i}\|_{\mathcal{M}_{i}}^2
       +\|y^{i}-x^{i}\|_{\mathcal{M}_{i}}^2\big)}{\sum_{i=1}^{k}c_{i}\alpha_i(1+\theta_{i})} \nonumber\\
 &\le  \frac{6\max\limits_{1\le i \le k}\{\alpha_{i}\}{\sum_{i=1}^{k}}\big\|y^{i}-x^{i}\big\|_{\mathcal{M}_{i}}^2}{\sum_{i=1}^{k}c_{i}\alpha_i(1+\theta_{i})}
 +\frac{{\sum_{i=1}^{k}}\alpha_{i}\tau_{i}\langle x^{i}-\overline{y}^{k},v^{i}\rangle}{{\sum_{i=1}^{k}}c_{i}\alpha_i(1+\theta_{i})},
\end{align}
where the first inequality holds according to the Cauchy-Schwartz inequality and the last inequality holds according to \eqref{convergence-eq5}. In addition,
$\|x^{i+1}-\overline{y}^{k}\|_{\mathcal{M}_{i}}^2 =\|x^{i}-\overline{y}^{k}\|_{\mathcal{M}_{i}}^2+\|\tau_{i}\mathcal{M}_{i}^{-1}v^{i}\|_{\mathcal{M}_{i}}^2 -2\langle\tau_{i}v^{i},x^{i}-\overline{y}^{k}\rangle$
holds by using  $x^{k+1}=x^{k}-(1+\theta_{k})c_{k}\mathcal{M}_{k}^{-1}v^{k}=x^{k}-\tau_{k}\mathcal{M}_{k}^{-1}v^{k}$.  Hence, we obtain
\begin{align*}
  ~2\alpha_{i}\langle\tau_{i}v^{i},x^{i}-\overline{y}^{k}\rangle
 &=\alpha_{i}\|\tau_{i}\mathcal{M}_{i}^{-1}v^{i}\|_{\mathcal{M}_{i}}^2\!+\!\alpha_{i}\|x^{i}-\overline{y}^{k}\|_{\mathcal{M}_{i}}^2
                  -\alpha_{i}\|x^{i+1}-\overline{y}^{k}\|_{\mathcal{M}_{i}}^2\nonumber\\
  &\le\alpha_{i}\|\tau_{i}\mathcal{M}_{i}^{-1}v^{i}\|_{\mathcal{M}_{i}}^2\!+\!\alpha_{i}\|x^{i}-\overline{y}^{k}\|_{\mathcal{M}_{i}}^2
                  -\frac{\alpha_{i}}{1+\xi_{i}}\|x^{i+1}-\overline{y}^{k}\|_{\mathcal{M}_{i+1}}^2\nonumber\\
  &\le\alpha_{i}\|\tau_{i}\mathcal{M}_{i}^{-1}v^{i}\|_{\mathcal{M}_{i}}^2\!+\!\alpha_{i}\|x^{i}-\overline{y}^{k}\|_{\mathcal{M}_{i}}^2
                  -\alpha_{i}\|x^{i+1}-\overline{y}^{k}\|_{\mathcal{M}_{i+1}}^2
                   +\alpha_{i}\xi_{i}\|x^{i+1}-\overline{y}^{k}\|_{\mathcal{M}_{i+1}}^2\nonumber\\
  &=\alpha_{i}\|\tau_{i}\mathcal{M}_{i}^{-1}v^{i}\|_{\mathcal{M}_{i}}^2\!+\!\alpha_{i}\|x^{i}-\overline{y}^{k}\|_{\mathcal{M}_{i}}^2
                  -\alpha_{i}\|x^{i+1}-\overline{y}^{k}\|_{\mathcal{M}_{i+1}}^2
                  +\alpha_{i}\xi_{i}\|x^{i+1}-\overline{y}^{k}\|_{\mathcal{M}_{i+1}}^2,
\end{align*}
where the first and the second inequalities hold due to $\mathcal{M}_{i+1}\preceq (1+\xi_{i})\mathcal{M}_{i}$ and $\frac{1}{1+\xi_{i}} \ge 1-\xi_{i}$, respectively.
Taking a summation on both sides of the above inequality, it holds that
 \begin{align}\label{iteration-complexity-VMOR-HPE-eq10}
  &\quad~ 2{\sum_{i=1}^{k}}\alpha_{i}\langle\tau_{i}v^{i},x^{i}-\overline{y}^{k}\rangle\\
  &\le\!{\sum_{i=1}^{k}}\!\alpha_{i}\|\tau_{i}\mathcal{M}_{i}^{-1}v^{i}\|_{\mathcal{M}_{i}}^2
    \!+\!\sum_{i=1}^{k}\!(\alpha_{i\!+\!1}\!-\!\alpha_{i})\|x^{i\!+\!1}\!-\!\overline{y}^{k}\|_{\mathcal{M}_{i\!+\!1}}^2\!+\!\alpha_{1}\|x^{1}\!-\!\overline{y}^{k}\|_{\mathcal{M}_{1}}^2
    \!+\!\sum_{i=1}^{k}\!\alpha_{i}\xi_{i}\|x^{i\!+\!1}\!-\!\overline{y}^{k}\|_{\mathcal{M}_{i\!+\!1}}^2 \nonumber\\
  &\le 4\max_{1\le i\le k}\{\alpha_{i}\}{\sum_{i=1}^{k}}\big\|y^{i}\!-\!x^{i}\big\|_{\mathcal{M}_{i}}^2
    +\max_{0\le i\le k}\{\|x^{i+1}\!-\!\overline{y}^{k}\|_{\mathcal{M}_{i+1}}^2\}\Big[{\sum_{i=1}^{k}}|\alpha_{i+1}\!-\!\alpha_{i}|\!+\!{\sum_{i=1}^{k}}\alpha_{i}\xi_{i}+\alpha_{1}\Big] \nonumber\\
  &\le 4\max_{1\le i\le k}\{\alpha_{i}\}{\sum_{i=1}^{k}}\big\|y^{i}\!-\!x^{i}\big\|_{\mathcal{M}_{i}}^2
    +\max_{0\le i\le k}\{\|x^{i+1}\!-\!\overline{y}^{k}\|_{\mathcal{M}_{i+1}}^2\}\Big[{\sum_{i=1}^{k}}|\alpha_{i+1}\!-\!\alpha_{i}|\!+\!\max_{1\le i\le k}\{\alpha_{i}\}({\sum_{i=1}^{k}}\xi_{i}+1)\Big] \nonumber,
 \end{align}
 where the last inequality holds according to \eqref{convergence-eq5}. This inequality combined with \eqref{iteration-complexity-VMOR-HPE-eq9} yields
 \begin{align}\label{iteration-complexity-VMOR-HPE-eq11}
  \overline{\epsilon}_{k}
 &\le  \frac{8\max\limits_{1\le i\le k}\{\alpha_{i}\}{\sum_{i=1}^{k}}\big\|y^{i}-x^{i}\big\|_{\mathcal{M}_{i}}^2}{\sum_{i=1}^{k}c_{i}\alpha_i(1+\theta_{i})}
 +\frac{\Big[{\sum_{i=1}^{k}}|\alpha_{i+1}\!-\!\alpha_{i}|\!+\!\max\limits_{1\le i\le k}\{\alpha_{i}\}({\sum_{i=1}^{k}}\xi_{i}+1)\Big]}{2{\sum_{i=1}^{k}}c_{i}\alpha_i(1+\theta_{i})}B_{k},
\end{align}
where $B_{k} = \max\limits_{0 \le i \le k}\{\|x^{i+1}-\overline{y}^{k}\|_{\mathcal{M}_{i+1}}^2\}$. Moreover, by the definition of $\overline{y}^{k}$, it holds that
 \begin{align*}
 \big\|x^{i+1}-\overline{y}^{k}\big\|^2_{\mathcal{M}_{i+1}}
 &\le 2\big\|x^{i+1}\big\|^2_{\mathcal{M}_{i+1}}+2\big\|\overline{y}^{k}\big\|^2_{\mathcal{M}_{i+1}}
 \le 2\big\|x^{i+1}\big\|^2_{\mathcal{M}_{i+1}}+ 2\max_{0\le j\le k}\{\big\|{y}^{j}\big\|^2_{\mathcal{M}_{i+1}}\},
 \end{align*}
 where the second inequality holds according to the convexity of $\|\cdot\|_{\mathcal{M}_{i+1}}^{2}$. Hence, we obtain
 \begin{align}\label{iteration-complexity-VMOR-HPE-eq12}
 B_{k} \le 2\Xi\overline{\omega}\max\limits_{0 \le i \le k}\big[\big\|x^{i+1}\big\|^2\!+\!\big\|{y}^{i+1}\big\|^2\big]
       \le 2\Xi\overline{\omega}\max\limits_{0 \le i \le k}\big[ 2\big\|x^{i+1}\big\|^2\!+\!\big\|x^{i+1}-{y}^{i+1}\big\|^2\big].
 \end{align}
 By \eqref{contractive-sequence} and \eqref{bound-xk}, it holds that
  $\big\|x^{i}-y^{i}\big\|_{\mathcal{M}_{i}}^2 \le \frac{\Xi}{(1-\sigma)(1+\underline{\theta})}\big\|x^{0}-x^*\big\|_{\mathcal{M}_{0}}^2$.
 Moreover, by \eqref{bound-xk}, it holds that $\frac{1}{2}\big\|x^{k}\big\|^2 \le \big\|x^*\big\|^2+\frac{\Xi}{\underline{\omega}}\big\|x^{0}-x^*\big\|^2_{\mathcal{M}_{0}}$.
 Substituting the two inequalities into \eqref{iteration-complexity-VMOR-HPE-eq12} yields
  \begin{align}\label{iteration-complexity-VMOR-HPE-eq13}
 B_{k} \le 2\Xi\Big[ \big\|x^*\big\|^2+\frac{\Xi}{\underline{\omega}}\big\|x^{0}-x^*\big\|^2_{\mathcal{M}_{0}}
        +\frac{\Xi}{(1-\sigma)\underline{\omega}(1+\underline{\theta})}\big\|x^{0}-x^*\big\|_{\mathcal{M}_{0}}^2\Big].
 \end{align}
 Combining \eqref{convergence-eq2},\eqref{iteration-complexity-VMOR-HPE-eq13} with \eqref{iteration-complexity-VMOR-HPE-eq11} and using the fact that $c_{k}\ge \underline{c}$ and $\theta_{k}\ge \underline{\theta}>-1$, we further obtain
 \begin{align*}
 \overline{\epsilon}_{k}
 &\le  \frac{8\max\limits_{0 \le i \le k}\{\alpha_{i}\}}{\sum_{i=1}^{k}c_{i}\alpha_i(1+\theta_{i})}\frac{\big(1+\sum_{i=1}^{k}\xi_{i}\big)\Xi}{(1-\sigma)(1+\underline{\theta})}\big\|x^{0}-x^*\big\|_{\mathcal{M}_{0}}^2 \\
 &\quad+\frac{{\sum_{i=1}^{k}}|\alpha_{i+1}\!-\!\alpha_{i}|\!+\!\max\limits_{1\le i\le k}\{\alpha_{i}\}({\sum_{i=1}^{k}}\xi_{i}\!+\!1)}{{\sum_{i=1}^{k}}c_{i}\alpha_i(1\!+\!\theta_{i})}
        \Xi\Big[\big\|x^*\big\|^2\!+\!\frac{\Xi}{\underline{\omega}}\Big(1\!+\!\frac{1}{(1\!-\!\sigma)(1\!+\!\underline{\theta})}\Big)\big\|x^{0}\!-\!x^*\big\|_{\mathcal{M}_{0}}^2\Big]\nonumber\\
 &\le \frac{8\max\limits_{0 \le i \le k}\{\alpha_{i}\}\big(1+\sum_{i=1}^{k}\xi_{i}\big)}{\underline{c}(1+\underline{\theta})^2\sum_{i=1}^{k}\alpha_i}
         \frac{\Xi\big\|x^{0}-x^*\big\|_{\mathcal{M}_{0}}^2}{(1-\sigma)}\nonumber \\
 &\quad  +\frac{{\sum_{i=1}^{k}}|\alpha_{i+1}\!-\!\alpha_{i}|\!+\!\max\limits_{1\le i\le k}\{\alpha_{i}\}({\sum_{i=1}^{k}}\xi_{i}+1)}{\underline{c}(1+\underline{\theta}){\sum_{i=1}^{k}}\alpha_i}
        \Big[\Xi\big\|x^*\big\|^2\!+\!\frac{\Xi^2}{\underline{\omega}}\big\|x^{0}-x^*\big\|_{\mathcal{M}_{0}}^2\Big] \nonumber\\
 &\quad  +\frac{{\sum_{i=1}^{k}}|\alpha_{i+1}\!-\!\alpha_{i}|\!+\!\max\limits_{1\le i\le k}\{\alpha_{i}\}({\sum_{i=1}^{k}}\xi_{i}+1)}{\underline{c}(1+\underline{\theta})^2{\sum_{i=1}^{k}}\alpha_i}
        \Big[\frac{\Xi^2}{\underline{\omega}}\frac{\big\|x^{0}-x^*\big\|_{\mathcal{M}_{0}}^2}{(1-\sigma)}\Big] \nonumber\\
 &\le \frac{(10+\underline{\theta})\max\limits_{1\le i\le k}\{\alpha_{i}\}\big(1+\sum_{i=1}^{k}\xi_{i}\big)+(2+\underline{\theta}){\sum_{i=1}^{k}}|\alpha_{i+1}\!-\!\alpha_{i}|}
         {\underline{c}(1+\underline{\theta})^2\sum_{i=1}^{k}\alpha_i}B,  \nonumber
\end{align*}
 where $B =\max\left\{\frac{\Xi}{(1\!-\!\sigma)}\big\|x^{0}\!-\!x^*\big\|_{\mathcal{M}_{0}}^2,\Xi\big\|x^*\big\|^2\!+\!\frac{\Xi^2}{\underline{\omega}}\big\|x^{0}\!-\!x^*\big\|_{\mathcal{M}_{0}}^2,
            \frac{\Xi^2}{(1\!-\!\sigma)\underline{\omega}}\big\|x^{0}\!-\!x^*\big\|_{\mathcal{M}_{0}}^2,M\right\}$.
 The proof is finished.
\end{proof}

\section{Proof of Proposition \ref{FBF-VMOR-HPE}}

 Recall that the over-relaxed {\bf F}orward-{\bf B}ackward-{\bf H}alf {\bf F}orward (FBHF)
 algorithm \cite{briceno2017forward} is defined as
 \begin{subnumcases}{}
  y^{k} := \mathcal{J}_{\gamma_{k}A}\big(x^{k}-\gamma_{k}(B_1+B_2)x^{k}\big),                             \label{FBHF-a} \\
   x^{k+1} := x^{k}+(1+\theta_{k})\big(y^{k}-x^{k}+\gamma_{k} B_{2}(x^{k})-\gamma_{k} B_{2}(y^{k})\big) . \label{FBHF-b-over-relaxed}
 \end{subnumcases}
 \begin{proposition*}
 Let $\{(x^{k},y^{k})\}$ be the sequence generated by the over-relaxed FBHF algorithm.
 Denote $\epsilon_{k}=\|x^{k}-y^{k}\|^2/(4\beta)$ and $v^{k}=\gamma_{k}^{-1}(x^{k}-y^{k})-B_{2}(x^{k})+B_{2}(y^{k})$. Then,
 \vspace{-2pt}
 \begin{subnumcases}{}
 (y^{k},v^{k})\in{\rm gph}\,T^{[\epsilon_{k}]}={\rm gph}\,(A+B_1+B_2)^{[\epsilon_{k}]},    \label{FBHF-VMOR-HPE-a} \\
 \theta_k\big\| \gamma_kv^{k}\big\|^2+\big\| \gamma_k v^{k}+(y^{k}-x^{k})\big\|^2+2 \gamma_k\epsilon \leq \sigma\big\|y^{k}- x^{k}\big\|^2, \label{FBHF-VMOR-HPE-b} \\
 x^{k+1}=x^{k}-(1+\theta_{k})\gamma_{k}v^{k},  \label{FBHF-VMOR-HPE-c}
 \end{subnumcases}
  \vspace{-1pt}
 where $(\gamma_{k},\theta_{k})$ satisfies $\theta_{k}\!\!\le\!\![\sigma\!-\!(\gamma_{k}L)^2\!+\!\gamma_{k}\!/\!(2\beta))]/[1\!+\!(\gamma_{k}L)^2]$.
 \end{proposition*}
 \begin{proof}
 By the definition of resolvent $\mathcal{J}_{\gamma_{k} A}$, the updating step \eqref{FBHF-a} of $y^{k}$ is formulated as follows
 \begin{align}\label{FBF-VMOR-HPE-eq1}
 x^{k}-\gamma_{k}(B_{1}+B_{2})(x^{k}) \in y^{k}+\gamma_{k}A(y^{k}).
 \end{align}
 By \citep[Lemma 2.2]{Svaiter2014A}, it holds that $B_{1}(x^{k}) \in B_1^{[\epsilon_{k}]}(y^{k})$ with $\epsilon_{k}\!=\! \|x^{k}-y^{k}\|^2/(4\beta)$. Then,
  \begin{align*}
 \gamma_{k}^{-1}(x^{k}-y^{k})-B_{2}(x^{k})+B_{2}(y^{k}) \
 &\in A(y^{k})+ B_{2}(y^{k})+B_{1}(x^{k}) \nonumber\\
 &\subseteq  A(y^{k})+ B_{2}(y^{k})+B_1^{[\epsilon_{k}]}y^{k}\nonumber\\
 &\subseteq (A+B_1+B_2)^{[\epsilon_{k}]}(y^{k}),
 \end{align*}
 where the first inclusion holds by \eqref{FBF-VMOR-HPE-eq1}, and the last inclusion holds by using the additivity property of enlargement operator \citep{burachik1998varepsilon}.
 Hence, utilizing $v^{k}=\gamma_{k}^{-1}(x^{k}-y^{k})-B_{2}(x^{k})+B_{2}(y^{k})$, we directly obtain \eqref{FBHF-VMOR-HPE-a} and \eqref{FBHF-VMOR-HPE-c}
 that $(y^{k},v^{k})\in{\rm gph}\,T^{[\epsilon_{k}]}$ and $x^{k+1}=x^{k}-(1+\theta_{k})\gamma_{k}v^{k}$, respectively. Next, we argue that \eqref{FBHF-VMOR-HPE-b} holds.
 By the monotonicity of $B_{2}$, it holds that
 \begin{align*}
 &\quad~\theta_{k}\|\gamma_{k} v^{k}\|^2+\|\gamma_{k} v^{k}+y^{k}-x^{k}\|^2 + 2\gamma_{k}\epsilon_{k} \nonumber\\
 &=\theta_{k}\big\|y^{k}-x^{k}+\gamma_{k} B_{2}(x^{k})-\gamma_{k} B_{2}(y^{k})\big\|^2+\big\|\gamma_{k}(B_{2}x^{k}-B_{2}y^{k})\big\|^2 + 2\gamma_{k}\epsilon_{k} \nonumber\\
 &\le \theta_{k}\big[ \|y^{k}-x^{k}\|^2+\|\gamma_{k} B_{2}(x^{k})-\gamma_{k} B_{2}(y^{k})\big\|^2\big]+\big\|\gamma_{k}(B_{2}x^{k}-B_{2}y^{k})\big\|^2 + 2\gamma_{k}\epsilon_{k} \nonumber\\
 &\le \big[\theta_{k}(1+\gamma_{k}^2L^2)+\gamma_{k}^2L^2+\gamma_{k}/(2\beta)\big]\|x^{k}-y^{k}\|^2 \le \sigma\|x^{k}-y^{k}\|^2,
 \end{align*}
 where the last inequality holds according to the definition of $\theta_{k}$.
 As a consequence, the FBHF algorithm with the iterations \eqref{FBHF-a} and \eqref{FBHF-b-over-relaxed} is a special case of the VMOR-HPE algorithm.
 \end{proof}

\section{Proof of Proposition \ref{nMFBHF}}
 Let $P$ be a bounded linear operator and $U=(P\!+\!P^*)/2$, $S=(P\!-\!P^*)/2$.
 The over-relaxed {\bf n}on self-adjoint {\bf M}etric {\bf F}orward-{\bf B}ackward-{\bf H}alf {\bf F}orward (nMFBHF)
 algorithm \cite{briceno2017forward} is defined as
 \begin{subnumcases}{}
 y^{k} := \mathcal{J}_{P^{-1} A}\big(x^{k}-P^{-1}(B_1+B_2)(x^{k})\big), \label{metric-FBHF-a} \\
 x^{k+1} := x^{k}+(1+\theta_{k})\big(y^{k}-x^{k}+U^{-1}[ B_{2}(x^{k})- B_{2}(y^{k})-S(x^{k}-y^{k})]\big). \label{metric-FBHF-b-over-relaxed}
 \end{subnumcases}
 \begin{proposition*}
 Let $\{(x^{k},y^{k})\}$ be the sequence generated by the over-relaxed nMFBHF algorithm.
 Denote $\epsilon_{k}= \|x^{k}-y^{k}\|^2/(4\beta)$ and $v^{k}=P(x^{k}-y^{k})+B_{2}(y^{k})-B_{2}(x^{k})$. The step-size $\theta_{k}$ satisfies
 $\theta_{k}+\frac{K^2(1+\theta_{k})}{\lambda^{2}_{\min}(U)}+\frac{1}{2\beta\lambda_{\min}(U)}\le\sigma$. Then,
 \begin{subnumcases}{}
 (y^{k},v^{k})\in{\rm gph}\,T^{[\epsilon_{k}]}={\rm gph}\,(A+B_1+B_2)^{[\epsilon_{k}]},  \label{metric-FBHF-VMOR-HPE-a} \\
 \theta_k\big\|U^{-1}v^{k}\big\|_{U}^2+\big\|U^{-1}v^{k}+ (y^{k} - x^{k})\big\|_{U}^2+ 2\epsilon  \leq  \sigma\big\|y^{k} -  x^{k}\big\|_{U}^2,  \label{metric-FBHF-VMOR-HPE-b}\\
  x^{k+1}=x^{k} - (1+\theta_{k})U^{-1}v^{k}. \label{metric-FBHF-VMOR-HPE-c}
 \end{subnumcases}
 \end{proposition*}
 \begin{proof}
 By the definition of \eqref{metric-FBHF-a}, it holds that $P(x^{k}-y^{k})-(B_{1}+B_{2})(x^{k}) \in A(y^{k})$, which indicates
 \begin{align}
 P(x^{k}-y^{k})+ B_{2}(y^{k})- B_{2}(x^{k})
 & \in A(y^{k})+B_{1}(x^{k})+B_{2}(y^{k}) \nonumber\\
 & \subseteq A(y^{k})+B_{1}^{[\epsilon_{k}]}(y^{k})+B_{2}(y^{k}) \nonumber\\
 & \subseteq (A+B_{1}+B_{2})^{[\epsilon_{k}]}(y^{k}).
 \end{align}
 By the definition of $v^{k}$, we derive \eqref{metric-FBHF-VMOR-HPE-a} that $(y^{k},v^{k})\in{\rm gph}\,T^{[\epsilon_{k}]}$.
 In addition, recall $U\!=\!(P\!+\!P^*)/2$ and  $S\!=\!(P\!-\!P^*)/2$. It is easy to check  $U^{-1}P - I = U^{-1}S$. Hence, we obtain
 \begin{align}
 x^{k+1}
 &= x^{k}+(1+\theta_{k})\big(y^{k}-x^{k}+U^{-1}[ B_{2}(x^{k})- B_{2}(y^{k})-S(x^{k}-y^{k})]\big) \nonumber\\
 &=x^{k}+(1+\theta_{k})\big(y^{k}-x^{k}-U^{-1}\big(S(x^{k}-y^{k})+B_{2}(y^{k}) - B_{2}(x^{k})\big) \big) \nonumber\\
 &=x^{k}+(1+\theta_{k})\big(y^{k}-x^{k}-U^{-1}\big(S(x^{k}-y^{k})\big)-U^{-1}\big(B_{2}(y^{k}) - B_{2}(x^{k})\big) \big) \nonumber\\
 &=x^{k}+(1+\theta_{k})\big(y^{k}-x^{k}+(I-U^{-1}P)(x^{k}-y^{k})-U^{-1}\big(B_{2}(y^{k}) - B_{2}(x^{k})\big) \big) \nonumber\\
 &=x^{k}+(1+\theta_{k})\big( U^{-1}\big(P(y^{k}-x^{k})\big)-U^{-1}\big(B_{2}(y^{k}) - B_{2}(x^{k})\big) \big) \nonumber\\
 &=x^{k}-(1+\theta_{k})U^{-1}v^{k},\nonumber
 \end{align}
 which indicates that \eqref{metric-FBHF-VMOR-HPE-c} holds. In what follows, we argue that \eqref{metric-FBHF-VMOR-HPE-b} holds.
 According to the above equality, it clearly holds that $U^{-1}v^{k}= x^{k}-y^{k}-U^{-1}[ B_{2}(x^{k})- B_{2}(y^{k})-S(x^{k}-y^{k})]$. Hence
 \begin{align*}
 &\quad~ \theta_{k} \big\|U^{-1}v^{k}\big\|^2_{U}+\big\|U^{-1}v^{k}+y^{k}-x^{k}\big\|^2_{U}+2\epsilon_{k}\nonumber\\
 &= \theta_{k}\big\|x^{k}\!-\!y^{k}\!-\!U^{-1}[(B_{2}\!-\!S)(x^{k})\!-\!(B_{2}\!-\!S)(y^{k})]\big\|^2_{U}\!+\!\big\|U^{-1}\big[(B_{2}\!-\!S)(x^{k})\!-\!(B_{2}\!-\!S)(y^{k})]\big\|^2_{U}\!+\!2\epsilon_{k} \nonumber\\
 &\le \theta_{k}\big\|x^{k}\!-\!y^{k}\big\|^2_{U}+ (1+\theta_{k})\big\|U^{-1}\big[(B_{2}\!-\!S)(x^{k})\!-\!(B_{2}\!-\!S)(y^{k})]\big\|^2_{U}+2\epsilon_{k} \nonumber\\
 &\le \theta_{k}\big\|x^{k}\!-\!y^{k}\big\|^2_{U}+(1+\theta_{k})\lambda^{-1}_{\min}(U)\|(B_{2}-S)x^{k}-(B_{2}-S)y^{k}\|^2+2\epsilon_{k} \nonumber\\
 &\le \theta_{k}\big\|x^{k}\!-\!y^{k}\big\|^2_{U}+\big[(1+\theta_{k})\lambda^{-1}_{\min}(U)K^2+1/(2\beta)\big]\|x^{k}-y^{k}\|^2 \nonumber\\
 &\le \big[\theta_{k}+[(1+\theta_{k})\lambda^{-1}_{\min}(U)K^2+1/(2\beta)]\lambda^{-1}_{\min}(U)\big]\big\|x^{k}\!-\!y^{k}\big\|^2_{U}\nonumber\\
 &\le \sigma\|x^{k}-y^{k}\|^2_{U},
 \end{align*}
 where the first inequality holds by the monotonicity of $B_{2}-S$, the second inequality holds by $\|U^{-1}\cdot\|^2_{U}\le \lambda_{\max}(U^{-1})\|\cdot\|^2=\lambda^{-1}_{\min}(U)\|\cdot\|^2$,
 the third inequality holds by the Lipschitz continuity of $B_{2}-S$, the fourth inequality holds by $\|\cdot\|^2\le \lambda^{-1}_{\min}(U)\|\cdot\|^2_{U}$, and the last inequality holds by
 $\theta_{k}+[K^2(1+\theta_{k})]/[\lambda^{2}_{\min}(U)]+1/[2\beta\lambda_{\min}(U)]\le\sigma$. Hence, \eqref{metric-FBHF-VMOR-HPE-b} holds.
 In conclusion, the over-relaxed non self-adjoint metric FBHF algorithm with the iterations \eqref{metric-FBHF-a} and \eqref{metric-FBHF-b-over-relaxed} falls into the framework of VMOR-HPE. The proof is finished.
 \end{proof}

\section{Proof of Proposition \ref{PPG-HPE}}

 The over-relaxed {\bf P}roximal-{\bf P}roximal-{\bf G}radient (PPG) algorithm \cite{ryu2017proximal} takes the following iterations:
 \begin{subnumcases}{}
 x^{k+\frac{1}{2}} := {\rm Prox}_{\alpha r}\big(\frac{1}{n}\sum_{i=1}^{n}z_{i}^{k}\big), \label{PPG-a} \\
 x_{i}^{k+1} := {\rm Prox}_{\alpha g_{i}}\big(2x^{k+\frac{1}{2}}-z_{i}^{k}-\alpha\nabla f_{i}(x^{k+\frac{1}{2}})\big),\ i=1,\ldots,n,  \label{PPG-b}\\
 z_{i}^{k+1} :=z_{i}^{k}+(1+\theta_{k})(x_{i}^{k+1}-x^{k+\frac{1}{2}}),\ i=1,\ldots,n.  \label{PPG-c-over-relaxed}
 \end{subnumcases}

 To establish Proposition \ref{PPG-HPE}, we need the following lemma which characterizes how to calculate the proximal mapping ${\rm Prox}_{\alpha \overline{r}}(\cdot)$.
 \begin{lemma}\label{proximal-bar-r}
 Given ${\bf z}\in \mathbb{X}^{n}$, ${\rm Prox}_{\alpha \overline{r}}(z)=\arg\min_{{\bf x}\in \mathbb{X}^{n}} \overline{r}({\bf x})+ \frac{1}{2\alpha}\|{\bf x}-{\bf z}\|^2$ can be  calculated in parallel with ${\rm Prox}_{\alpha \overline{r}}({\bf z})=\big({\rm Prox}_{\alpha r}(\frac{1}{n}\sum_{i=1}^{n}z_{i}),{\rm Prox}_{\alpha r}(\frac{1}{n}\sum_{i=1}^{n}z_{i}),\cdots,{\rm Prox}_{\alpha r}(\frac{1}{n}\sum_{i=1}^{n}z_{i})\big) \in V$.
 \end{lemma}
 \begin{proof}
  By the definition of $\overline{r}({\bf x})$, it holds that the components of ${\rm Prox}_{\alpha \overline{r}}({\bf z})$ are equal to each other.
  Let ${\bf 1}=(1,1,\cdots,1)\in\mathbb{X}^{n}$. By definitions of $V$ and $\overline{r}({\bf x})$, the following equalities hold
  \begin{align}
  \arg\min_{{\bf x}\in \mathbb{X}^{n}} \overline{r}({\bf x})+ \frac{1}{2\alpha}\|{\bf x}-{\bf z}\|^2
  &=\arg\min_{{\bf x}\in \mathbb{X}^{n}} {\bf 1}_{V}({\bf x})+\frac{1}{n}\sum_{i=1}^{n}r(x_{i}) + \frac{1}{2\alpha}\|{\bf x}-{\bf z}\|^2 \nonumber\\
  &= \arg\min_{{\bf x}\in V}  \frac{1}{n}\sum_{i=1}^{n}r(x_{i}) + \frac{1}{2\alpha}\|{\bf x}-{\bf z}\|^2. \label{proximal-bar-r-eq1}
  \end{align}
  Let ${\rm Prox}_{\alpha r}(\frac{1}{n}\sum_{i=1}^{n}z_{i})=\arg\min_{x \in \mathbb{X}}  r(x)+\frac{1}{2\alpha}\|x{\bf 1}-{\bf z}\|^2$.
  By the definition of $V$, we obtain
  \[
   \min_{{\bf x}\in V}  \frac{1}{n}\sum_{i=1}^{n}r(x_{i}) + \frac{1}{2\alpha}\|{\bf x}-{\bf z}\|^2
  =\min_{x \in \mathbb{X}}  r(x)+\frac{1}{2\alpha}\|x{\bf 1}-{\bf z}\|^2,
  \]
  and that ${\rm Prox}_{\alpha r}(\frac{1}{n}z{\bf 1}^T)$ solves \eqref{proximal-bar-r-eq1}.
  Hence, ${\rm Prox}_{\alpha r}(\frac{1}{n}z{\bf 1}^T){\bf 1}={\rm Prox}_{\alpha \overline{r}}(z)$. The proof is completed.
 \end{proof}

\begin{proposition*}
 Let $(x^{k+\frac{1}{2}},x_{i}^{k},z_{i}^{k})$ be the sequence generated by the over-relaxed PPG algorithm.
 Denote ${\bf x}^{k}=(x_{1}^{k},\cdots,x_{n}^{k})$, ${\bf z}^{k}=(z_{1}^{k},\cdots,z_{n}^{k})$, ${\bf 1}=(1,\cdots,1)\!\in\!\mathbb{X}^{n}$,
 ${\bf y}^{k}={\bf z}^{k}+ {\bf x^{k+1}}- x^{k+\frac{1}{2}}{\bf 1}$, ${\bf v}^{k}= x^{k+\frac{1}{2}}{\bf 1}-{\bf x}^{k+1}$, and $\epsilon_{k}=L\sum_{i=1}^{n}\|x_{i}^{k+1}-x^{k+\frac{1}{2}}\|/4$.
 Parameters $(\theta_{k},\alpha)$ are constrained by $\theta_k + L\alpha/2\le \sigma$. Then, it holds that
  \begin{subnumcases}{}
  ({\bf y}^{k},{\bf v}^{k})\in {\rm gph}\,\mathcal{S}_{\alpha,\nabla\overline{f}+\partial\overline{g},\overline\partial{r}}^{[\alpha\epsilon_{k}]}={\rm gph}\,T^{[\alpha\epsilon_{k}]},  \label{PPG-VMOR-HPE-a} \\
  \theta_k\big\|{\bf v}^{k}\big\|^2+\big\|{\bf v}^{k}+({\bf y}^{k}-{\bf z}^{k})\big\|^2+2\alpha\epsilon_{k} \leq\sigma\big\|{\bf y}^{k}\!-\!{\bf z}^{k}\big\|^2,  \label{PPG-VMOR-HPE-b}\\
  {\bf z}^{k+1}={\bf z}^{k}- (1+\theta_{k}){\bf v}^{k}.  \label{PPG-VMOR-HPE-c}
 \end{subnumcases}
 \end{proposition*}
 \begin{proof}
 By Lemma \ref{proximal-bar-r} and equation \eqref{PPG-a}, we derive $x^{k+\frac{1}{2}}{\bf 1} = {\rm Prox}_{\alpha \overline{r}}({\bf z}^{k})$. Hence,
 \begin{align}\label{kkt-proximal-bar-r}
 \alpha^{-1}\big({\bf z}^{k}- x^{k+\frac{1}{2}}{\bf 1}\big) \in \partial\overline{r}(x^{k+\frac{1}{2}}{\bf 1})
 \end{align}
 Unitizing $\overline{g}$ and $\overline{f}$, \eqref{PPG-b} is reformulated as ${\bf x}^{k+1} = {\rm Prox}_{\alpha \overline{g}}\big(2x^{k+\frac{1}{2}}{\bf 1}-{\bf z}^{k}-\alpha\nabla {\overline f}(x^{k+\frac{1}{2}}{\bf 1})\big)$.
 Then,
 \begin{align}\label{proximal-bar-g}
 \alpha^{-1}\big(2x^{k+\frac{1}{2}}{\bf 1}-{\bf x}^{k+1}-{\bf z}^{k}\big)
 &\in \partial\overline{g}({\bf x}^{k+1})+\nabla {\overline f}(x^{k+\frac{1}{2}}{\bf 1}) \\
 & \subseteq \partial\overline{g}({\bf x}^{k+1})+ \big[\nabla {\overline f}\big]^{[\epsilon_{k}]}({\bf x^{k+1}}) \nonumber\\
 &\subseteq \big[\partial\overline{g}+ \nabla {\overline f}\big]^{[\epsilon_{k}]}({\bf x^{k+1}}), \nonumber
 \end{align}
 where $\epsilon_{k}=L\|{\bf x}^{k+1}-x^{k+\frac{1}{2}}{\bf 1}\|/4=L\sum_{i=1}^{n}\|x_{i}^{k+1}-x^{k+\frac{1}{2}}\|/4$ and the second inclusion holds by \citep[Lemma 2.2]{Svaiter2014A}.
 Combining \eqref{kkt-proximal-bar-r}, \eqref{proximal-bar-g} and using simple calculations, we obtain
 \begin{align*}
 x^{k+\frac{1}{2}}{\bf 1}- {\bf x^{k+1}}
 &\in \mathcal{S}_{\alpha,[\nabla\overline{f}+\partial\overline{g}]^{[\epsilon_{k}]},\overline\partial{r}}\big({\bf x^{k+1}} + \alpha[\alpha^{-1}\big({\bf z}^{k}- x^{k+\frac{1}{2}}{\bf 1}\big)]\big) \nonumber\\
 &=  \mathcal{S}_{\alpha,[\nabla\overline{f}+\partial\overline{g}]^{[\epsilon_{k}]},\overline\partial{r}}\big({\bf z}^{k}+ {\bf x^{k+1}}- x^{k+\frac{1}{2}}{\bf 1}\big) \nonumber\\
 &\subseteq  \mathcal{S}^{[\alpha\epsilon_{k}]}_{\alpha,[\nabla\overline{f}+\partial\overline{g}],\overline\partial{r}}\big({\bf z}^{k}+ {\bf x^{k+1}}- x^{k+\frac{1}{2}}{\bf 1}\big)
  = \mathcal{S}^{[\alpha\epsilon_{k}]}_{\alpha,[\nabla\overline{f}+\partial\overline{g}],\overline\partial{r}}\big({\bf y}^{k}\big),
 \end{align*}
 where the first inclusion holds by
 ${\bf x^{k+1}} +\alpha[\alpha^{-1}\big(2x^{k+\frac{1}{2}}{\bf 1}-{\bf x}^{k+1}-{\bf z}^{k}\big)] = x^{k+\frac{1}{2}}{\bf 1} - \alpha[\alpha^{-1}\big({\bf z}^{k}- x^{k+\frac{1}{2}}{\bf 1}\big)]$
 and using the definition of $\mathcal{S}_{\alpha,\nabla\overline{f}+\partial\overline{g},\overline\partial{r}}$, and the last inclusion holds by \cite{shen2017over}. By using the notation ${\bf v}^{k}$,
 \eqref{PPG-VMOR-HPE-a} directly holds. In addition, \eqref{PPG-c-over-relaxed} can also be equivalently reformulated as ${\bf z}^{k+1} = {\bf z}^{k}+(1+\theta_{k})({\bf x}^{k+1}-x^{k+\frac{1}{2}}{\bf 1})$, which
 is equivalent to ${\bf z}^{k+1} = {\bf z}^{k}-(1+\theta_{k}){\bf v}^{k}$ by utilizing the definition of ${\bf v}^{k}$. Hence, \eqref{PPG-VMOR-HPE-c} holds. Next, using the definition of ${\bf v}^{k}$, it holds that
 \begin{align*}
 &\quad ~\theta_k\big\|{\bf v}^{k}\big\|^2+\big\|{\bf v}^{k}+({\bf y}^{k}-{\bf z}^{k})\big\|^2+2\alpha\epsilon_{k}  \nonumber\\
 &= \theta_k\big\|x^{k+\frac{1}{2}}{\bf 1}-{\bf x}^{k+1}\big\|^2+\big\|x^{k+\frac{1}{2}}{\bf 1}-{\bf x}^{k+1}+({\bf z}^{k}+ {\bf x^{k+1}}- x^{k+\frac{1}{2}}{\bf 1}-{\bf z}^{k})\big\|^2+2\alpha\epsilon_{k} \nonumber\\
 &= \big(\theta_k + L\alpha/2 \big)\big\|x^{k+\frac{1}{2}}{\bf 1}-{\bf x}^{k+1}\big\|^2 \nonumber\\
 &\le  \sigma\big\|{\bf y}^{k}- {\bf z}^{k}\big\|^2, \nonumber
 \end{align*}
 where the first equality holds due to the definitions of ${\bf v}^{k}$ and ${\bf y}^{k}$, the second equality holds due to the definition of $\epsilon_{k}$, and the last inequality holds due to
 $\theta_k + L\alpha/2 \le \sigma$, which indicates that \eqref{PPG-VMOR-HPE-b} holds. In conclusion,
 the over-relaxed PPG algorithm with the iterations \eqref{PPG-a},\eqref{PPG-b},\eqref{PPG-c-over-relaxed} falls into the framework of VMOR-HPE. The proof is finished.
 \end{proof}

\section{Proof of Proposition \ref{AFBAS-VMORHPE}}
The {\bf A}symmetric {\bf F}orward {\bf B}ackward {\bf A}djoint {\bf S}plitting (AFBAS) algorithm \cite{latafat2017asymmetric} is defined as:
 \begin{subnumcases}{}
 \overline{x}^{k} := (H+A)^{-1}\big(H-M-C\big)x^{k}     \label{AFBAS-a} \\
 x^{k+1} := x^{k}+\alpha_{k}S^{-1}(H+M^*)(\overline{x}^{k}-x^{k}), \qquad \label{AFBAS-b}
 \end{subnumcases}
 where  $\alpha_{k}=\left[\lambda_{k}\|\overline{z}^{k}-z^{k}\|^2_{P}\|\right]{\big/}\left[\|(H+M^*)(\overline{z}^{k}-z^{k})\|^{2}_{\!S^{-1}}\right]$
 and $\lambda_{k}\in [\underline{\lambda},\overline{\lambda}]\le [0,(2-1/(2\beta)]$.
\begin{proposition*}
 Let $(x^{k},\overline{x}^{k})$ be the sequence generated by the AFBAS algorithm.
 Denote $\theta_{k}=\alpha_{k}-1$, $v^{k}=(H+M^*)(x^{k})-(H+M^*)(\overline{x}^{k})$, and $\epsilon_{k} = \frac{\|\overline{z}^{k}-z^{k}\|^2_{P}}{4\beta}$. Then,
  \begin{subnumcases}\!{}
 \!\!(\overline{x}^{k},v^{k}) \in {\rm gph}\,(A+M+C)^{[\epsilon_{k}]},   \label{AFBAS-VMOR-HPE-a} \\
 \!\!\theta_k\big\|S^{-1}v^{k}\big\|_{S}^2\!+\!\big\|S^{-1}v\!+\!(\overline{x}^{k}\!-\!x^{k})\big\|_{S}^2\!+\!2\epsilon \!\leq\! \sigma\big\|\overline{x}^{k}\!-\!x^{k}\big\|_{S}^2, \label{AFBAS-VMOR-HPE-b}\\
 \!\!x^{k+1}=x^{k}-(1+\theta_{k})S^{-1}v^{k}.\label{AFBAS-VMOR-HPE-c}
 \end{subnumcases}
 \end{proposition*}
  \begin{proof}
 We first argue that  $C(z) \in C^{[\epsilon]}(x)$ with $\epsilon =  \|x-z\|^2_{P}/(4\beta)$ for any $x,z \in \mathbb{X}$. Notice that for any $y\in \mathbb{X}$,
 \begin{align*}
 \langle x-y,C(z)-C(y) \rangle
 &= \langle x-z,C(z)-C(y) \rangle+\langle z-y,C(z)-C(y) \rangle \nonumber\\
 &\ge \langle x-z,C(z)-C(y) \rangle+ \beta\|C(z)-C(y)\|_{P^{-1}}^2 \nonumber\\
 &\ge -\|x-z\|_{P}\|C(z)-C(y)\|_{P^{-1}}+\beta\|C(z)-C(y)\|_{P^{-1}}^2 \nonumber\\
 &\ge \inf_{t\ge0} \beta t^2 -\|x-z\|_{P}t = -\|x-z\|_{P}^2/(4\beta),
 \end{align*}
 where the first inequality holds by  $\big\langle x-x',C(x)-C(x')\big\rangle \ge \beta\big\|C(x)-C(x')\big\|^2_{P^{-1}}$,
 which implies $C(z) \in C^{[\epsilon]}(x)$ with $\epsilon = \|x-z\|^2_{P}/(4\beta)$
 by the definition of $C^{[\epsilon]}(x)$. Specifying $(x,z)$ as $(x^{k},\overline{x}^{k})$, it holds that $C(x^{k}) \in C^{[\epsilon_{k}]}(\overline{x}^{k})$
 with $\epsilon_{k} = \|x^{k}-\overline{x}^{k}\|^2_{P}/(4\beta)$.  This inclusion equation, in combination with \eqref{AFBAS-a}, yields
 \begin{align*}
 (H-M)(x^{k})-(H-M)(\overline{x}^{k})
 &\in A(\overline{x}^{k})+M(\overline{x}^{k})+C(x^{k}) \\
 &\subseteq A(\overline{x}^{k})+M(\overline{x}^{k})+C^{[\epsilon_{k}]}(\overline{x}^{k})\nonumber\\
 &\subseteq (A+M+C)^{[\epsilon_{k}]}(\overline{x}^{k}). \nonumber
 \end{align*}
 Due to the definition of $v^{k}$ and the operator $M$ being skew-adjoint, the above inequality indicates $v^{k}\in (A+M+C)^{[\epsilon_{k}]}(\overline{x}^{k})$, \textit{i.e.}, \eqref{AFBAS-VMOR-HPE-a} holds.
 Next, we argue that \eqref{AFBAS-VMOR-HPE-b} holds. Utilizing the formula of $v^{k}$, we obtain
 \begin{align}
 &\quad~\theta\|S^{-1}v^{k}\|^2_{S}+\|S^{-1}v^{k}+\overline{z}^{k}-z^{k}\|^2_{S}+2\epsilon_{k} \nonumber\\
 &=\theta\|(H+M^*)(x^{k}-\overline{x}^{k})\|^2_{S^{-1}}+\|(H+M^*-S)(x^{k}-\overline{z}^{k})\|^2_{S^{-1}}+\|x^{k}-\overline{x}^{k}\|^2_{P}/(2\beta)\nonumber\\
 &=\|x^{k}-\overline{x}^{k}\|^2_{\theta_{k}(H-M)S^{-1}(H+M^*)+(H-M-S)S^{-1}(H+M^*-S)+P/(2\beta)}\nonumber\\
 &=\|x^{k}-\overline{x}^{k}\|^2_{(\theta_{k}+1)(H-M)S^{-1}(H+M^*)-2H+S+P/(2\beta)}\nonumber\\
 &=\|x^{k}-\overline{x}^{k}\|^2_{(\theta_{k}+1)(H-M)S^{-1}(H+M^*)-(2-1/(2\beta)P+S}\nonumber\\
 &\le \sigma\|x^{k}-\overline{x}^{k}\|^2_{S}, \nonumber
 \end{align}
 where the first equality holds by using the definition of $\epsilon_{k}$, the second and the third equalities hold according to $M$ being skew-adjoint, the fourth equality
 holds by $H=P+K$ and $K$ being skew-adjoint, and the last inequality holds by the condition on $\theta_{k}=\alpha_{k}-1$, which implies that \eqref{AFBAS-VMOR-HPE-b} holds.
 At last, $x^{k+1}= x^{k}+\alpha_{k}S^{-1}(H+M^*)(\overline{x}^{k}-x^{k})=x^{k}-(1+\theta_{k})S^{-1}v^{k}$ holds by utilizing the definitions of $v^{k}$ and $\theta_{k}$.
 Hence, \eqref{AFBAS-VMOR-HPE-c} holds. By now, we have shown that the AFBAS algorithm with the iterations \eqref{AFBAS-a}-\eqref{AFBAS-b} falls into the framework of VMOR-HPE.
 The proof is finished.
 \end{proof}

\section{Proof of Proposition \ref{condat-vu-HPE}}

The Condat-Vu {\bf P}rimal-{\bf D}ual {\bf S}plitting (Condat-Vu  PDS) algorithm \cite{vu2013splitting,condat2013primal} takes the following iterations:
 \begin{subnumcases}\!{}
 \widetilde{x}^{k+1}:= {\rm Prox}_{r^{-1} g}\big(x^{k}- r^{-1}\nabla f(x^{k})-r^{-1}{B}^*y^{k}\big), \label{condat-vu-a}\\
 \widetilde{y}^{k+1}:={\rm Prox}_{s^{-1} h^*}\big(y^{k}+s^{-1}{B}(2\widetilde{x}^{k+1}-x^{k})\big), \label{condat-vu-b}\\
 (x^{k\!+1},y^{k\!+1}) := (x^{k},y^{k})+(1+\theta_{k})\big((\widetilde{x}^{k+1},\widetilde{y}^{k+1})-(x^{k},y^{k})\big). \label{condat-vu-c}
 \end{subnumcases}

  \begin{proposition*}
 Let $(x^{k},y^{k},\widetilde{x}^{k},\widetilde{y}^{k})$ be the sequence generated by the Condat-Vu PDS algorithm.
 Let $z^{k}\!=\!(x^{k},y^{k})$, and $w^{k}\!=\!(\widetilde{x}^{k+1},\widetilde{y}^{k+1})$.
 Parameters $(r,s,\theta_{k})$ satisfy $s-r^{-1}\|\mathcal{B}\|^2 > 0$, and $\theta_{k}+L/[2(s-r^{-1}\|\mathcal{B}\|^2)] \le \sigma$.
 Denote $v^{k}=\mathcal{M}(z^{k}-w^{k})$ and $\epsilon_{k}=L\|x^{k}-\widetilde{x}^{k+1}\|^2/4$. Then,
 \begin{subnumcases}{}
  v^{k} \in T^{[\epsilon_{k}]}(w^{k}),\label{Condat-Vu-a}\\
 \theta_{k}\!\big\|\!\mathcal{M}^{-1}v^{k}\big\|_{\!\mathcal{M}}^2+\big\|\mathcal{M}^{-1}v^{k}+w^{k}-z^{k}\big\|_{\!\mathcal{M}}^2 +2\epsilon_{k}\leq {\sigma}\big\|w^{k}-z^{k}\big\|_{\!\mathcal{M}}^2, \label{Condat-Vu-b} \\
 z^{k+1} = z^{k}-(1+\theta_{k})\mathcal{M}^{-1}v^{k}. \label{Condat-Vu-c}
 \end{subnumcases}
 \end{proposition*}
  \begin{proof}
 By the definition of ${\rm Prox}_{r^{-1} g}$, \eqref{condat-vu-a} yields $r(x^{k}-\widetilde{x}^{k+1})-{B}^*y^{k} \in  \partial{g}(\widetilde{x}^{k+1})+\nabla f(x^{k})$.
 Using \citep[Lemma 2.2]{Svaiter2014A}, we obtain $\nabla f(x^{k}) \in (\nabla f)^{[\epsilon_{k}]}(\widetilde{x}^{k+1})$ with $\epsilon_{k}\!=\!L\|x^{k}-\widetilde{x}^{k+1}\|^2/4$.
 Combining the above two inclusions and performing simple calculations yield
 \begin{align}\label{condat-ve-eq1}
 r(x^{k}-\widetilde{x}^{k+1})-{B}^*(y^{k}-\widetilde{y}^{k+1}) \in  \partial{g}(\widetilde{x}^{k+1})+(\nabla f)^{[\epsilon_{k}]}(\widetilde{x}^{k+1})+{B}^*\widetilde{y}^{k+1}.
 \end{align}
 Using the definition of ${\rm Prox}_{s^{-1} h^*}$ and performing similar operations on $\widetilde{y}^{k+1}$ as $\widetilde{x}^{k+1}$, we obtain
 \begin{align}\label{condat-ve-eq2}
 s(y^{k}-\widetilde{y}^{k+1})-{B}(x^{k}-\widetilde{x}^{k+1}) \in  \partial{h}^*(\widetilde{y}^{k+1})-{B}\widetilde{x}^{k+1}.
 \end{align}
 By the definitions of $\mathcal{M},z^{k},w^{k},T $ and $T^{[\epsilon]}$,  \eqref{condat-ve-eq1} and \eqref{condat-ve-eq2} indicate that $\mathcal{M}(z^{k}-w^{k})\in T^{[\epsilon_{k}]}(w^{k})$.
 Thus, \eqref{condat-vu-a} holds by utilizing $v^{k}=\mathcal{M}(z^{k}-w^{k})$. In addition, \eqref{condat-vu-c} can be equivalently reformulated as
 $z^{k+1} = z^{k}+(1+\theta_{k})(w^{k}-z^{k})= z^{k}-(1+\theta_{k})\mathcal{M}^{-1}v^{k}$ by using the definitions of $z^{k},w^{k}$ and $v^{k}$. Hence, \eqref{Condat-Vu-c} holds.
 Below, we argue that \eqref{Condat-Vu-b} holds. By the definition of $v^{k}$, it holds that
 \begin{align}
\theta_{k}\big\|\mathcal{M}v^{k}\big\|_{\mathcal{M}}^2+\big\|\mathcal{M}^{-1}v^{k}+w^{k} -z^{k}\big\|_{\mathcal{M}}^2+2\epsilon_{k}
 &\le \theta_{k}\big\|w^{k}-z^{k}\big\|_{\mathcal{M}}^2+L\|x^{k}-\widetilde{x}^{k+1}\|^2/2 \nonumber\\
 &\le  \big(\theta_{k}+L/(2\lambda_{\min}(\mathcal{M}))\big)\|z^{k}-w^{k}\|_{\mathcal{M}}^2 \nonumber\\
 &\le \big[\theta_{k}+L/[2(s-r^{-1}\|\mathcal{B}\|^2)]\big]\big\|w^{k} -z^{k}\big\|_{\mathcal{M}}^2\nonumber\\
 &\le \sigma\big\|w^{k} -z^{k}\big\|_{\mathcal{M}}^2, \nonumber
 \end{align}
 where the first and the second inequalities hold by using $\epsilon_{k}$ and $\|x^{k}\!-\!\widetilde{x}^{k+1}\|^2 \!\le\! \|z^{k}\!-\!w^{k}\|^2\le \|z^{k}-w^{k}\|_{\mathcal{M}}^2/\lambda_{\min}(\mathcal{M})$, respectively.
 Hence, \eqref{Condat-Vu-b} holds. In conclusion, the Condat-Vu PDS algorithm with the iterations \eqref{condat-vu-a}-\eqref{condat-vu-c} falls into the framework of VMOR-HPE. The proof is finished.
 \end{proof}

 \section{Proof of Proposition \ref{AFBAS-PD-HPE}}

 The {\bf A}symmetric {\bf F}orward {\bf B}ackward {\bf A}djoint {\bf S}plitting {\bf P}rimal-{\bf D}ual (AFBAS-PD) algorithm \cite{latafat2017asymmetric} is defined as
 \begin{subnumcases}{}
 \overline{x}^{k}:= {\rm Prox}_{\gamma_{1}g}\big(x^{k}-\gamma_{1}{B}^*y^{k}-\gamma_{1}\nabla f(x^{k})\big), \label{AFBAS-PD-a}\\
 \overline{y}^{k}:={\rm Prox}_{\gamma_{2} h^*}\big(y^{k}+\gamma_{2}{B}((1-\theta)x^{k}+\theta\overline{x}^{k})\big), \label{AFBAS-PD-b}\\
 x^{k+1} := x^{k}+\alpha_{k}\big((\overline{x}^{k}-x^{k})-\mu\gamma_{1}(2-\theta)B^*(\overline{y}^{k}-y^{k})\big),\label{AFBAS-PD-c} \\
 y^{k+1} :=  y^{k}+\alpha_{k}\big(\gamma_2(1-\mu)(2-\theta)B(\overline{x}^{k}-x^{k})+(\overline{y}^{k}-y^{k})\big), \label{AFBAS-PD-d}
 \end{subnumcases}

 where $\alpha_{k}=\big[\lambda_{k}(\gamma_{1}^{-1}\|\overline{x}^{k}-x^{k}\|^2+\gamma_{2}^{-1}\|\overline{y}^{k}-y^{k}\|^2-\theta\langle \overline{x}^{k}-x^{k},B^*(\overline{y}^{k}-y^{k})\rangle)\big]
                     {\big/}V(\overline{x}^{k}-x^{k},\overline{y}^{k}-y^{k})$, $\lambda_{k}\in [\underline{\lambda},\overline{\lambda}]\subseteq (0,\delta)$,
and $\delta$ and $V(x,y)$ are defined as $\delta=2-L(\gamma_{1}^{-1}-\gamma_{2}\theta^2\|B\|^2/4)^{-1}/2$
 and $V(x,y)=\gamma_{1}^{-1}\|x\|^2+\gamma_{2}^{-1}\|y\|^2+(1-\mu)\gamma_2(1-\theta)(2-\theta)\|Bx\|^2+\mu\gamma_{1}(2-\theta)\|B^*y\|^2+2((1-\mu)(1-\theta)-\mu)\langle x,B^*y\rangle$
 which requires $\gamma_{1}^{-1}-\gamma_{2}\theta^2\|B\|^2/4 > L/4$ and $\mu\in [0,1],\theta\in [0,\infty)$.

 Denote a linear operator $M:\mathbb{Z}\to\mathbb{Z}$ that $M=RS^{-1}$, where $R,S:\mathbb{Z}\to\mathbb{Z}$ are defined as below
 \begin{align}\label{AFBAS-PD-P-R}
 R=\left[
   \begin{array}{cc}
     \gamma_{1}^{-1} & -B^* \\
     (1\!-\!\theta)B & \gamma_{2}^{-1} \\
   \end{array}
 \right],\
  S =\left[
   \begin{array}{cc}
     1 & -\mu\gamma_{1}(2\!-\!\theta)B^* \\
     \gamma_2(1\!-\!\mu)(2\!-\!\theta)B & 1 \\
   \end{array}
 \right].
 \end{align}
  By the block matrix inversion formula \citep{horn1990matrix},  $R^{-1}$ and $M^{-1}$ are derived as below
  \begin{gather*}
  R^{-1}=\left[
   \begin{array}{cc}
      \gamma_{2}^{-1}\Xi &  \Xi B^* \\
     -(1-\theta)B\Xi & \gamma_{2}-\gamma_2(1-\theta)B\Xi B^* \\
   \end{array}
 \right],\  \Xi=\big[\gamma_{1}^{-1}\gamma_{2}^{-1}+(1-\theta)B^*B\big]^{-1}, \\
 M^{-1}=SR^{-1}
  =\left[
   \begin{array}{cc}
      \gamma_{1}\mu(2-\theta)+\gamma_{2}^{-1}[1-\mu(2-\theta)]\Xi &  [1-\mu(2-\theta)]\Xi B^* \\
      {[}1-\mu(2-\theta){]}B\Xi  & \gamma_{2}+\gamma_2[1-\mu(2-\theta)]B\Xi B^* \\
   \end{array}
 \right].
 \end{gather*}
 Here, we claim that $\Xi\!=\!\big[\gamma_{1}^{-1}\gamma_{2}^{-1}\!+\!(1-\theta)B^*B\big]^{-1}\!\succ\! 0$. In fact, if $\theta \!\le\! 1$, it is obvious that  $\Xi\!\succ\! 0$, otherwise,
 $\gamma_{1}^{-1}\!-\!\gamma_{2}\theta^2\|B\|^2/4 > L/4>0$ indicates $\gamma_{1}^{-1}\gamma_{2}^{-1}\!>\!\theta^2\|B\|^2/4 \!>\! (\theta-1)\|B\|^2\succeq (\theta-1)B^*B$. Hence,
 $\Xi\!\succ\! 0$ holds for $\theta\!\ge\! 0$. In addition, $M$ is a self-adjoint positive definite linear operator by
 Schur complement theorem \cite{horn1990matrix}.

  \begin{proposition*}
 Let $\{(\overline{x}^{k},\overline{y}^{k},x^{k},y^{k})\}$ be the sequence generated by the AFBAS-PD algorithm.
 Denote $w^{k}=(\overline{x}^{k},\overline{y}^{k})$, $z^{k}=(x^{k},y^{k})$, $v^{k}=R(z^{k}-w^{k})$, $\epsilon_{k}=L\|x^{k}-\overline{x}^{k}\|^2/4$,
 and $\theta_{k}=\alpha_{k}-1$. Then, it holds that
  \begin{subnumcases}{}
 v^{k} \in T^{[\epsilon_{k}]}(w^{k}),\label{AFBAS-PD-HPE-a}\\
 \theta_{k}\big\|\mathcal{M}^{-1}v^{k}\big\|_{\mathcal{M}}^2+\big\|\!\mathcal{M}^{-1}v^{k}+w^{k}-z^{k}\big\|_{\mathcal{M}}^2+2\epsilon_{k}\leq{\sigma}\big\|w^{k}-z^{k}\big\|_{\mathcal{M}}^2,  \label{AFBAS-PD-HPE-b} \\
 z^{k+1} = z^{k}-(1+\theta_{k})\mathcal{M}^{-1}v^{k}.  \label{AFBAS-PD-HPE-c}
 \end{subnumcases}
 \end{proposition*}
 \begin{proof}
 By the definition of ${\rm Prox}_{\gamma_{1}g}$, \eqref{AFBAS-PD-a} indicates $x^{k}\!-\!\gamma_{1}{B}^*y^{k}\!-\!\gamma_{1}\nabla f(x^{k})\in \overline{x}^{k}\!+\!\gamma_{1}\partial{g}(\overline{x}^{k})$, \textit{i.e.},
 \begin{align}\label{AFBAS-PD-HPE-eq1}
 \gamma_{1}^{-1}(x^{k}-\overline{x}^{k})-{B}^*(y^{k}-\overline{y}^{k}) \in  \partial{g}(\overline{x}^{k})+(\nabla f)^{[\epsilon_{k}]}(\overline{x}^{k})+{B}^*\overline{y}^{k}
 \end{align}
 by using $\nabla f(x^{k}) \in (\nabla f)^{[\epsilon_{k}]}(\overline{x}^{k})$. Similarly, by the definition of ${\rm Prox}_{\gamma_{2} h^*}$, \eqref{AFBAS-PD-a} indicates that
  \begin{align}\label{AFBAS-PD-HPE-eq2}
 (1-\theta)B(x^{k}-\overline{x}^{k})+\gamma_{2}^{-1}(y^{k}-\overline{y}^{k})\in  \partial{g}(\overline{y}^{k})-B\overline{x}^{k}.
 \end{align}
 By the definitions of $(z^{k},w^{k},v^{k},T^{k})$ and using the additivity property of enlargement operator \citep{burachik1998varepsilon},
 the two inclusions \eqref{AFBAS-PD-HPE-eq1}-\eqref{AFBAS-PD-HPE-eq2} indicate that $v^{k}=R(z^{k}-w^{k})\in T^{[\epsilon_{k}]}(w^{k})$. Hence, \eqref{AFBAS-PD-HPE-a} holds.
 By using $(z^{k},w^{k})$, \eqref{AFBAS-PD-c}-\eqref{AFBAS-PD-d} can be reformulated as a compact form that
 \begin{align}
 z^{k+1}=z^{k}-\alpha_{k}S(z^{k}-w^{k})=z^{k}-\alpha_{k}M^{-1}R(z^{k}-w^{k})=z^{k}-\alpha_{k}M^{-1}v^{k},
 \end{align}
 which indicates that \eqref{AFBAS-PD-HPE-c} holds. At last, we verify \eqref{AFBAS-PD-HPE-b}.  By the definition of $(M,\epsilon_{k},v^{k})$, it holds
 \begin{align*}
 &\quad~\theta_{k}\big\|\mathcal{M}^{-1}v^{k}\big\|_{\mathcal{M}}^2+\big\|\mathcal{M}^{-1}v^{k}+w^{k} -z^{k}\big\|_{\mathcal{M}}^2+2\epsilon_{k}-\sigma\|w^{k} -z^{k}\big\|_{\mathcal{M}}^2\nonumber\\
 &= \|w^{k} -z^{k}\|^2_{(\theta_{k}+1)S^*MS-S^*M-MS+(1-\sigma)M}+L\|x^{k}-\overline{x}^{k}\|^2/2\nonumber\\
 &=\|w^{k} -z^{k}\|^2_{\alpha_{k}S^*R-R^*-R+(1-\sigma)M}+L\|x^{k}-\overline{x}^{k}\|^2/2,
 \end{align*}
 where the first equality holds due to $\mathcal{M}^{-1}v^{k}=S(z^{k} -w^{k})$, and the second equality holds due to $MS=R$.
 Hence, $\theta_{k}\big\|\mathcal{M}^{-1}v^{k}\big\|_{\mathcal{M}}^2+\big\|\mathcal{M}^{-1}v^{k}+w^{k} -z^{k}\big\|_{\mathcal{M}}^2+2\epsilon_{k}<\sigma\|w^{k} -z^{k}\big\|_{\mathcal{M}}^2$,
 \textit{i.e.}, \eqref{AFBAS-PD-HPE-b} holds if it can be shown that $\alpha_{k} < \big{[}\|w^{k} -z^{k}\|^2_{R^*\!+\!R}-L\|x^{k}-\overline{x}^{k}\|^2/2\big{]}{\big/}\|w^{k} -z^{k}\|^2_{S^*R}$.
 Notice
 \[
 S^*R
 = \left[
     \begin{array}{cc}
       \gamma_{1}^{-1}+\gamma_{2}(1-\mu)(2-\theta)(1-\theta)B^*B & {[}(1-\mu)(1-\theta)-\mu{]}B^* \\
       {[}(1-\mu)(1-\theta)-\mu{]}B & \gamma_{2}^{-1}+\mu\gamma_{1}(2-\theta)BB^* \\
     \end{array}
   \right].
 \]
 Simple algebraic manipulations yield $\|w^{k} -z^{k}\|^2_{S^*R}=V(x^{k}-\overline{x}^{k},y^{k}-\overline{y}^{k})$. In addition,
 \begin{align*}
 &\quad~ \|w^{k} -z^{k}\|^2_{R^*\!+\!R}-L\|x^{k}-\overline{x}^{k}\|^2/2 \nonumber \\
 &= 2\big[\gamma^{-1}_{1}\|x^{k}-\overline{x}^{k}\|^2+\gamma^{-1}_{2}\|y^{k}-\overline{y}^{k}\|^2-\theta\langle x^{k}-\overline{x}^{k}, B^*(y^{k}-\overline{y}^{k})\rangle\big] -L\|x^{k}-\overline{x}^{k}\|^2/2\nonumber\\
 &\ge \big[2-L/[2(\gamma_{1}^{-1}-\gamma_{2}\theta^2\|B\|^2/4)]\big]
      \big[\gamma^{-1}_{1}\|x^{k}-\overline{x}^{k}\|^2\!+\!\gamma^{-1}_{2}\|y^{k}\!-\!\overline{y}^{k}\|^2\!-\!\theta\langle x^{k}-\overline{x}^{k}, B^*(y^{k}-\overline{y}^{k})\rangle\big],\nonumber
 \end{align*}
 where the first equality holds by using the definition of $R$, and the second inequality holds by the fact that
 \[
 \|x^{k}-\overline{x}^{k}\|^2
 \le \|x^{k}-\overline{x}^{k}\|_{P}^2\lambda_{\max}(P^{-1})\le \|x^{k}-\overline{x}^{k}\|_{P}^2\lambda^{-1}_{\min}(P) \le (\gamma_{1}^{-1}-\gamma_{2}\theta^2\|B\|^2/4)^{-1}\|x^{k}-\overline{x}^{k}\|_{P}^2,
 \]
 where $P = \left(
              \begin{array}{cc}
                \gamma_{1}^{-1} & -\theta B^*/2 \\
                -\theta B/2 & \gamma_{2}^{-1} \\
              \end{array}
            \right) \succ 0$. Hence, we have that
 $\theta_{k}=\alpha_{k} < \big{[}\|w^{k} -z^{k}\|^2_{R^*\!+\!R}-L\|x^{k}-\overline{x}^{k}\|^2/2\big{]}{\big/}\|w^{k} -z^{k}\|^2_{S^*R}$ holds.
 In conclusion, the AFBAS-PD algorithm with the iterations \eqref{AFBAS-PD-a}-\eqref{AFBAS-PD-d} falls into the framework of VMOR-HPE algorithm.
 The proof is finished.
 \end{proof}

 \section{Proof of Theorem \ref{convergence-nPADMM-EQN}}

 \begin{theorem*}
 Let $(\widetilde{x}^{k},\widetilde{y}^{k},x^{k},y^{k})$ be the sequence generated by the PADMM-EBB algorithm.
 Denote $v^{k} =U^{k}(z^{k}-w^{k})$, $\epsilon_{k}\!=\!\|x^{k}\!-\!\widetilde{x}^{k+1}\|_{\mathcal{D}}/4$, and operator $T$ as \eqref{inclusion-linearly-constriant}.
 Then, it holds that
 \begin{subnumcases}{}
 v^{k} \in T^{[\epsilon_{k}]}(w^{k}),\label{nPADMM-EQN-HPE-a}\\
 \theta_{k}\big\|\mathcal{M}_{k}^{-1}v^{k}\!\big\|_{\mathcal{M}_{k}}^2+\big\|\mathcal{M}_{k}^{-1}v^{k}+w^{k}-z^{k}\big\|_{\mathcal{M}_{k}}^2
                 +2\epsilon_{k}\leq\sigma\big\|w^{k}-z^{k}\!\big\|_{\mathcal{M}_{k}}^2,\qquad \label{nPADMM-EQN-HPE-b} \\
 z^{k+1} = z^{k}-(1+\theta_{k})\mathcal{M}_{k}^{-1}v^{k}. \label{nPADMM-EQN-HPE-c}
 \end{subnumcases}
 Besides, {\bf (i)}\ $(x^{k},\widetilde{x}^{k})$ and $(y^{k},\widetilde{y}^{k})$ converge to $x^{\infty}$ and $y^{\infty}$, respectively, belonging to the optimal primal-dual solution set of \eqref{linearly-constriant}. \\
 {\bf (ii)}\ There exists an integer $\overline{k}\in \{1,2,\ldots,k\}$ such that
 \[
 \sum_{i=1}^{p}{\rm dist}\big((\partial g_{i}+\nabla f_{i})(\widetilde{x}^{\overline{k}})+ \mathcal{A}_{i}\widetilde{y}^{\overline{k}},0\big)
 +\big\|b-\sum_{i=1}^{p}\mathcal{A}_{i}^{*}\widetilde{x}^{\overline{k}}_{i}\big\|\!\le\! \mathcal{O}(\frac{1}{\sqrt{k}}).
 \]
 {\bf (iii)}\ Let $\alpha_{i}=1\ {\rm or}\ i$. There exists $0\le\overline{\epsilon}^{x_{i}}_{k}\le\mathcal{O}(\frac{1}{k})$ such that
 \[
 \sum_{i=1}^{p}{\rm dist}\big((\partial g_{i}+\nabla f_{i})_{\overline{\epsilon}^{x_{i}}_{k}}(\overline{x}^{k})+ \mathcal{A}_{i}\overline{y}^{k},0\big)
 +\big\|b-\sum_{i=1}^{p}\mathcal{A}_{i}^{*}\overline{x}^{k}_{i}\big\|\le \mathcal{O}(\frac{1}{k}),
 \]
 where $\overline{x}^{k}=\frac{\sum_{i=1}^{k}(1+\theta_{i})\alpha_{i}\widetilde{x}^{i+1}}{{\sum_{i=1}^{k}}(1+\theta_{i})\alpha_{i}}$ and
 $\overline{y}^{k}=\frac{{\sum_{i=1}^{k}}(1+\theta_{i})\alpha_{i}\widetilde{y}^{i+1}}{{\sum_{i=1}^{k}}(1+\theta_{i})\alpha_{i}}$.
   \vspace{1pt} \\
 {\bf (iv)}\ If $T$ satisfies metric subregularity at $\big((x^{\infty},y^{\infty}),0\big)\!\in\! {\rm gph}T$ with modulus $\kappa\!>\!0$.
 Then, there exists $\overline{k}\!>\!0$ such that
 \begin{align*}
 {\rm dist}_{\mathcal{M}_{k+1}}\big((x^{k+1},y^{k+1}),T^{-1}(0) \big)\leq  \Big(1-\frac{\varrho_{k}}{2}\Big){\rm dist}_{\mathcal{M}_{k}}\big((x^{k},y^{k}),T^{-1}(0)\big),\ \forall k\ge\overline{k},
 \end{align*}
  \vspace{-0.15cm}
 where $\varrho_{k}:= \left[(1-\sigma)(1+\theta_{k})\right]{\Big /}\left[\big(1+\kappa\sqrt{\frac{\Xi\overline{\omega}}{\underline{\omega}}}\big)^{2}\big(1+\sqrt{\sigma+\frac{4\max\{-\theta_{k},0\}}{(1+\theta_{k})^2}}\big)^{2}\right]\in (0,1)$.
 \end{theorem*}
\begin{proof}
 By the optimality condition of the subproblem of $\widetilde{x}^{k+1}_{i}$, the following inclusion directly holds for $i=1,\ldots, p$ that
\begin{equation*}
0\!\in\! \nabla{f}_{i}(x^{k})
\!+\!\partial{g}_i(\widetilde{x}^{k+1}_i)\!+\!\mathcal{A}_{i}y^{k}\!+\!\beta_{k}\mathcal{A}_{i}\big(\sum_{j=1}^{i}\mathcal{A}^*_{j}\widetilde{x}^{k+1}_{j}
\!+\!\sum_{j=i+1}^{p}\mathcal{A}^*_{j}x^{k}_{j}\!-\!b\big)\!+\!\big(\widehat{\Sigma}_{i}\!+\!P^{k}_{i}\big)(\widetilde{x}^{k+1}_{i}\!-\!x_{i}^{k}).
\end{equation*}
Substituting $y^{k}=\widetilde{y}^{k+1}-\beta_{k}\big(\mathcal{A}^*_{1}\widetilde{x}^{k+1}_{1}+\sum_{i=2}^{p}\mathcal{A}^*_{i}x^{k}_{i}-b\big)$ into the above inclusion, we obtain
\begin{equation}\label{npADMM-EQN-eq1}
\big(\widehat{\Sigma}_{i}+P^{k}_{i}\big)(x_{i}^{k}-\widetilde{x}^{k+1}_{i})+\beta_{k}\mathcal{A}_{i}\sum_{j=2}^{i}\mathcal{A}^*_{j}(x^{k}_{j}-\widetilde{x}^{k+1}_{j})
\in \nabla{f}_{i}(x^{k})+\partial{g}_i(\widetilde{x}^{k+1}_i)+\mathcal{A}_{i}\widetilde{y}^{k+1}.
\end{equation}
 Stacking \eqref{npADMM-EQN-eq1} for $i=1,2,\ldots, p$ and $y^{k}=\widetilde{y}^{k+1}\!-\!\beta_{k}\big(\mathcal{A}^*_{1}\widetilde{x}^{k+1}_{1}+\sum_{i=2}^{p}\mathcal{A}^*_{p}{x}^{k}_{p}-b\big)$,
 we obtain
 \[\!
 \left[\!
   \begin{array}{c}
 \! (\widehat{\Sigma}_{1}\!+\!P^{k}_{1})(x_{1}^{k}\!-\!\widetilde{x}^{k+1}_{1})  \\
 \! \vdots   \\
 \!\big(\widehat{\Sigma}_{i}\!+\!P^{k}_{i}\big)(x_{i}^{k}\!-\!\widetilde{x}^{k+1}_{i})\!+\!\beta_{k}\mathcal{A}_{i}\sum_{j=2}^{i}\mathcal{A}^*_{j}(x^{k}_{j}-\widetilde{x}^{k+1}_{j}) \!\\
 \!\vdots   \\
 \!\big(\widehat{\Sigma}_{p}\!+\!P^{k}_{p}\big)(x_{p}^{k}\!-\!\widetilde{x}^{k+1}_{p})\!+\!\beta_{k}\mathcal{A}_{p}\sum_{j=2}^{p}\mathcal{A}^*_{j}(x^{k}_{j}-\widetilde{x}^{k+1}_{j}) \!\\
 \!\beta^{-1}_{k}(y^{k}-\widetilde{y}^{k+1})+\sum_{i=2}^{p}\mathcal{A}^*_{p}({x}^{k}_{i}-\widetilde{x}_{i}^{k+1})
   \end{array}
   \!\right]
 \!\in\! \left[\!
   \begin{array}{c}
 \partial g_{1}(\widetilde{x}^{k+1}_{1})     \\
 \vdots   \\
 \partial g_{i}(\widetilde{x}^{k+1}_{i}) \\
 \vdots   \\
 \partial g_{p}(\widetilde{x}^{k+1}_{p}) \\
 b
   \end{array}
 \!  \right]
 \!+\!\left[\!
   \begin{array}{c}
 \nabla f_{1}(x^{k}) \!+\! \mathcal{A}_{1}\widetilde{y}^{k+1}   \\
                   \vdots                                             \\
  \nabla f_{i}(x^{k}) \!+\!  \mathcal{A}_{i}\widetilde{y}^{k+1}  \\
                   \vdots                                             \\
  \nabla f_{p}(x^{k}) \!+\!  \mathcal{A}_{p}\widetilde{y}^{k+1}   \\
  -\sum_{i=1}^{p}\mathcal{A}^*_{i}\widetilde{x}^{k+1}_{i}
   \end{array}
   \!\right].
  \]
 By utilizing the notations $U^{k},z^{k},w^{k}$ and $T$,  the above inclusion is further reformulated as:
\begin{align}\label{npADMM-EQN-eq2}
U^{k}(z^{k}-w^{k})
&\in \left[
                        \begin{array}{c}
                          \partial{g}(\widetilde{x}^{k+1}) \\
                          b \\
                        \end{array}
                      \right]
+\left[
                        \begin{array}{c}
                          \nabla{f}(x^{k}) \\
                          0 \\
                        \end{array}
                      \right]
+\left[
                        \begin{array}{c}
                          \mathcal{A}^*\widetilde{y}^{k} \\
                          -\sum_{i=1}^{p}\mathcal{A}^*_{i}\widetilde{x}^{k+1}_{i} \\
                        \end{array}
                      \right] \\
&\subseteq \left[
                        \begin{array}{c}
                          \partial{g}(\widetilde{x}^{k+1}) \\
                          b \\
                        \end{array}
                      \right]
+\left[
                        \begin{array}{c}
                          \nabla{f}^{[\epsilon_{k}]}(\widetilde{x}^{k+1}) \\
                          0 \\
                        \end{array}
                      \right]
+\left[
                        \begin{array}{c}
                          \mathcal{A}^*\widetilde{y}^{k} \\
                          -\sum_{i=1}^{p}\mathcal{A}^*_{i}\widetilde{x}^{k+1}_{i} \\
                        \end{array}
                      \right], \nonumber
 \end{align}
 where $g(x)=\sum_{i=1}^{p}g_{i}(x_{i})$, and $\mathcal{A}=[\mathcal{A}_{1}\ \ \mathcal{A}_{2}\ \  \cdots\ \ \mathcal{A}_{p}]$. Using the additivity property of enlargement operator \citep{burachik1998varepsilon}
 and the definition of $T$, the above inclusion indicates
 \begin{align*}
 v^{k} = U^{k}(z^{k}-w^{k}) \in T^{[\epsilon_{k}]}(w^{k}).
 \end{align*}
 Besides, by utilizing the updating step of $(x^{k+1},y^{k+1})$  and the definition of $(v^{k},w^{k},z^{k})$ , it holds that
 \begin{align*}
 z^{k+1}
 &=(x^{k+1},y^{k+1})=(x^{k},y^{k})+(1+\theta_k)\mathcal{M}^{-1}_{k}U^{k}\big(\widetilde{x}^{k+1}-x^{k},\widetilde{y}^{k+1}-y^{k})\nonumber\\
 &=z^{k}+(1+\theta_k)\mathcal{M}^{-1}_{k}U^{k}(w^{k}-z^{k})\nonumber\\
 &=z^{k}-(1+\theta_k)\mathcal{M}^{-1}_{k}v ^{k}.
 \end{align*}
 Hence, \eqref{nPADMM-EQN-HPE-a} and \eqref{nPADMM-EQN-HPE-c} hold.  At last, we check \eqref{nPADMM-EQN-HPE-b}. By the definition of $(v^{k},\epsilon_{k})$, it holds that
 \begin{align*}
 &\theta_{k}\big\|\mathcal{M}_{k}^{-1}v^{k}\big\|_{\mathcal{M}_{k}}^2+\big\|\mathcal{M}_{k}^{-1}v^{k}+w^{k} -z^{k}\big\|_{\mathcal{M}_{k}}^2
                 +2\epsilon_{k}- {\sigma}\big\|w^{k} -z^{k}\big\|_{\mathcal{M}_{k}}^2 \nonumber\\
 &=\big\|w^{k} -z^{k}\big\|^2_{(1+\theta_{k})(U^k)^*\mathcal{M}_{k}^{-1}U^{k}-(U^{k})^*-U^{k}+(1-\sigma)\mathcal{M}_{k}+\mathcal{D}/2}\nonumber\\
 &\le 0,
 \end{align*}
 where the last inequality holds by the setting of over-relaxed step-size $\theta_{k}$. Hence, PADMM-EBB is equivalently reformulated as \eqref{nPADMM-EQN-HPE-a}-\eqref{nPADMM-EQN-HPE-c}, \textit{i.e.},
 it falls into the framework of VMOR-HPE.
 By Theorem \ref{VMOR-HPE-convergence}, {\bf (i)} directly holds that $(x^{k},y^{k})$ and $(\widetilde{x}^{k},\widetilde{y}^{k})$ simultaneously converge to a point $(x^{\infty},y^{\infty})$ belonging to $T^{-1}(0)$
 which is exactly the primal-dual optimal solution set of \eqref{linearly-constriant}. In the following, we argue that {\bf (ii)} and {\bf (iii)} hold by utilizing Theorem \ref{iteration-complexity-VMOR-HPE}.
 In fact, using \eqref{npADMM-EQN-eq2}, we have
 \begin{align*}
v^{k}
+\left[
                        \begin{array}{c}
                          \nabla{f}(\widetilde{x}^{k+1}) \\
                          0 \\
                        \end{array}
                      \right]
-\left[
                        \begin{array}{c}
                          \nabla{f}(x^{k}) \\
                          0 \\
                        \end{array}
                      \right]
&\in \left[
                        \begin{array}{c}
                          \partial{g}(\widetilde{x}^{k+1}) \\
                          b \\
                        \end{array}
                      \right]
+\left[
                        \begin{array}{c}
                          \nabla{f}(\widetilde{x}^{k+1}) \\
                          0 \\
                        \end{array}
                      \right]
+\left[
                        \begin{array}{c}
                          \mathcal{A}^*\widetilde{y}^{k} \\
                          -\sum_{i=1}^{p}\mathcal{A}^*_{i}\widetilde{x}^{k+1}_{i} \\
                        \end{array}
                      \right]  = T(w^{k}).
\end{align*}
 Hence, ${\rm dist}\big(T(w^{k}), 0\big)\le \|v^{k}\|+L\|x^{k}-\widetilde{x}^{k+1}\|=\|v^{k}\|+4\epsilon_{k}$. This, in combination with \eqref{pointwise-vk-epsilonk}, yields the desired result {\bf (i)},
 \textit{i.e.},  there exists an integer $\overline{k}\in \{1,2,\ldots,k\}$ such that
 \[
 \sum_{i=1}^{p}{\rm dist}\big(\partial g_{i}(\widetilde{x}^{\overline{k}}_{i})+\nabla f_{i}(\widetilde{x}^{\overline{k}})  + \mathcal{A}_{i}\widetilde{y}^{\overline{k}},0\big)
 +\big\|b-\sum_{i=1}^{p}\mathcal{A}_{i}^{*}\widetilde{x}^{\overline{k}}_{i}\big\|={\rm dist}\big(T(w^{\overline{k}-1}), 0\big) \le \mathcal{O}(\frac{1}{\sqrt{k}}).
 \]
 Next, we claim that $\overline{\epsilon}^{x_{j}}_{k}=\frac{{\textstyle\sum_{i=1}^{k}}(1+\theta_{i})\alpha_{i}
\big(\epsilon^{x_{j}}_{k}+\langle \widetilde{x}_{j}^{i+1}-\overline{x}_{j}^{k},G_{x_{j}}^{i}-\overline{G}_{x_{j}}^{k}\rangle\big)}{{\sum_{i=1}^{k}}(1+\theta_{i})\alpha_{i}}$,
 where $\epsilon^{x_{j}}_{k} = \frac{L_{j}\|x_{j}^{k}-\widetilde{x}^{k+1}_{j}\|^2}{4}$,
 \begin{gather*}
G^{i}_{x_{1}}=(\widehat{\Sigma}_{1}+\beta_{k}{P}^{k}_1)\big({x}_1^k-\widetilde{x}_1^{i+1}\big)-\mathcal{A}_1\widetilde{y}^{i+1},\ \\
G^{i}_{x_{2}}=\big(\widehat{\Sigma}_{2}+\beta_{k}(\mathcal{A}_{2}\mathcal{A}_{2}^*+{P}^k_2)\big)\big({x}_2^i-\widetilde{x}_2^{i+1}\big)-\mathcal{A}_2\widetilde{y}^{i+1},\    \\
 \vdots \\
G^{i}_{x_{p}}=\big(\widehat{\Sigma}_{p}+\beta_{k}(\mathcal{A}_{p}\mathcal{A}_{p}^*+{P}^{k}_p)\big)\big({x}_p^i-\widetilde{x}_p^{i+1}\big)
                    +\sum_{j=2}^{p-1}\sigma\mathcal{A}_{i}\mathcal{A}_{j}^*\big({x}_j^i-\widetilde{x}_j^{i+1}\big)-\mathcal{A}_{p}\widetilde{y}^{i+1},\,\\
\overline{G}_{x_{1}}^{k}=\frac{{\sum_{i=1}^{k}}(1+\theta_{i})\alpha_{i}G_{x_{1}}^{i}}{{\sum_{i=1}^{k}}(1+\theta_{i})\alpha_{i}},\,
\overline{G}_{x_{2}}^{k}=\frac{{\sum_{i=1}^{k}}(1+\theta_{i})\alpha_{i}G_{x_{2}}^{i}}{{\sum_{i=1}^{k}}(1+\theta_{i})\alpha_{i}},\,
            \cdots,
\overline{G}_{x_{p}}^{k}=\frac{{\sum_{i=1}^{k}}(1+\theta_{i})\alpha_{i}G_{x_{p}}^{i}}{{\sum_{i=1}^{k}}(1+\theta_{i})\alpha_{i}},\,\\
G_{y}^{i} = \beta_{k}^{-1}(y^{i}\!-\!\widetilde{y}^{i+1})+\sum_{j=2}^{p}\mathcal{A}^*_{p}({x}^{i}_{j}-\widetilde{x}_{j}^{i+1}),\,
{\rm \ and\ }\overline{G}_{y}^{k}=\frac{{\sum_{i=1}^{k}}(1+\theta_{i})\alpha_{i}G_{y}^{i}}{{\sum_{i=1}^{k}}(1+\theta_{i})\alpha_{i}}.
\end{gather*}
 Define $\overline{w}^{k} = \frac{{\textstyle\sum_{i=1}^{k}}(1+\theta_{i})\alpha_{i}w^{i}}{{\textstyle\sum_{i=1}^{k}}(1+\theta_{i})\alpha_{i}}$,
 $\overline{v}^{k}=\frac{{\textstyle\sum_{i=1}^{k}}(1+\theta_{i})\alpha_{i}{v}^{i}}{{\textstyle\sum_{i=1}^{k}}(1+\theta_{i})\alpha_{i}}$ and
 $\overline{\epsilon}^{k} = \frac{{\textstyle\sum_{i=1}^{k}}(1+\theta_{i})\alpha_{i}\big(\epsilon_{i}+\langle w^{i}-\overline{w}^{k},v^{i}-\overline{v}^{k}\rangle\big)}{{\textstyle\sum_{i=1}^{k}}(1+\theta_{i})\alpha_{i}}$
 as Theorem \ref{iteration-complexity-VMOR-HPE}. Hence, utilizing \eqref{weighted-vk}-\eqref{weighted-epsilonk} and \eqref{nPADMM-EQN-HPE-a}-\eqref{nPADMM-EQN-HPE-a},
 we obtain $\|\overline{v}^{k}\| \le \frac{1}{k}$ and $\overline{\epsilon}^{k}\le \frac{1}{k}$ by setting
 $\alpha_{i}=1$ or $\alpha_{i}=i$. Using \eqref{npADMM-EQN-eq2} and the definitions of $G^{k}_{x_{1}},\cdots,G^{k}_{x_{p}}$ and $\overline{G}^{k}_{x_{1}},\cdots,\overline{G}^{k}_{x_{p}}$,
 we have
 \begin{align*}
\left(
  \begin{array}{c}
    G^{k}_{x_{1}}+\mathcal{A}_1\widetilde{y}^{k+1}\\
    \vdots \\
    G^{k}_{x_{p}}+\mathcal{A}_{p}\widetilde{y}^{k+1}\\
  \end{array}
\right)
\in
\left(
  \begin{array}{c}
    \big(\partial{g}_1+\nabla{f}_{1})_{[\epsilon^{x_{1}}_{k}]}(\widetilde{x}_1^{k+1})+\mathcal{A}_1\widetilde{y}^{k+1}\\
     \vdots \\
    \big(\partial{g}_{p}+\nabla{f}_{p})_{[\epsilon^{x_{p}}_{k}]}(\widetilde{x}_{p}^{k+1})+\mathcal{A}_{p}\widetilde{y}^{k+1}\\
  \end{array}
\right).
\end{align*}
 By utilizing \citep[theorem 2.3]{burachik1998varepsilon},  it holds that $\overline{\epsilon}^{x_{i}}_{k} \ge 0$ for all $i\in \{1,\cdots,p\}$ and
\begin{align*}
&\left(
  \begin{array}{c}
    \overline{G}^{k}_{x_{1}}+\mathcal{A}_1\overline{y}^{k} \\
    \vdots \\
    \overline{G}^{k}_{x_{p}}+\mathcal{A}_{p}\overline{y}^{k}\\
  \end{array}
\right)
\subseteq
\left(
  \begin{array}{c}
    \big(\partial{g}_1+\nabla{f}_{1})_{[\overline{\epsilon}^{x_{1}}_{k}]}(\overline{x}_1^{k})+\mathcal{A}_1\overline{y}^{k}\\
     \vdots \\
    \big(\partial{g}_{p}+\nabla{f}_{p})_{[\overline{\epsilon}^{x_{p}}_{k}]}(\overline{x}_{p}^{k})+\mathcal{A}_{p}\overline{y}^{k}  \\
  \end{array}
\right).
\end{align*}
By \eqref{npADMM-EQN-eq2} and $G_{y}^{k}=\beta^{-1}_{k}(y^{k}-\widetilde{y}^{k+1})+\sum_{i=2}^{p}\mathcal{A}^*_{p}({x}^{k}_{i}-\widetilde{x}_{i}^{k+1})=b-\sum_{i=1}^{p}\mathcal{A}^*_{i}\widetilde{x}^{k+1}_{i}$, we get that
\begin{align*}
\overline{v}^{k}&=\left(
  \begin{array}{c}
    \overline{G}^{k}_{x_{1}}+\mathcal{A}_1\overline{y}^{k} \\
    \vdots \\
    \overline{G}^{k}_{x_{p}}+\mathcal{A}_{p}\overline{y}^{k}\\
    \overline{G}_{y}^{k}
  \end{array}
\right)
\subseteq
\left(
  \begin{array}{c}
    \big(\partial{g}_1+\nabla{f}_{1})_{[\overline{\epsilon}^{x_{1}}_{k}]}(\overline{x}_1^{k})+\mathcal{A}_1\overline{y}^{k}\\
     \vdots \\
    \big(\partial{g}_{p}+\nabla{f}_{p})_{[\overline{\epsilon}^{x_{p}}_{k}]}(\overline{x}_{p}^{k})+\mathcal{A}_{p}\overline{y}^{k}  \\
    b-\sum_{i=1}^{p}\mathcal{A}^*_{i}\overline{x}^{k}_{i}
  \end{array}
\right)
\subseteq  T^{[\overline{\epsilon}^{x_{1}}_{k}+\ldots+\overline{\epsilon}^{x_{p}}_{k}]}(w^{i}).
\end{align*}
Hence, we obtain $\sum_{i=1}^{p}{\rm dist}\big((\partial g_{i}+\nabla f_{i})_{\overline{\epsilon}^{x_{i}}_{k}}(\overline{x}^{k})  + \mathcal{A}_{i}\overline{y}^{k},0\big)
 +\big\|b-\sum_{i=1}^{p}\mathcal{A}_{i}^{*}\overline{x}^{k}_{i}\big\|\le \|\overline{v}^{k}\| \le \mathcal{O}(\frac{1}{k})$.
 Next, we show that $0\!\le\!\overline{\epsilon}^{x_{i}}_{k}\!\le\!\mathcal{O}(\frac{1}{k})$ for all $i=1,2,\ldots,p$.  Notice
 \begin{align}\label{majorized-epsilon-estimation}
\overline{\epsilon}^{x_{1}}_{k}+\cdots+\overline{\epsilon}^{x_{p}}_{k}
&=\sum_{j=1}^{p}\left\{\frac{1}{{\sum_{i=1}^{k}}(1+\theta_{i})\alpha_{i}}{\sum_{i=1}^{k}}(1+\theta_{i})\alpha_{i}
\big(\epsilon^{x_{i}}_{k}+\langle \widetilde{x}_{j}^{i+1}-\overline{x}_{j}^{k},G_{x_{j}}^{i}-\overline{G}_{x_{j}}^{k}\rangle\big)\right\}\nonumber\\
&=\frac{1}{{\sum_{i=1}^{k}}(1+\theta_{i})\alpha_{i}}{\sum_{i=1}^{k}}(1+\theta_{i})\alpha_{i}
\big(\sum_{j=1}^{p}\epsilon^{x^{j}}_{i}+\sum_{j=1}^{p}\langle \widetilde{x}_{j}^{i+1}-\overline{x}_{j}^{k},G_{x_{j}}^{i}-\overline{G}_{x_{j}}^{k}\rangle\big)\nonumber\\
&=\frac{1}{{\sum_{i=1}^{k}}(1+\theta_{i})\alpha_{i}}{\sum_{i=1}^{k}}(1+\theta_{i})\alpha_{i}
\big(\epsilon_{i}+\langle \widetilde{x}^{i+1}-\overline{x}^{k},G_{x}^{i}-\overline{G}_{x}^{k}\rangle\big),
\end{align}
where the third equality holds according to $\epsilon_{i}=\sum_{j=1}^{p}\epsilon^{x_{j}}_{i}$, and $(\widetilde{x}^{i+1},\overline{x}^{k},G_{x}^{i},\overline{G}_{x}^{k})$ are defined as

\[
\widetilde{x}^{i+1}=\left(
          \begin{array}{c}
            \widetilde{x}_{1}^{i+1} \\
            \vdots \\
            \widetilde{x}_{p}^{i+1} \\
          \end{array}
        \right),\,
\overline{x}^{k}=\left(
                       \begin{array}{c}
                         \overline{x}_{1}^{k} \\
                         \vdots \\
                         \overline{x}_{p}^{k} \\
                       \end{array}
                     \right),\,
G_{x}^{i}=\left(
                \begin{array}{c}
                  G_{x_{1}}^{i} \\
                  \vdots \\
                  G_{x_{p}}^{i} \\
                \end{array}
              \right),\,
\overline{G}_{x}^{k}=\left(
                           \begin{array}{c}
                             \overline{G}_{x_{1}}^{k} \\
                             \vdots \\
                             \overline{G}_{x_{p}}^{k} \\
                           \end{array}
                         \right).
\]
Let $v_{i}^{k}=  G^{k}_{x_{i}}+\mathcal{A}_{i}\widetilde{y}^{k+1}$ be the $i$-th component of $v^{k}$. Using
$\widetilde{x}^{i+1},\,\overline{x}^{k},\,G_{x}^{i},\,\overline{G}_{x}^{k}$,  we obtain
\begin{align}\label{majorized-epsilon-estimation-eq1}
&\quad \sum_{i=1}^{k}(1+\theta_{i})\alpha_{i}\langle \widetilde{x}^{i+1}-\overline{x}^{k},G_{x}^{i}-\overline{G}_{x}^{k}\rangle
 = \sum_{i=1}^{k}(1+\theta_{i})\alpha_{i}\langle \widetilde{x}^{i+1}-\overline{x}^{k},G_{x}^{i}\rangle \\
&= \sum_{i=1}^{k}(1+\theta_{i})\alpha_{i}\langle \widetilde{x}^{i+1}-\overline{x}^{k},[v^{i}_{1}-\mathcal{A}_{1}\widetilde{y}^{i+1},\cdots,v^{i}_{p}-\mathcal{A}_{p}\widetilde{y}^{i+1}]^T\rangle\nonumber\\
&= \sum_{i=1}^{k}(1+\theta_{i})\alpha_{i}\langle \widetilde{x}^{i+1}-\overline{x}^{k},[v^{i}_{1},\cdots, v^{i}_{p}]^T\rangle
-\sum_{i=1}^{k}(1+\theta_{i})\alpha_{i}\langle \widetilde{x}^{i+1}-\overline{x}^{k},[\mathcal{A}_{1}\widetilde{y}^{i+1}, \cdots,\mathcal{A}_{p}\widetilde{y}^{i+1}]^T
\rangle\nonumber\\
&= \sum_{i=1}^{k}(1+\theta_{i})\alpha_{i}\langle \widetilde{x}^{i+1}-\overline{x}^{k},[v^{i}_{1}, \cdots, v^{i}_{p}]^T\rangle
-\sum_{i=1}^{k}(1+\theta_{i})\alpha_{i}(\widetilde{y}^{i+1})^\top\sum_{j=1}^{p}\mathcal{A}_{j}^*({x}_{j}^{i+1}-\overline{x}_{j}^{k})\nonumber\\
& = -\sum_{i=1}^{k}(1+\theta_{i})\alpha_{i}\langle \widetilde{y}^{i+1},G_{y}^{i}-\overline{G}_{y}^{i})\rangle
    -\sum_{i=1}^{k}(1+\theta_{i})\alpha_{i}(\widetilde{y}^{i+1})^\top\sum_{j=1}^{p}\mathcal{A}_{j}^*(\widetilde{x}_{j}^{i+1}-\overline{x}_{j}^{k})
     +\sum_{i=1}^{k}(1+\theta_{i})\alpha_{i}\langle w^{i}-\overline{w}^{k},v^{i}\rangle,\nonumber
\end{align}
where the last equality holds by using the definitions of $v^{k},w^{k}$ and $\overline{v}^{k},\overline{w}^{k}$. In addition,
\begin{align*}
&\quad~\sum_{i=1}^{k}(1+\theta_{i})\alpha_{i}(\widetilde{y}^{i+1})^\top\sum_{j=1}^{p}\mathcal{A}_{j}^*(\widetilde{x}_{j}^{i+1}-\overline{x}_{j}^{k})
+\sum_{i=1}^{k}(1+\theta_{i})\alpha_{i}\langle \widetilde{y}^{i+1},G_{y}^{i}-\overline{G}_{y}^{i})\rangle \nonumber\\
&=\sum_{i=1}^{k}(1\!+\!\theta_{i})\alpha_{i}(\widetilde{y}^{i+1})^\top\left\{\sum_{j=1}^{p}\mathcal{A}_{j}^*\widetilde{x}_{j}^{i+1}\!-\!b
                        \!-\!\big(\sum_{j=1}^{p}\mathcal{A}_{j}^*\overline{x}_{j}^{k}\!-\!b\big)\right\}
\!+\!\sum_{i=1}^{k}(1\!+\!\theta_{i})\alpha_{i}\langle \widetilde{y}^{i+1},G_{y}^{i}\!-\!\overline{G}_{y}^{i})\rangle \nonumber\\
&=\sum_{i=1}^{k}(1+\theta_{i})\alpha_{i}\langle \widetilde{y}^{i+1}, \overline{G}^{i}_{y}-G^{i}_{y}\rangle
+\sum_{i=1}^{k}(1+\theta_{i})\alpha_{i}\langle \widetilde{y}^{i+1},G_{y}^{i}-\overline{G}_{y}^{i})\rangle =0 \nonumber.
\end{align*}
By the definition of $\epsilon_{k}$ and combining the above equality with \eqref{majorized-epsilon-estimation} and \eqref{majorized-epsilon-estimation-eq1}, it directly holds that
\[
\overline{\epsilon}^{x_{1}}_{k}+\cdots+\overline{\epsilon}^{x_{p}}_{k} = {\epsilon}^{x}_{k} \le \mathcal{O}\big(\frac{1}{k}\big).
\]
Thus,  {\bf (iii)} has been established. At last, {\bf (iv)} is directly derived according to Theorem \ref{linear-rate} by setting $c_{k}=\underline{c}=1$.
As a consequence, the proof is completed.
\end{proof}

\section{More Experiments}
 Actually, to make the subproblems of PADMM-EBB, PLADMM-PSAP \cite{liu2013linearized,lin2015linearized}, PGSADMM and M-GSJADMM \cite{lu2017unified}
 have closed-form solutions, we equivalently reformulate problem \eqref{DGR-LRR} as the following form by introducing two slack variables $(H,F)$ to separate
 the sparsity and nonnegativity of  $(Z,G)$:
 \begin{align}\label{slack-DGR-LRR}
 &\min~ \|H\|_{*}+\|F\|_{*}+\lambda\|E\|_{1}+\frac{\mu}{2}\|Z\|^2_{L_{Z}}+\frac{\gamma}{2}\|G\|^2_{L_{G}}\\
 &~~ {\rm s.t.}\  X=XZ+GX+E,Z\ge 0,\ G\ge 0,Z = H,G = F. \nonumber
 \end{align}

 In the implementation, we measure the performance of the four solvers of PADMM-EBB, PLADMM-PSAP \cite{liu2013linearized,lin2015linearized}, PGSADMM and M-GSJADMM \cite{lu2017unified}
 in terms of the proximal KKT residual defined as \eqref{inclusion-linearly-constriant}, objective value, and feasibility of \eqref{DGR-LRR} over iterations and runtime.
  Below, we report the performance on $X={\rm randn}(200,200)$ and PIE\_pose27 of PADMM-EBB, PLADMM-PSAP, PGSADMM and M-GSJADMM with new hyperparameters $(\lambda,\mu, \gamma)=(10^2,10^4,10^4)$.
  In addition, we conduct experiments on two extra real datasets (COIL20, YaleB\_32x32)\footnote{http://dengcai.zjulearning.org:8081/Data/FaceDataPIE.html}
  with hyperparameters $(\lambda,\mu, \gamma)=(10^2,10^4,10^4)$ and  $(\lambda,\mu, \gamma)=(10^3,10^4,10^4)$ .
  In the implementation of PLADMM-PSAP, PGSADMM and M-GSJADMM, the penalty parameters $\beta_{k}$ are all updated via the suggestions from \cite{lu2017unified},
  \textit{i.e.}, $\beta_{k+1}= \min(\rho\beta_{k},1.0e10)$  where $\rho = 1.1$ and $\beta_{0} = 1.0e-4$.

  \begin{figure*}[htpb]
\centering
\subfigure{\label{fig:pic1}
\includegraphics[width=0.24\linewidth]{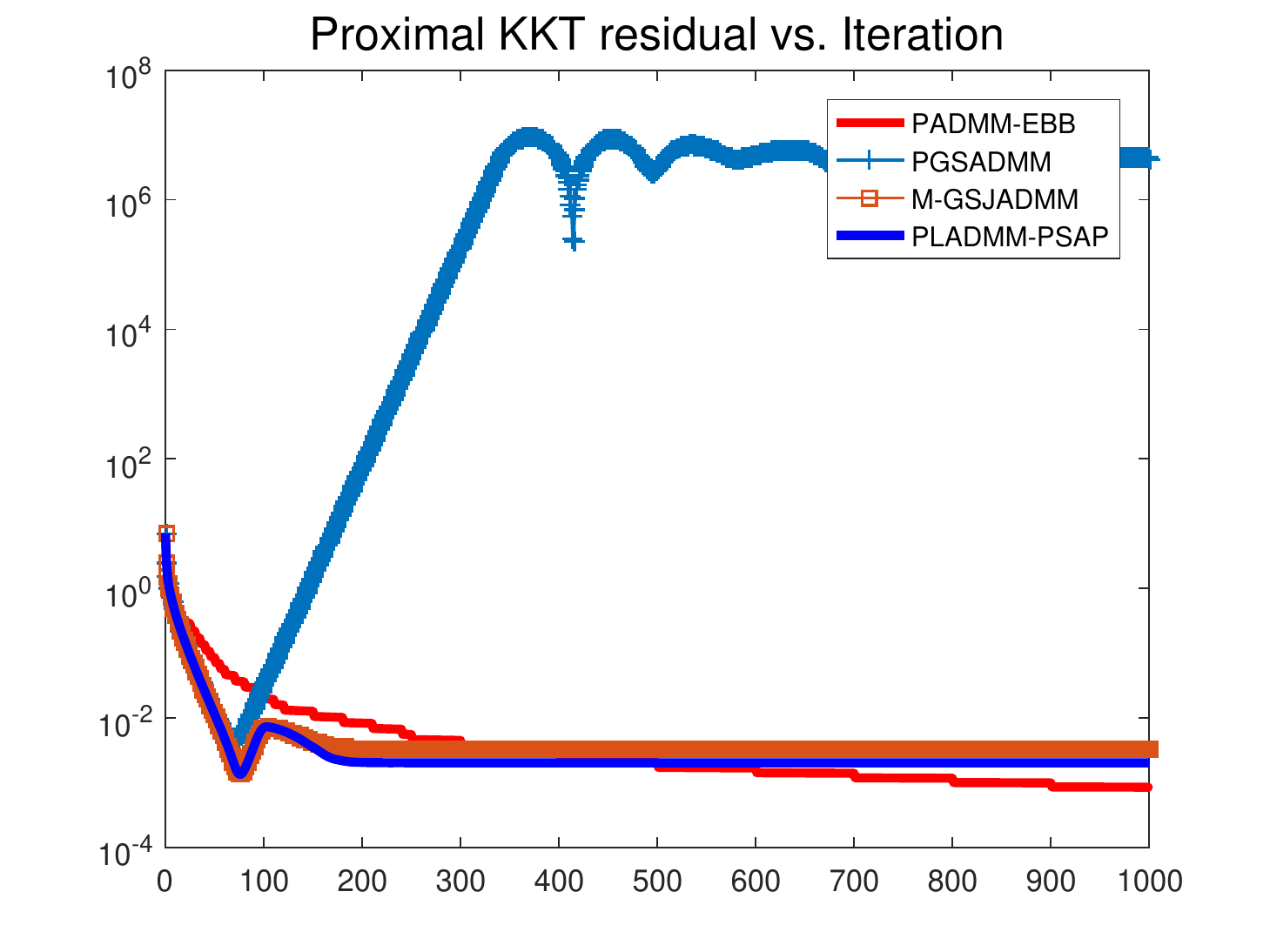}}
\subfigure{\label{fig:pic1}
\includegraphics[width=0.24\linewidth]{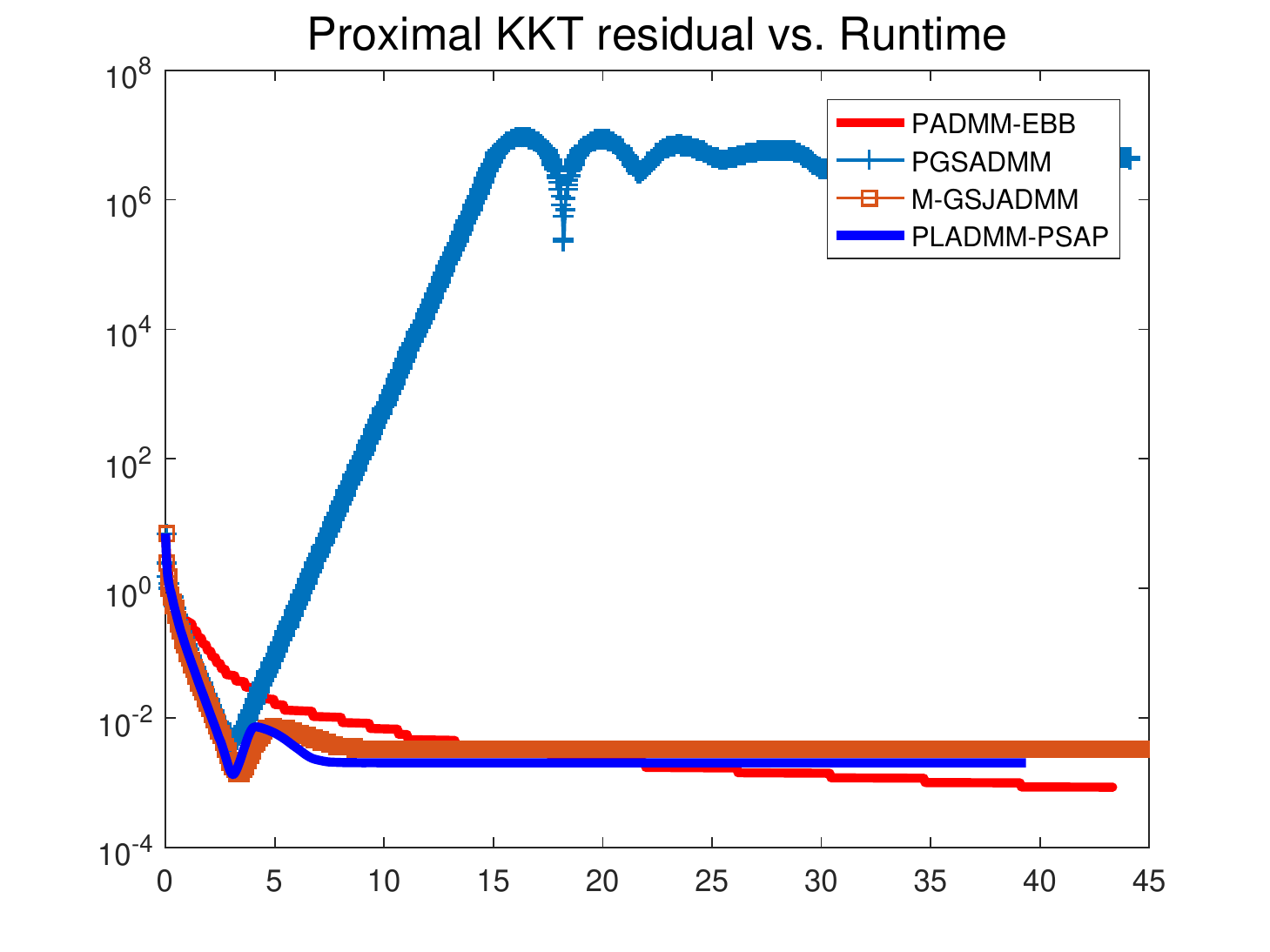}}
\subfigure{\label{fig:pic1}
\includegraphics[width=0.24\linewidth]{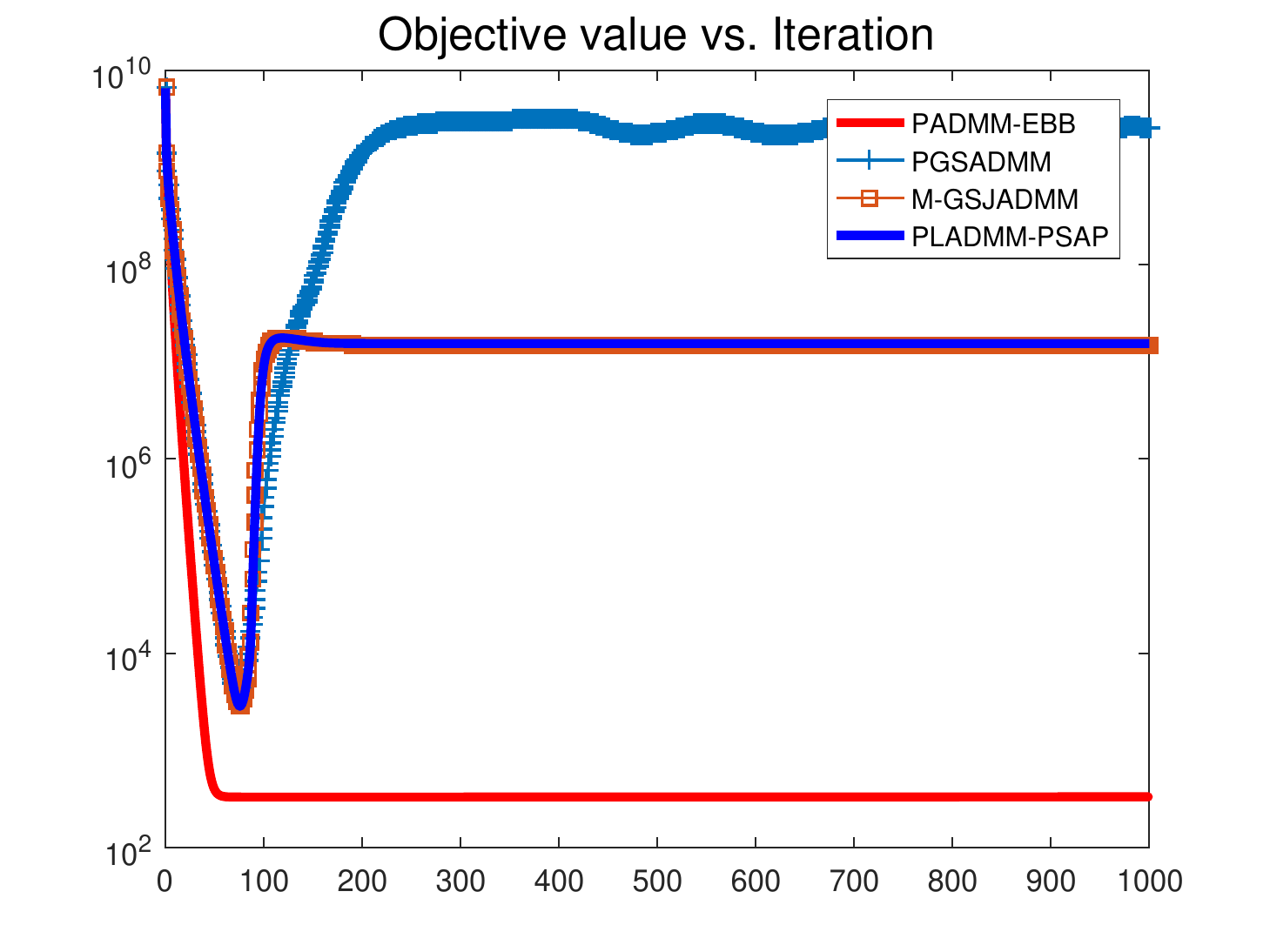}}
\subfigure{\label{fig:pic1}
\includegraphics[width=0.24\linewidth]{random_feasi_iter_1000_10000-eps-converted-to}}
 \caption{ The above four figures illustrate the proximal KKT residual vs. iteration,
  proximal KKT residual vs. runtime, objective value vs. iteration,
  and feasibility vs. iteration on the synthetic dataset with parameters $(\lambda,\mu, \gamma)=(10^2,10^4,10^4)$, respectively.}
\label{fig:overlap-3}
\end{figure*}
\begin{figure*}[htpb]
\centering
\subfigure{\label{fig:pic1}
\includegraphics[width=0.24\linewidth]{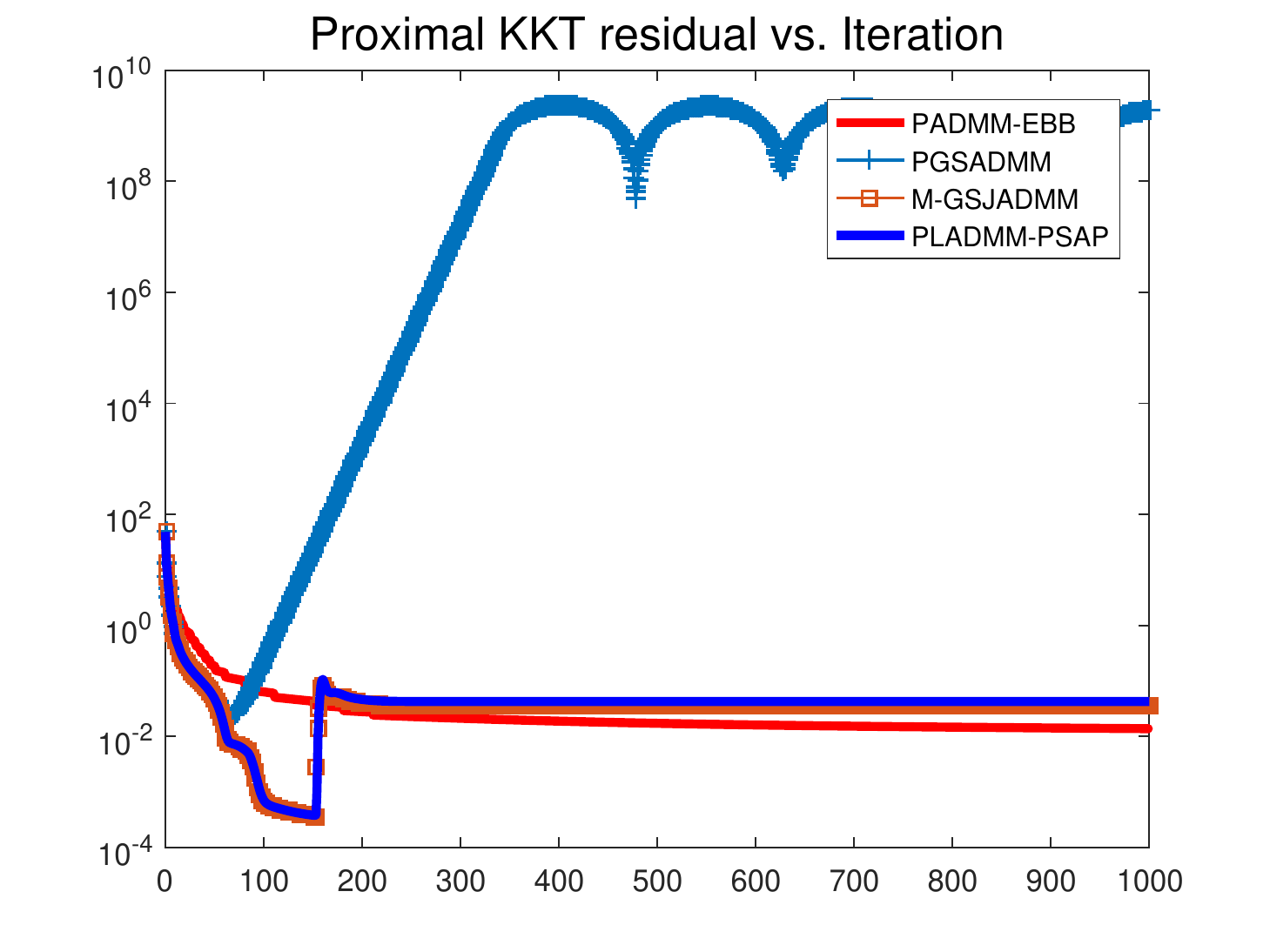}}
\subfigure{\label{fig:pic1}
\includegraphics[width=0.24\linewidth]{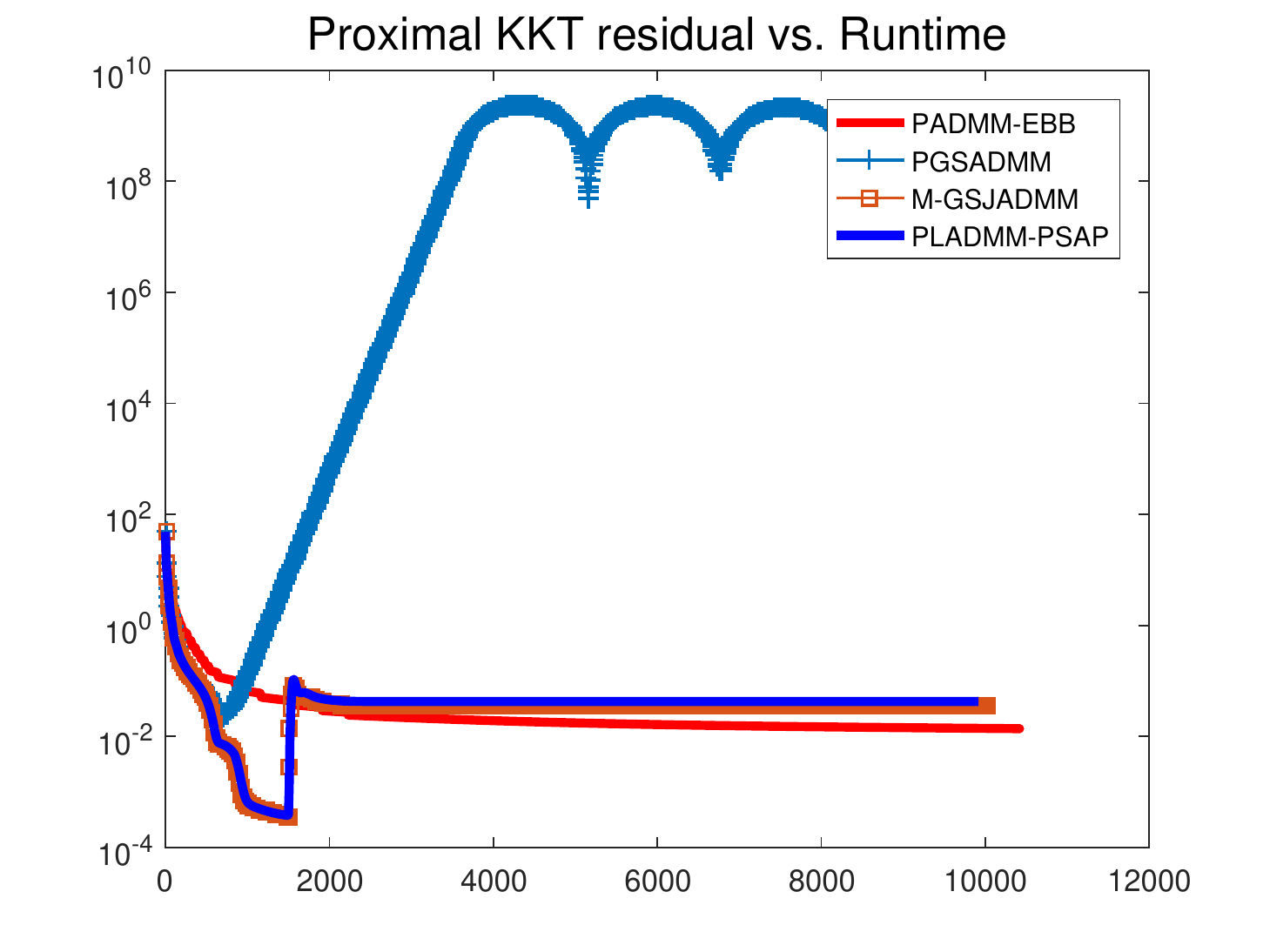}}
\subfigure{\label{fig:pic1}
\includegraphics[width=0.24\linewidth]{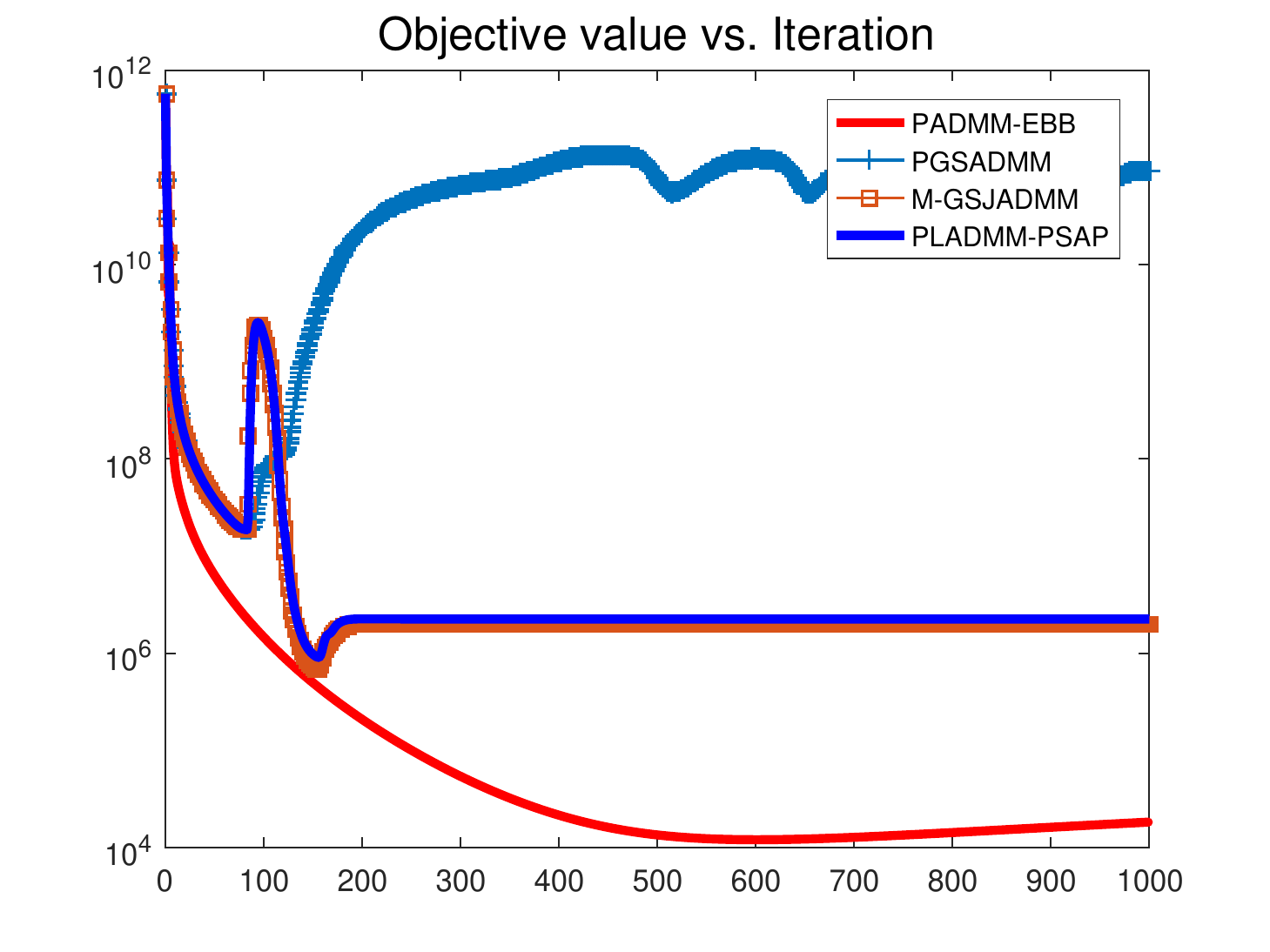}}
\subfigure{\label{fig:pic1}
\includegraphics[width=0.24\linewidth]{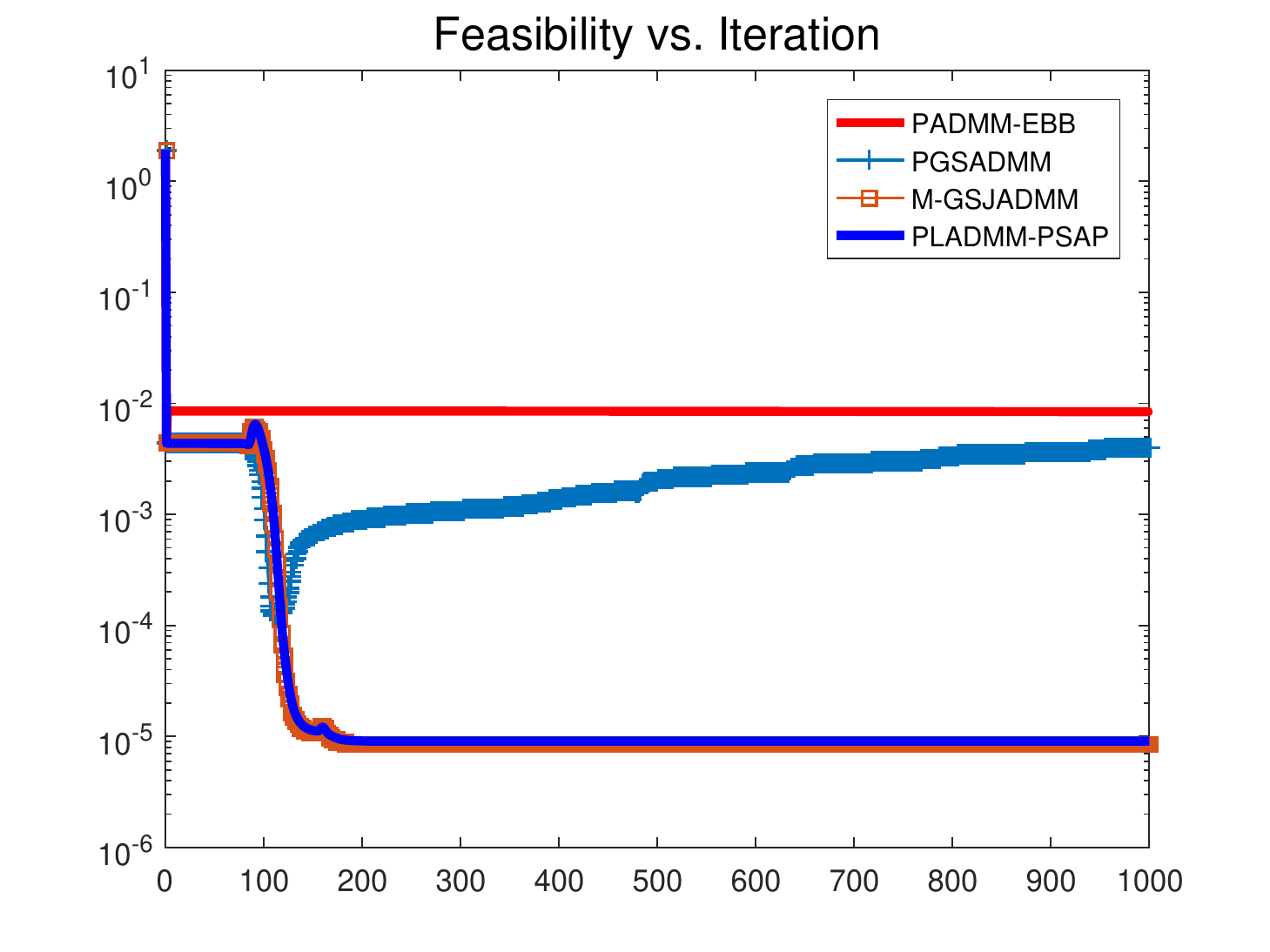}}
 \caption{The above four figures illustrate the proximal KKT residual vs. iteration,
  proximal KKT residual vs. runtime, objective value vs. iteration, and feasibility vs. iteration
  on the real dataset PIE\_pose27 with parameters $(\lambda,\mu, \gamma)=(10^2,10^4,10^4)$, respectively.}
\label{fig:overlap-4}
\end{figure*}

  \begin{figure*}[htpb]
\centering
\subfigure{\label{fig:pic1}
\includegraphics[width=0.24\linewidth]{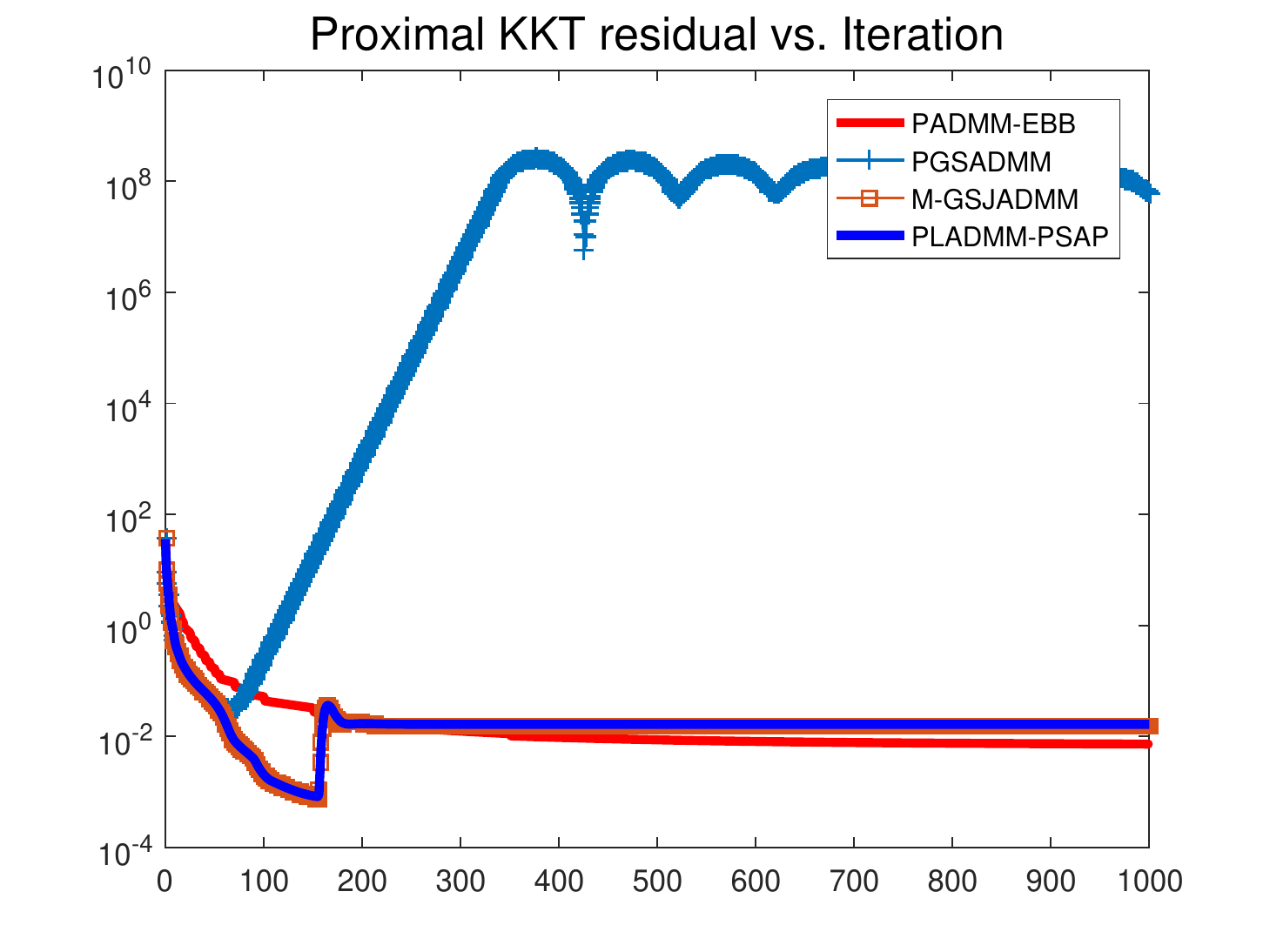}}
\subfigure{\label{fig:pic1}
\includegraphics[width=0.24\linewidth]{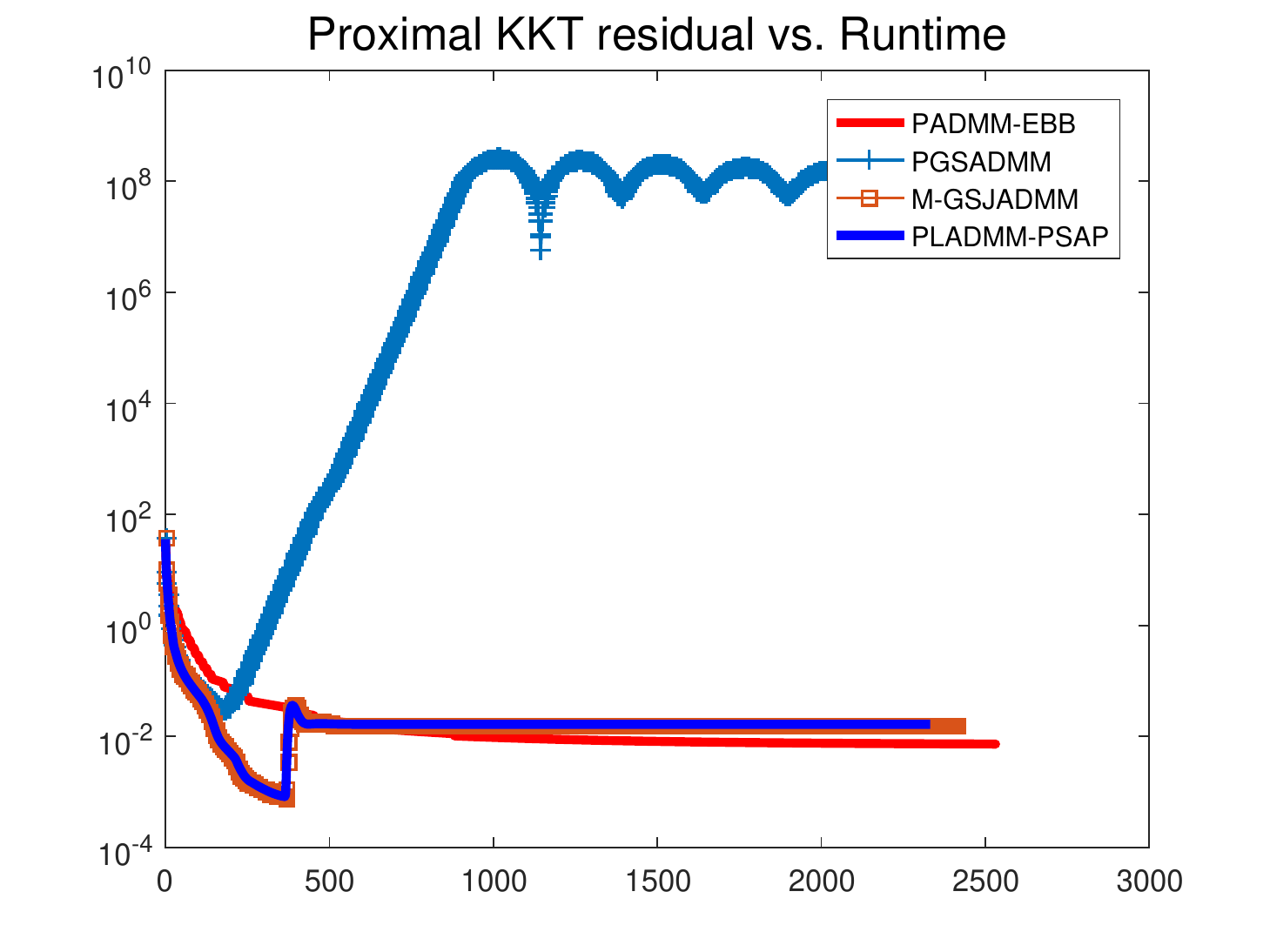}}
\subfigure{\label{fig:pic1}
\includegraphics[width=0.24\linewidth]{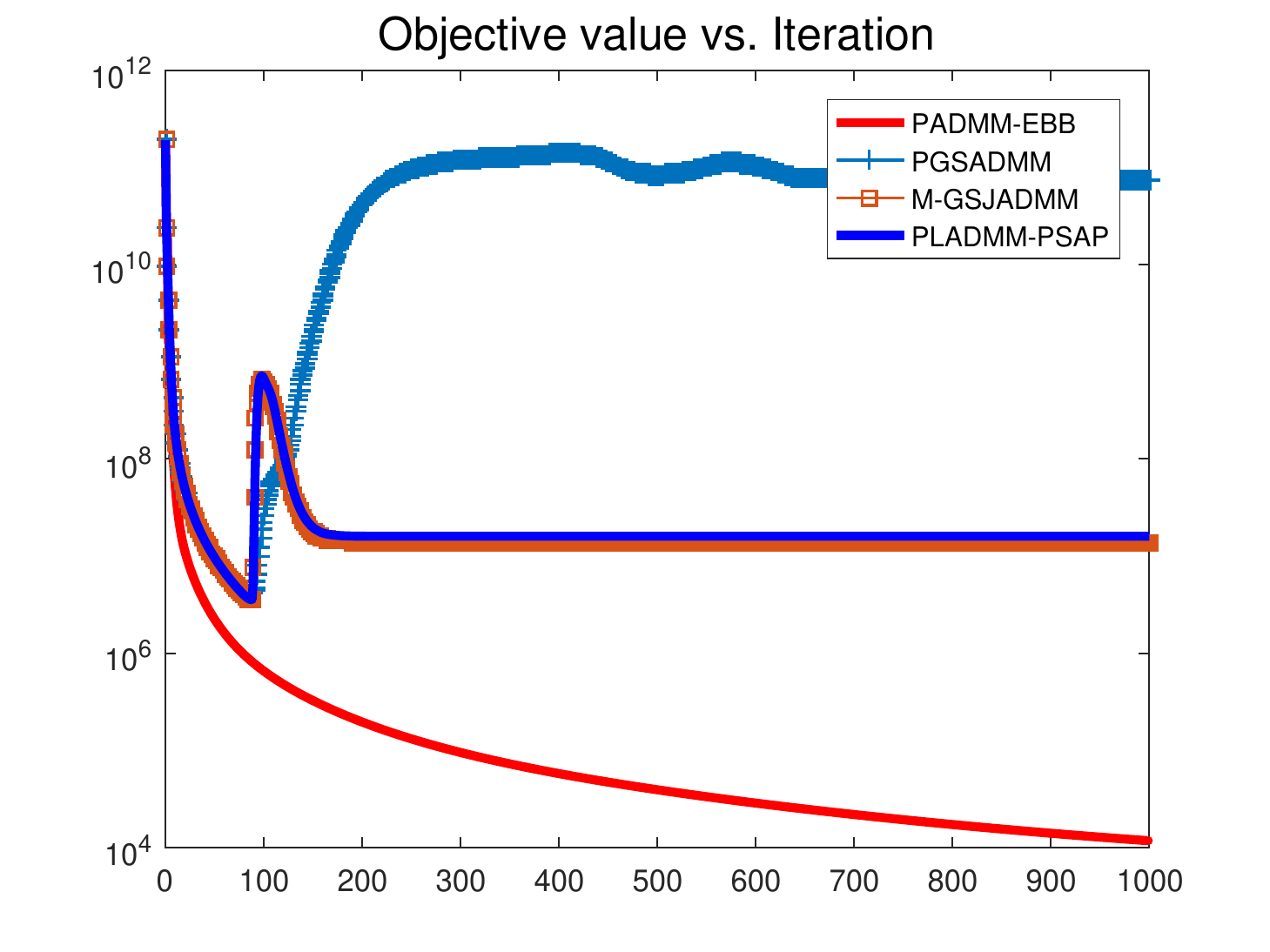}}
\subfigure{\label{fig:pic1}
\includegraphics[width=0.24\linewidth]{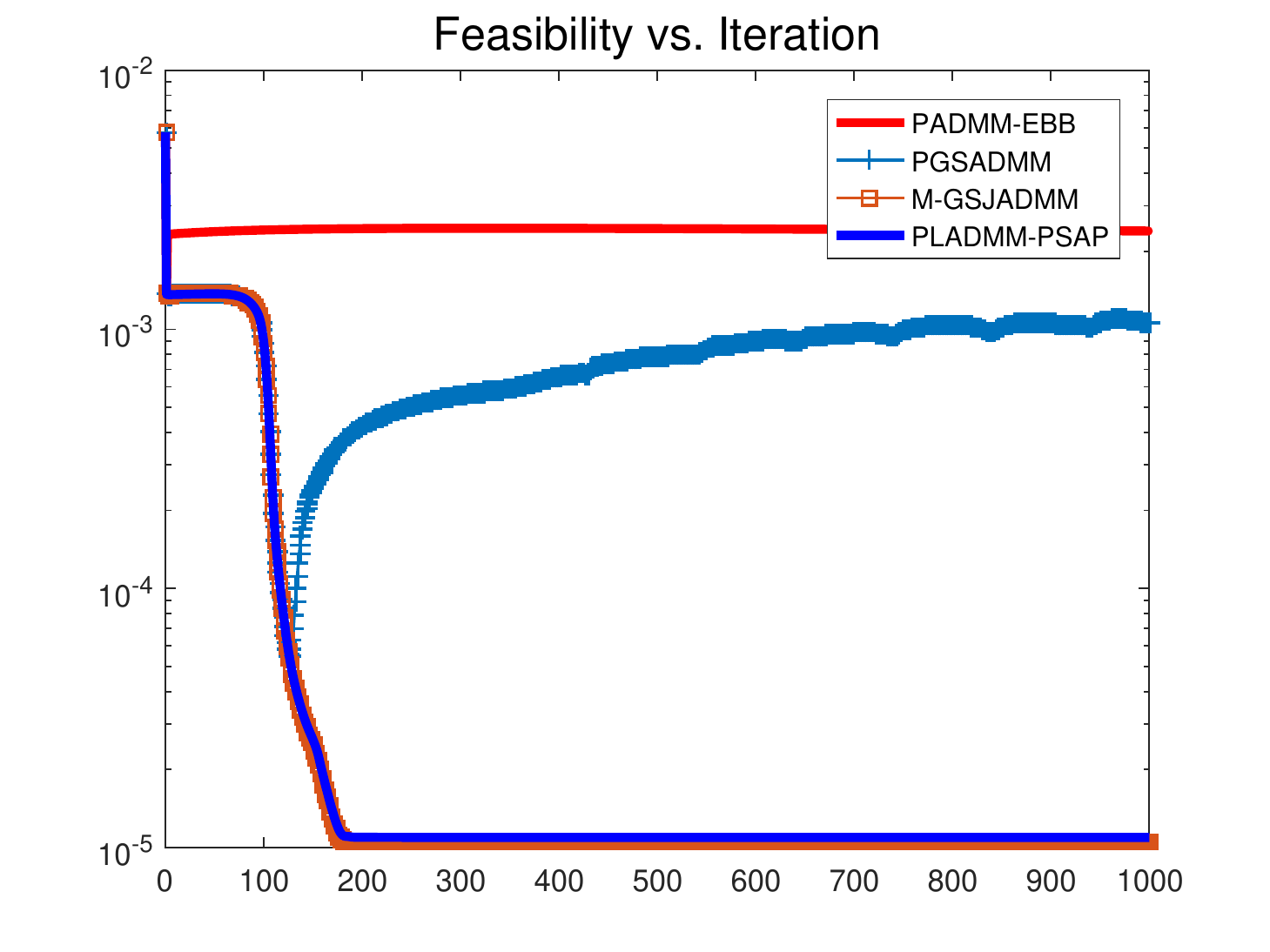}}
 \caption{ The above four figures illustrate the proximal KKT residual vs. iteration,
  proximal KKT residual vs. runtime, objective value vs. iteration,
  and feasibility vs. iteration on the real dataset COIL20  with parameters $(\lambda,\mu, \gamma)=(10^2,10^4,10^4)$, respectively.}
\label{fig:overlap-5}
\end{figure*}
\begin{figure*}[htpb]
\centering
\subfigure{\label{fig:pic1}
\includegraphics[width=0.24\linewidth]{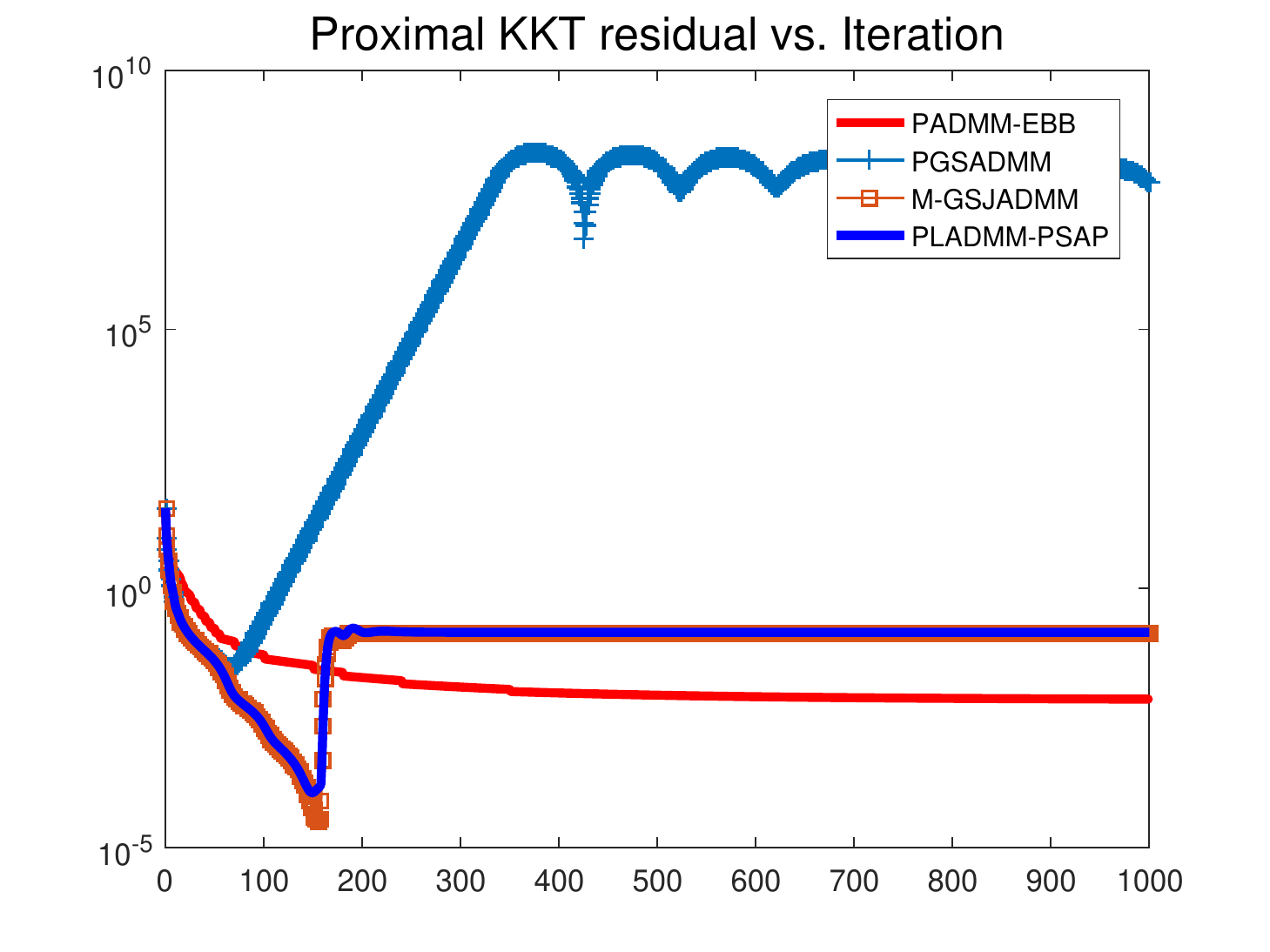}}
\subfigure{\label{fig:pic1}
\includegraphics[width=0.24\linewidth]{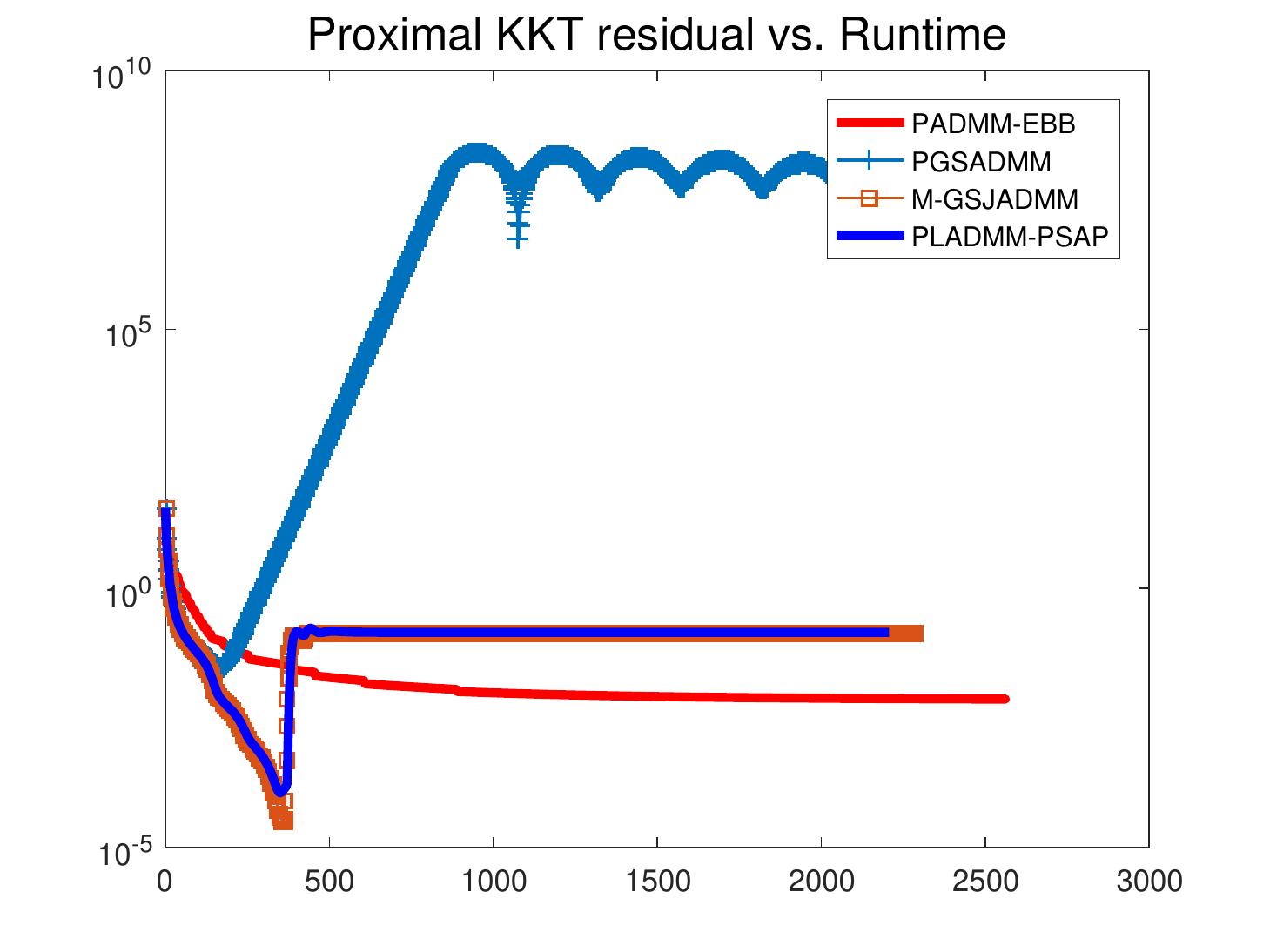}}
\subfigure{\label{fig:pic1}
\includegraphics[width=0.24\linewidth]{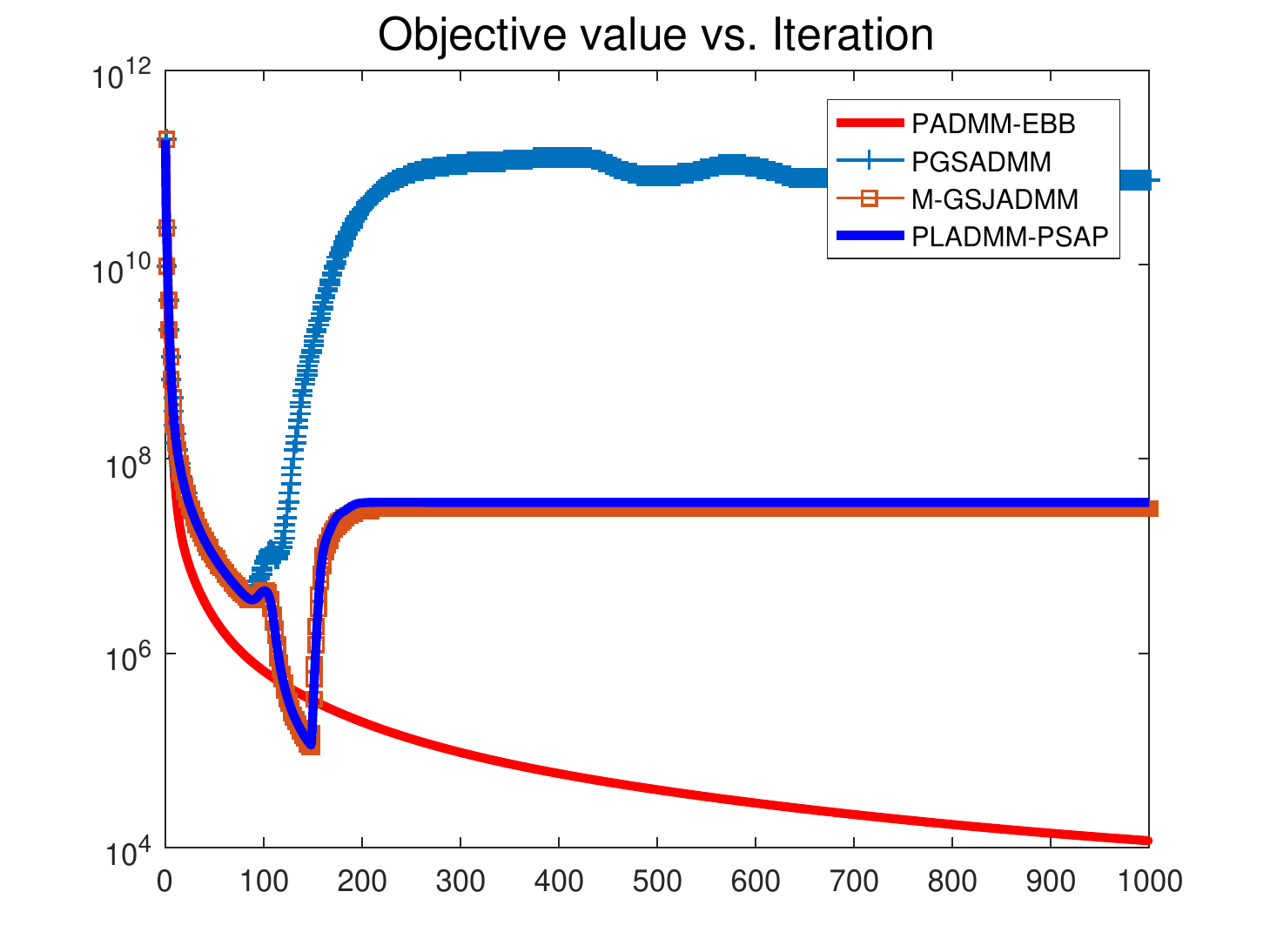}}
\subfigure{\label{fig:pic1}
\includegraphics[width=0.24\linewidth]{coil_feasi_iter_1000_10000-eps-converted-to}}
 \caption{ The above four figures illustrate the proximal KKT residual vs. iteration,
  proximal KKT residual vs. runtime, objective value vs. iteration, and feasibility vs. iteration
  on  the real dataset COIL20  with parameters $(\lambda,\mu, \gamma)=(10^3,10^4,10^4)$, respectively.}
\label{fig:overlap-6}
\end{figure*}

 \begin{figure*}[htpb]
\centering
\subfigure{\label{fig:pic1}
\includegraphics[width=0.24\linewidth]{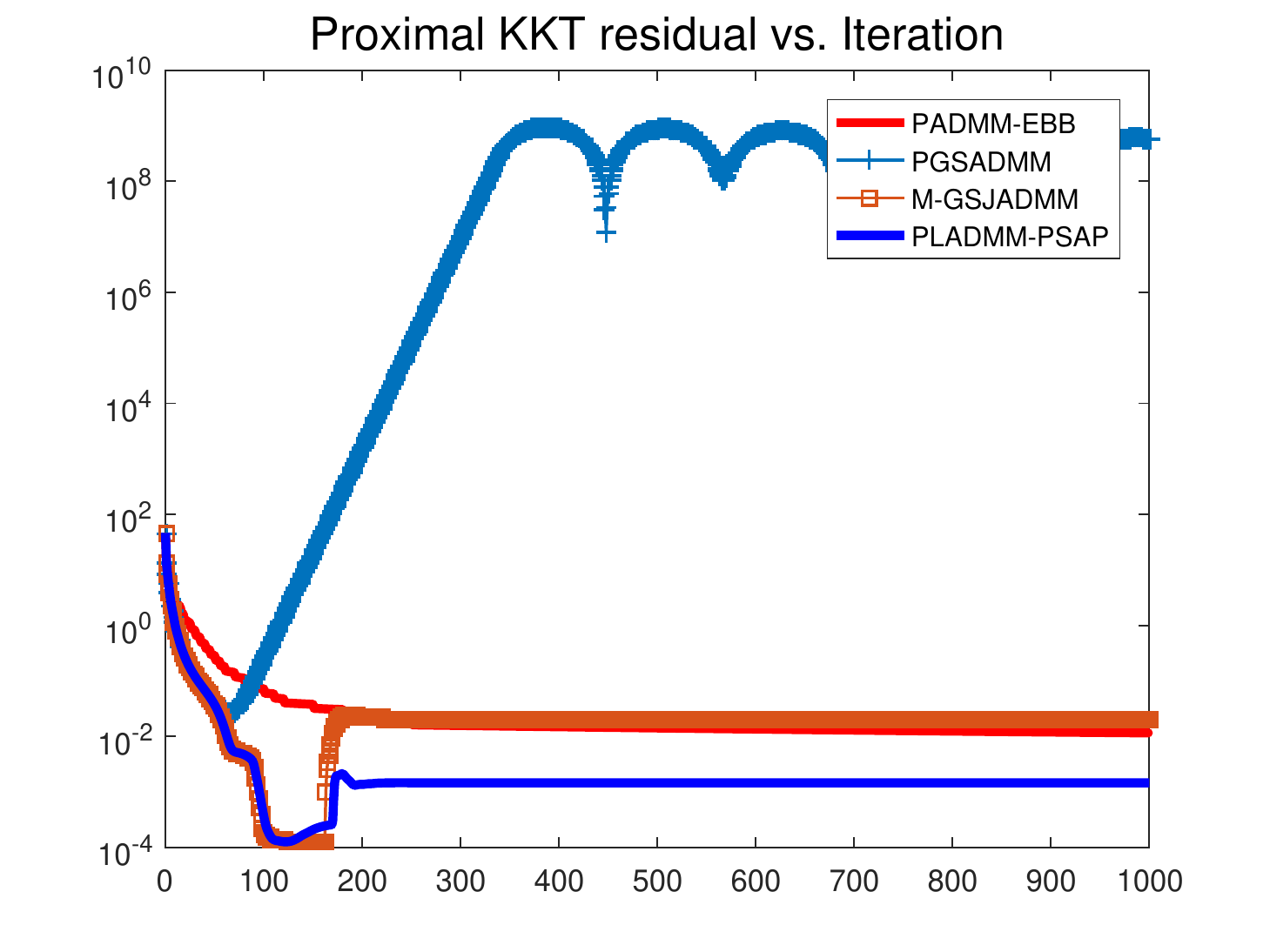}}
\subfigure{\label{fig:pic1}
\includegraphics[width=0.24\linewidth]{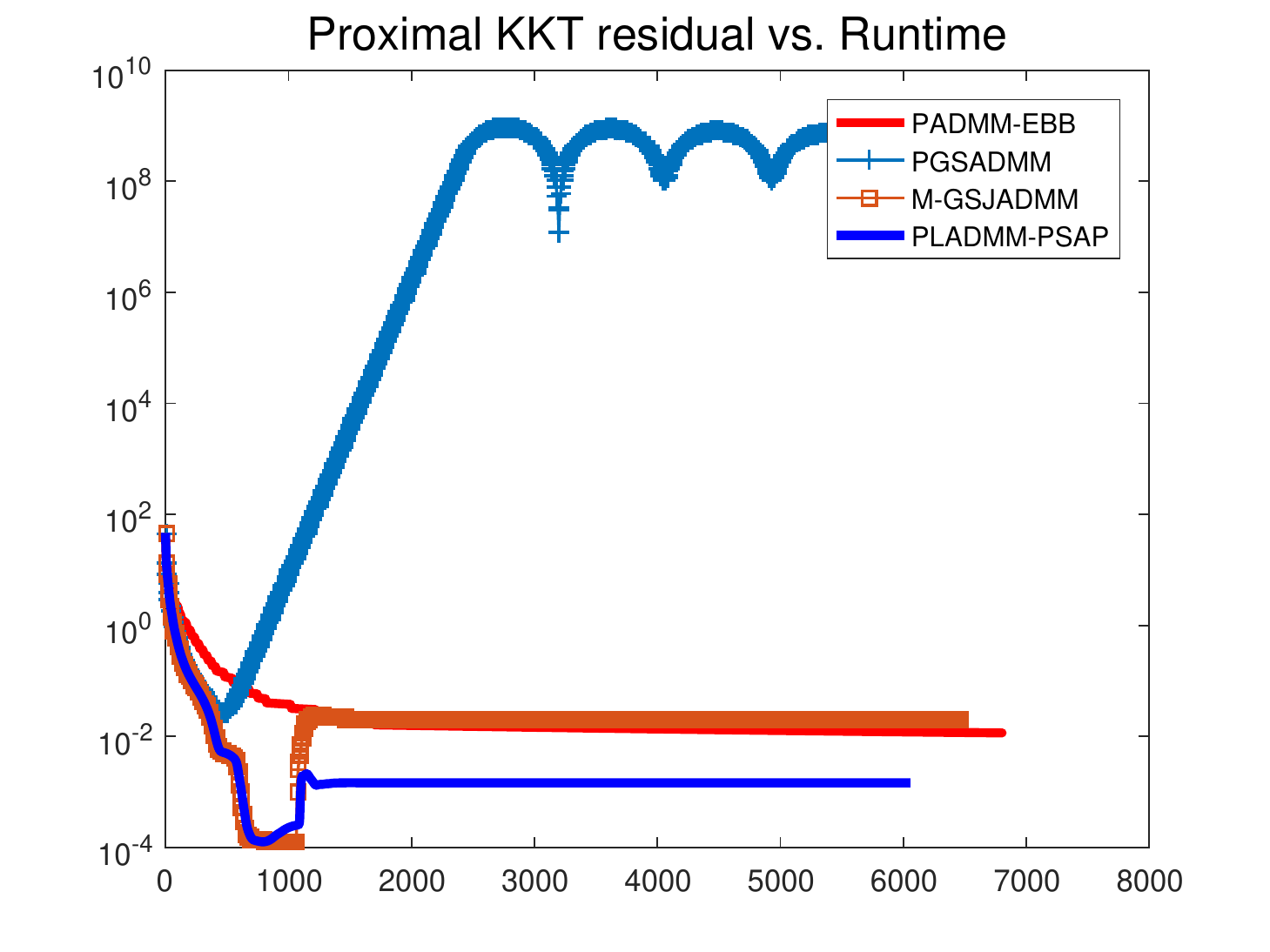}}
\subfigure{\label{fig:pic1}
\includegraphics[width=0.24\linewidth]{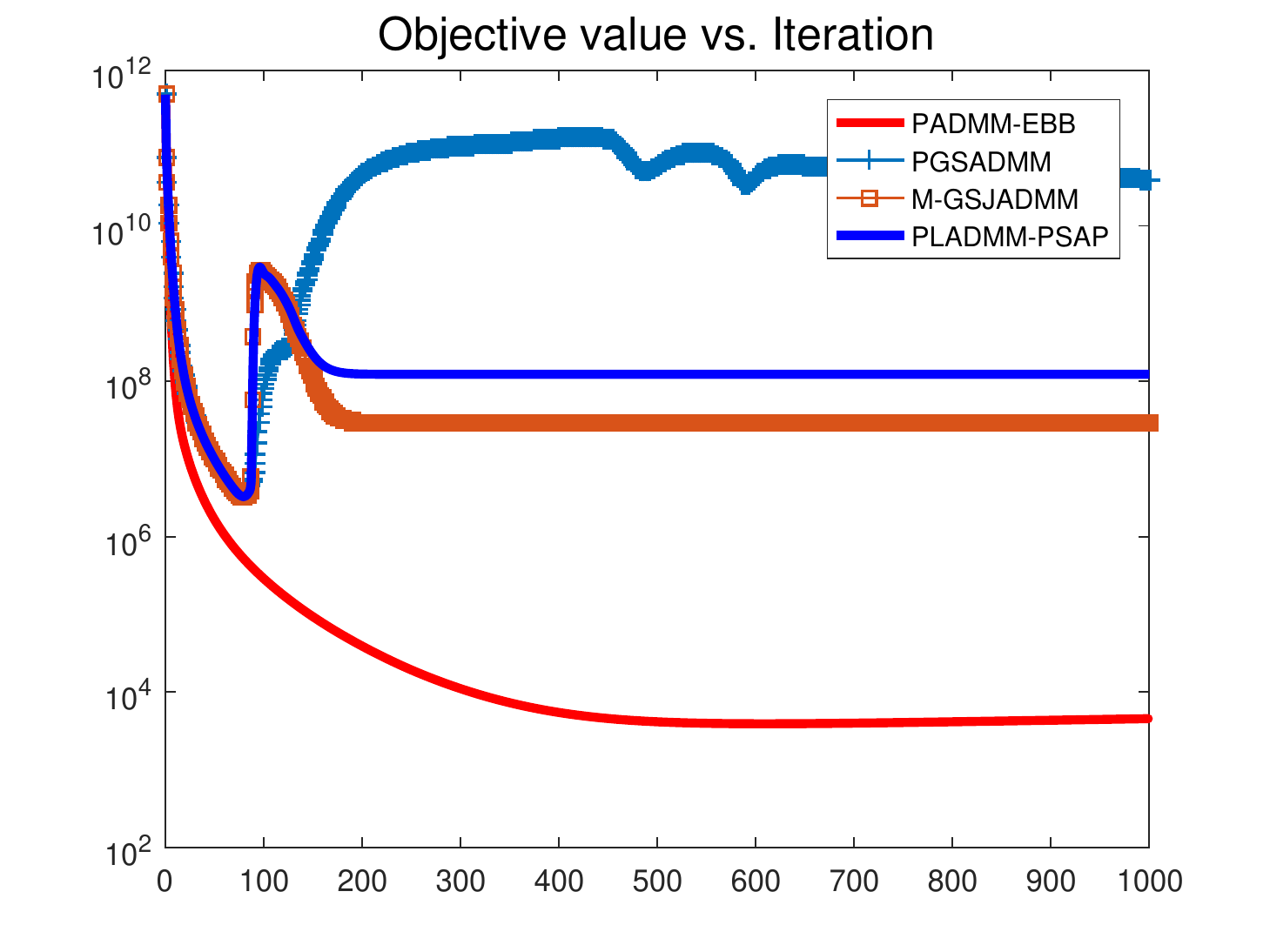}}
\subfigure{\label{fig:pic1}
\includegraphics[width=0.24\linewidth]{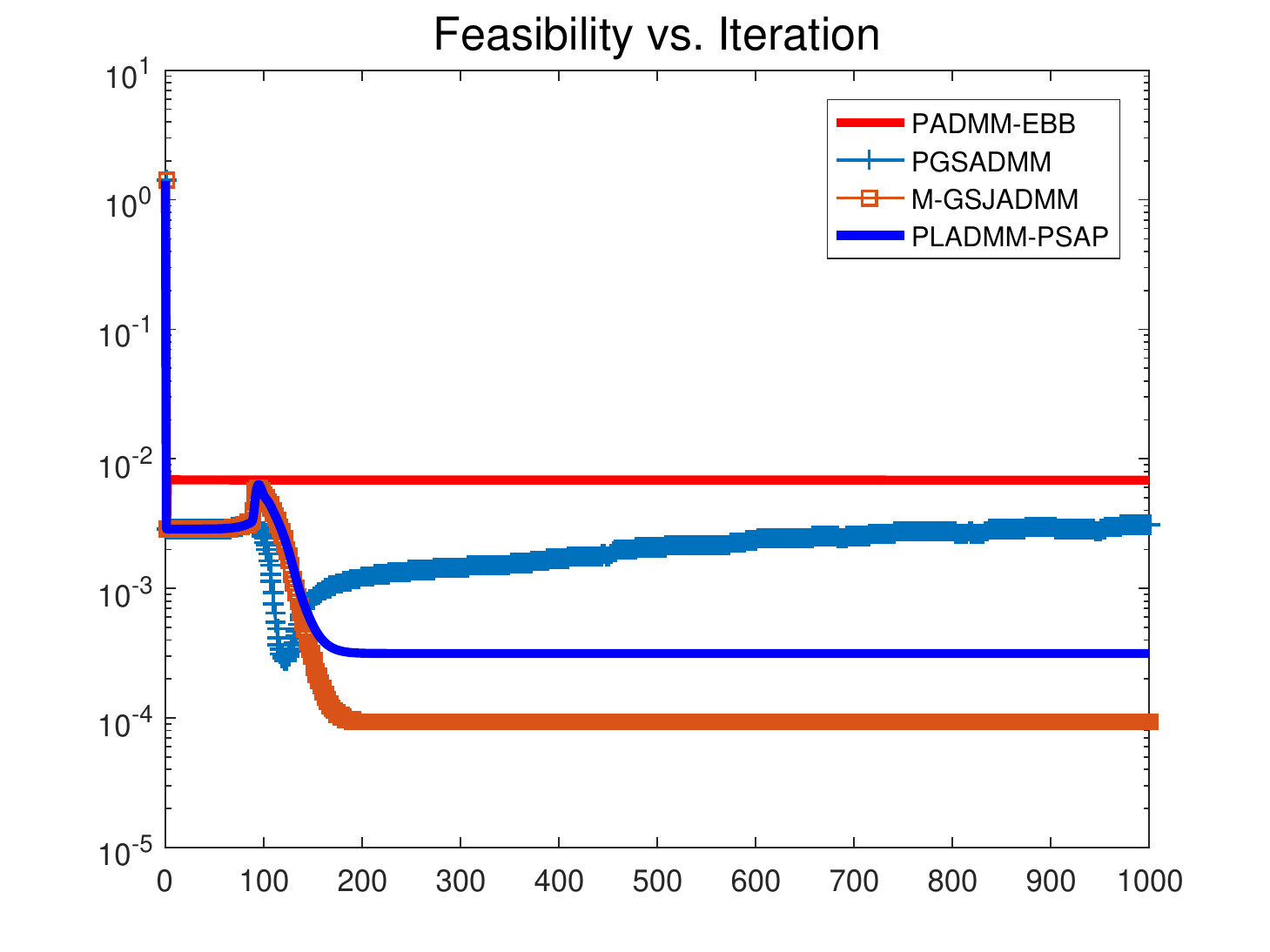}}
 \caption{ The above four figures illustrate the proximal KKT residual vs. iteration,
  proximal KKT residual vs. runtime, objective value vs. iteration, and feasibility vs. iteration
  on the real dataset YaleB\_32x32  with parameters $(\lambda,\mu, \gamma)=(10^2,10^4,10^4)$, respectively.}
\label{fig:overlap-7}
\end{figure*}
\begin{figure*}[htpb]
\centering
\subfigure{\label{fig:pic1}
\includegraphics[width=0.24\linewidth]{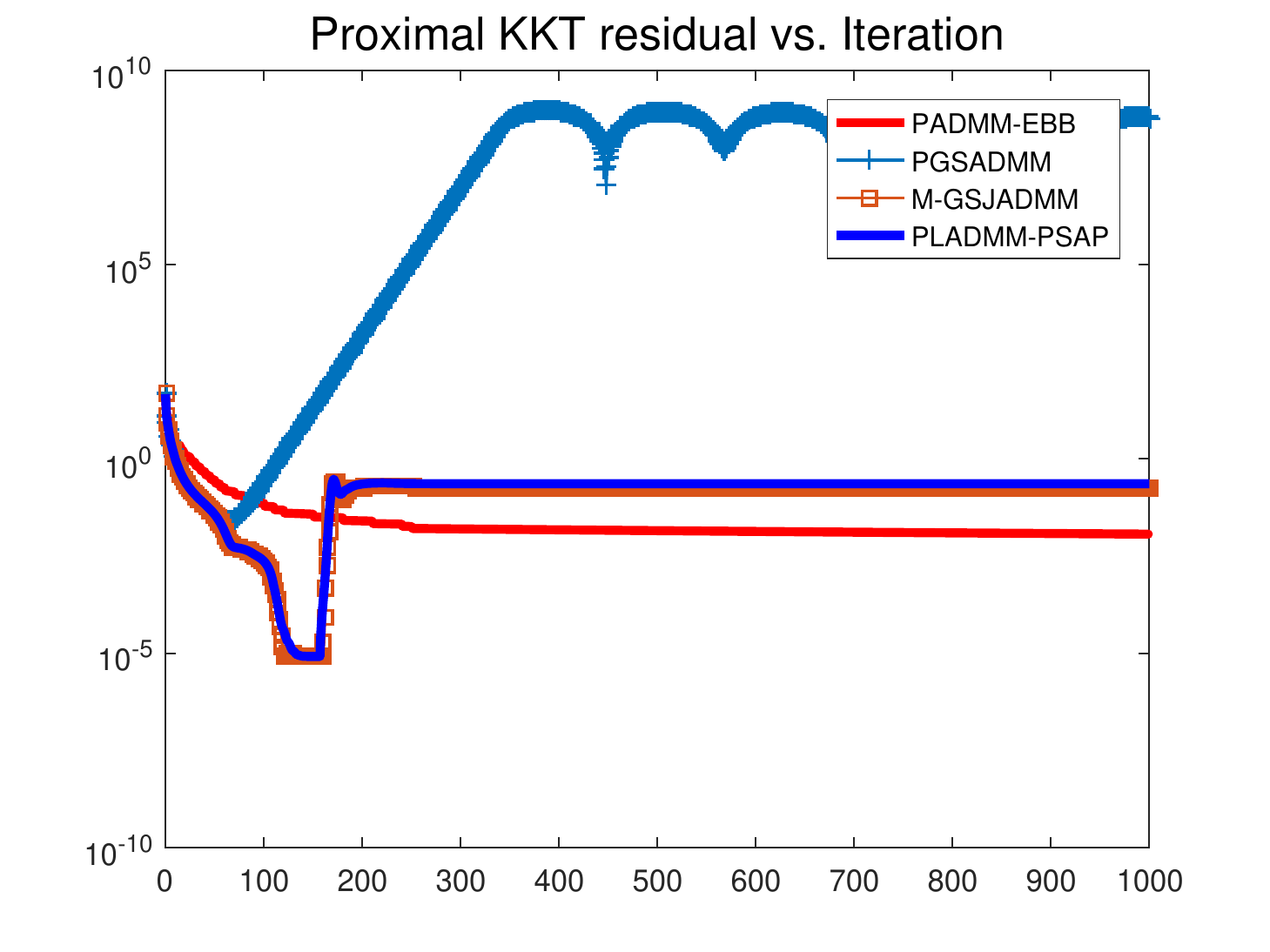}}
\subfigure{\label{fig:pic1}
\includegraphics[width=0.24\linewidth]{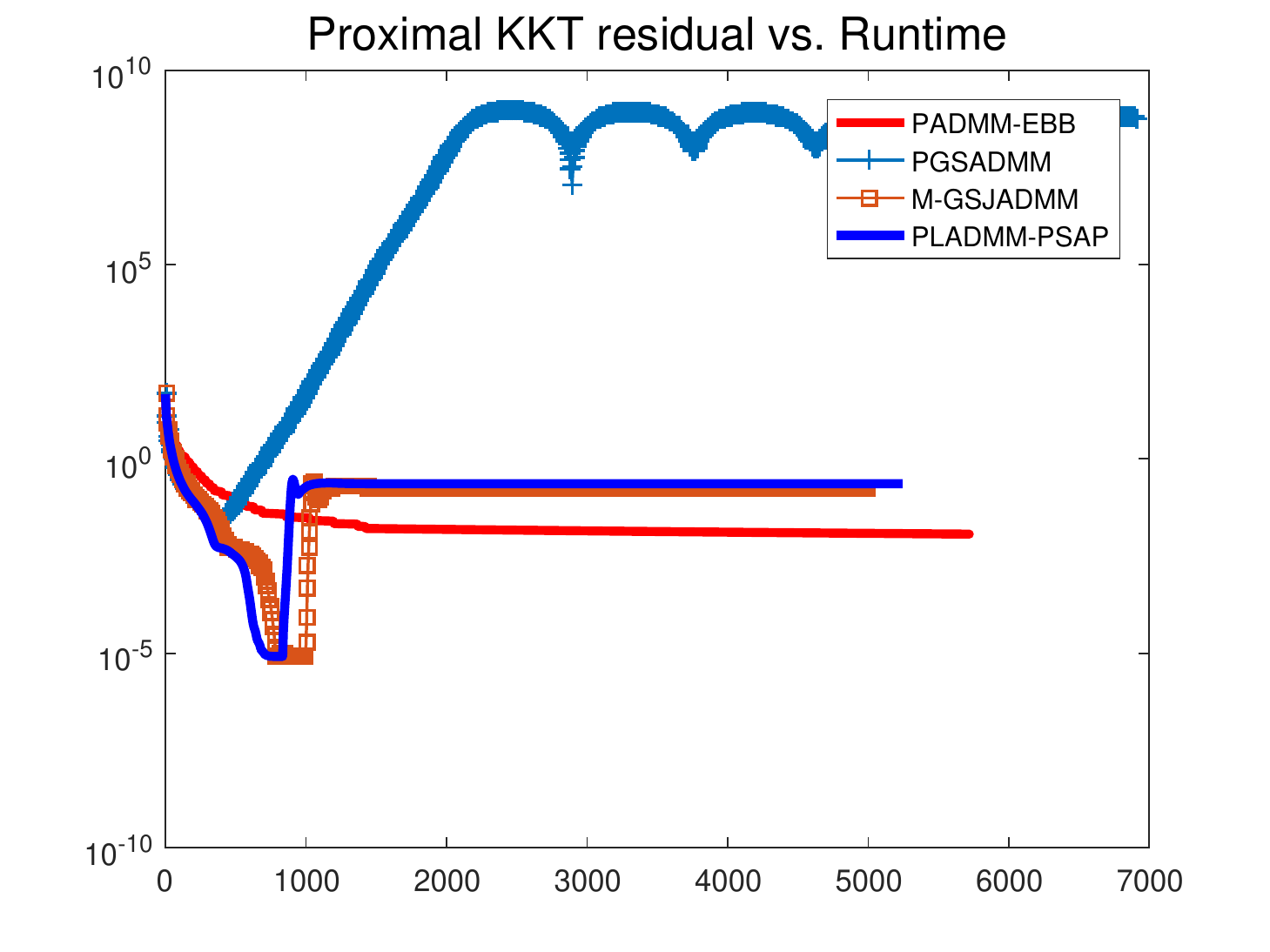}}
\subfigure{\label{fig:pic1}
\includegraphics[width=0.24\linewidth]{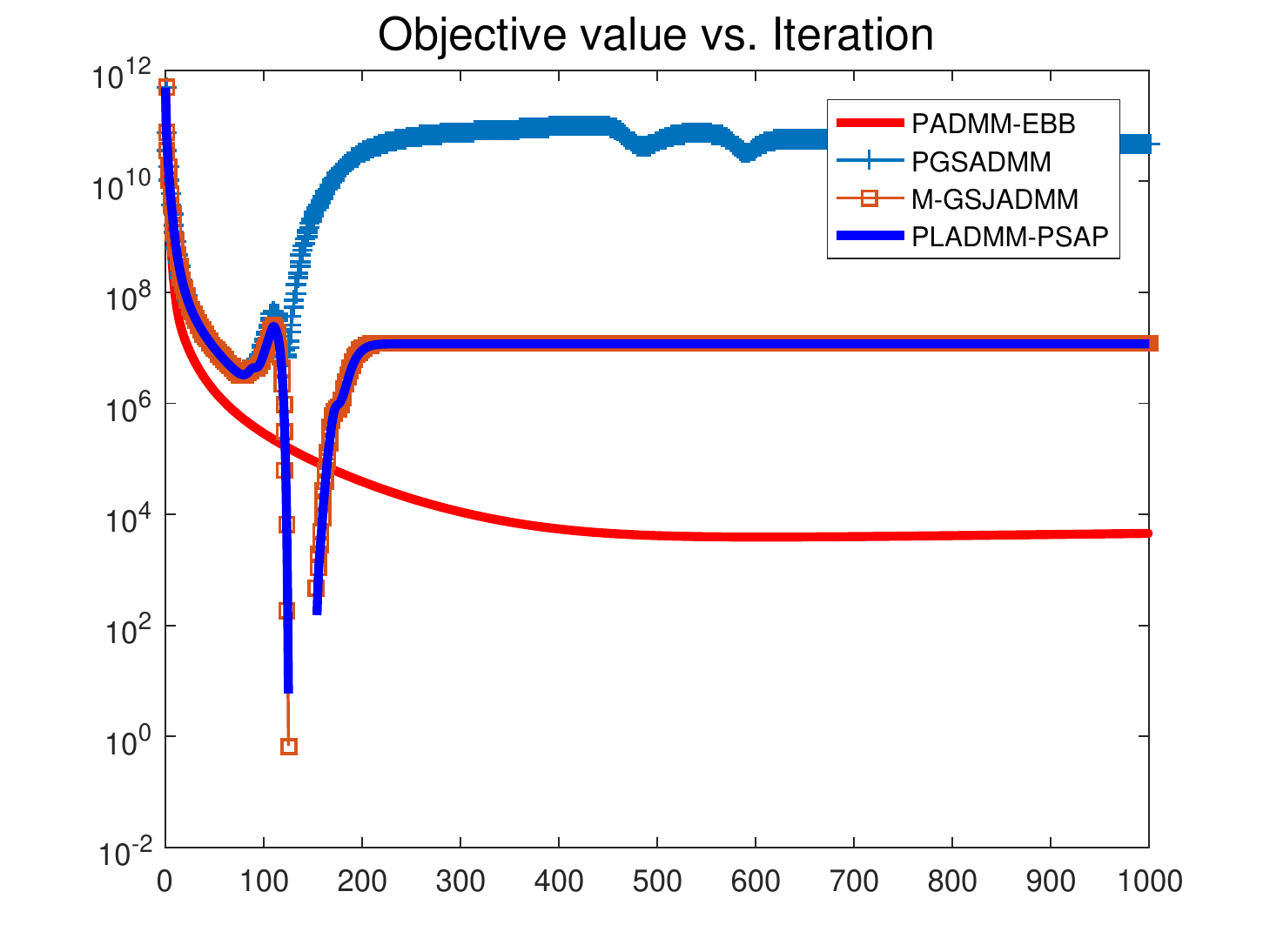}}
\subfigure{\label{fig:pic1}
\includegraphics[width=0.24\linewidth]{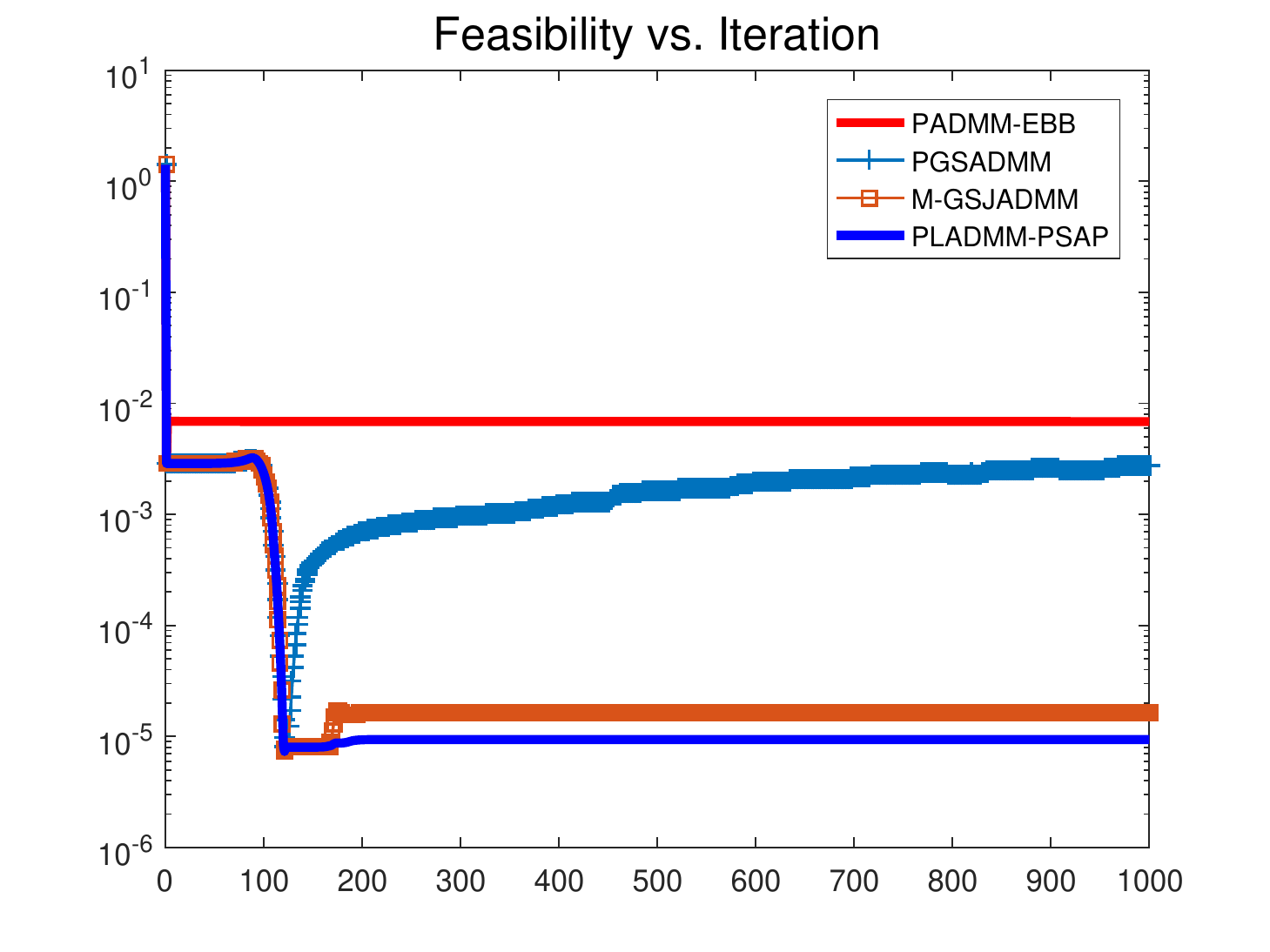}}
 \caption{ The above four figures illustrate the proximal KKT residual vs. iteration,
  proximal KKT residual vs. runtime, objective value vs. iteration, and feasibility vs. iteration on
  the real dataset YaleB\_32x32  with parameters $(\lambda,\mu, \gamma)=(10^3,10^4,10^4)$, respectively.}
\label{fig:overlap-8}
\end{figure*}

%
%
%

\end{document}